\documentclass[11pt,reqno]{amsart} 
\usepackage[a4paper, margin=3cm]{geometry}

\usepackage{graphicx}
\graphicspath{ {./images/} }
\usepackage{lipsum}
\usepackage{subcaption}

\usepackage{xcolor}

\usepackage{amsthm}
\usepackage{bbm}
\usepackage{mathrsfs}
\usepackage{dynkin-diagrams}

\usepackage[utf8]{inputenc}

\usepackage{tikz-cd}

\usepackage{enumitem}

\usepackage{amssymb}
\usepackage{float}
\usepackage{amsbsy}
\usepackage[all]{xy}\usepackage{xypic}
\usepackage{braket}
\usepackage[colorlinks = true,citecolor=black]{hyperref}
\usepackage{amsmath}

\usepackage{wrapfig}

\usepackage[upint]{stix}

\newcommand\myeq{\mathrel{\stackrel{\makebox[0pt]{\mbox{\normalfont\tiny loc}}}{=}}}

\hypersetup{
     colorlinks=true,
     linkcolor=blue,
     filecolor=blue,
     citecolor = blue,      
     urlcolor=cyan,
     }

\makeatletter
\newcommand\opteq[1]{\mathrel{\mathpalette\opt@eq{#1}}}
\newcommand{\opt@eq}[2]{%
  \begingroup
  \sbox\z@{$#1#2$}%
  \sbox\tw@{\resizebox{!}{.5\ht\z@}{$\m@th#1($}}%
  \nonscript\hskip-\wd\tw@
  \mkern1mu
  \raisebox{-.35\ht\z@}[0pt][0pt]{\resizebox{!}{.5\ht\z@}{$\m@th#1($}}%
  \mkern-1mu
  {#2}%
  \mkern-1mu
  \raisebox{-.35\ht\z@}[0pt][0pt]{\resizebox{!}{.5\ht\z@}{$\m@th#1)$}}%
  \mkern1mu
  \nonscript\hskip-\wd\tw@
  \endgroup
}
\makeatother

\newcommand{\geoq}{\opteq{\geq}}

\def\Xint#1{\mathchoice
   {\XXint\displaystyle\textstyle{#1}}%
   {\XXint\textstyle\scriptstyle{#1}}%
   {\XXint\scriptstyle\scriptscriptstyle{#1}}%
   {\XXint\scriptscriptstyle\scriptscriptstyle{#1}}%
   \!\int}
\def\XXint#1#2#3{{\setbox0=\hbox{$#1{#2#3}{\int}$}
     \vcenter{\hbox{$#2#3$}}\kern-.5\wd0}}

\def\dashint{\Xint-}

\newtheorem{theorem}{Theorem}[section]
\newtheorem*{theorem*}{Theorem}
\newtheorem{theorem-non}{Theorem}
\newtheorem{lemma-non}{Lemma}

\theoremstyle{definition} 
\newtheorem{thm}{Theorem}

\theoremstyle{definition}

\newtheorem{conjecture-non}{Conjecture}

\newtheorem{corollary-non}[thm]{Corollary}
\newtheorem{proposition}[theorem]{Proposition}
\newtheorem{lemma}[theorem]{Lemma}
\newtheorem*{lemma*}{Lemma}
\newtheorem{corollary}[theorem]{Corollary}

\newtheorem*{conjecture*}{Conjecture}

\newtheorem{problem}[theorem]{Problem}
\theoremstyle{definition}
\newtheorem{definition}[theorem]{Definition}
\newtheorem{example}[theorem]{Example}

\theoremstyle{remark}
\newtheorem{remark}[theorem]{Remark}

\DeclareMathOperator{\rank}{rank}

\numberwithin{equation}{section}

\begin{document}
\title{Prescribed mean curvature problems on homogeneous vector bundles}

\author{Eder M. Correa}


\address{{IMECC-Unicamp, Departamento de Matem\'{a}tica. Rua S\'{e}rgio Buarque de Holanda 651, Cidade Universit\'{a}ria Zeferino Vaz. 13083-859, Campinas-SP, Brazil}}
\address{E-mail: {\rm ederc@unicamp.br}}

\begin{abstract}
In this paper, we investigate the existence of weak singular Hermite-Einstein structures on homogeneous holomorphic vector bundles over rational homogeneous varieties. Using Cartan's highest weight theory, we establish an explicit algebraic criterion for a homogeneous vector bundle ${\bf{E}}$ to admit a topological splitting ${\bf{E}} \cong {\bf{E}}_{0} \otimes {\bf{L}}_{0}$, where ${\bf{L}}_{0} \in {\rm{Pic}}(X)$ and $c_{1}({\bf{E}}_{0}) = 0$. When this condition is satisfied, the prescribed mean curvature equation completely decouples. By shifting the topological obstruction entirely to the line bundle ${\bf{L}}_{0}$, this splitting reduces the non-abelian prescribed mean curvature problem on ${\bf{E}}$ to Demailly's abelian theory of singular line bundle metrics. As a main application, we obtain a sufficient algebraic condition, expressed in terms of intersection numbers, under which an $L^{2}$-function can be realized as the mean curvature of a singular Hermitian structure on an irreducible homogeneous bundle. Ultimately, by overcoming the bounded curvature restrictions inherent to the classical Bando-Siu framework, this approach provides a robust mechanism to construct singular Hermitian structures accommodating prescribed singularities along analytic subvarieties.
\end{abstract}

\maketitle

\hypersetup{linkcolor=blue}
\tableofcontents

\hypersetup{linkcolor=black}

\section{Introduction} 
Given a Hermitian holomorphic vector bundle $({\bf{E}},{\bf{h}})$ over a compact K\"{a}hler manifold $(X,\omega)$, we have (locally) $F({\bf{h}}) = (F({\bf{h}})_{ij})$, such that 
\begin{equation}
F({\bf{h}})_{ij} = \sum_{\alpha,\beta} R^{i}_{j\alpha \overline{\beta}}{\rm{d}}z_{\alpha} \wedge {\rm{d}}\overline{z_{\beta}},
\end{equation}
where $F({\bf{h}})$ is the curvature of the underlying Chern connection. From this, we can define the $\omega$-trace of $F({\bf{h}})$ as being the endomorphism $\Lambda_{\omega}(F({\bf{h}})) = (\Lambda_{\omega}(F({\bf{h}})_{ij}))$, such that 
\begin{equation}
\Lambda_{\omega}(F({\bf{h}})_{ij}) = \sum_{\alpha,\beta}g^{\alpha \overline{\beta}} R^{i}_{j\alpha \overline{\beta}},
\end{equation}
here we denote by $g^{-1} = (g^{\alpha \overline{\beta}})$ the inverse matrix of the associated K\"{a}hler metric $g = \omega({\rm{1}} \otimes J)$. From above, we can define the following notions of curvature (e.g. \cite{Kobayashi+1987}):
\begin{enumerate}
\item we define the mean curvature of $({\bf{E}},{\bf{h}})$ as being ${\rm{K}}({\bf{E}},{\bf{h}}) := \sqrt{-1}\Lambda_{\omega}(F({\bf{h}})) $,
\item the scalar curvature of $({\bf{E}},{\bf{h}})$ is defined by $\sigma({\bf{E}},{\bf{h}}) := \sqrt{-1}{\rm{tr}}\big (\Lambda_{\omega}(F({\bf{h}})) \big )$.
\end{enumerate}
In the above setting, if
\begin{equation}
\label{WHYM}
{\rm{K}}({\bf{E}},{\bf{h}}) = \varphi 1_{{\bf{E}}},
\end{equation}
for some smooth function $\varphi \in C^{\infty}(X)$, we say that $({\bf{E}},{\bf{h}})$ is weak Hermite-Einstein. As can be seen, if $({\bf{E}},{\bf{h}})$ is weak Hermite-Einstein, then $\sigma({\bf{E}},{\bf{h}}) = r\varphi$, where $r = \rank({\bf{E}})$. Also, if $\varphi \equiv c$, such that $c \in \mathbbm{R}$, then $({\bf{E}},{\bf{h}})$ satisfying Eq. (\ref{WHYM}) is said to be Hermite-Einstein. In this last case, the underlying Chern connection is called Hermitian-Yang-Mills.

Given an indecomposable holomorphic vector bundle ${\bf{E}}$ over a compact K\"{a}hler manifold $(X,\omega)$, it is known that {\bf{E}} admits an Hermitian-Yang-Mills (HYM) connection if, and only if, {\bf{E}} is stable in the sense of Mumford-Takemoto. This result, also known as Kobayashi-Hitchin correspondence, was independently conjectured in the 1980s by Kobayashi \cite{kobayashi1982curvature} and Hitchin \cite{kotake1979non}. The Kobayashi-Hitchin correspondence was proved for complex curves by Narasimhan and Seshadri \cite{narasimhan1965stable}, and then in full generality by Donaldson \cite{donaldson1985anti} and Uhlenbeck-Yau \cite{uhlenbeck1986existence}. 

In \cite{bando1994stable}, Bando and Siu established the existence of admissible singular Hermite-Einstein metrics on stable reflexive sheaves. Their machinery generally requires the mean curvature to remain bounded, tackling the non-abelian PDE by resolving topological singularities of codimension $\geq 2$. For more results on singular Hermitian structures, we suggest \cite{damailly1992singular}, \cite{cataldo1998singular}, \cite{raufi2015singular}.

In Riemannian geometry, a classical problem that can be traced back to the work of Kazdan-Warner \cite{kazdan1975direct}, \cite{kazdan1975existence}, \cite{kazdan1975scalar}, consists of finding admissible functions as the scalar curvature of Riemannian metrics. More precisely, the problem is as follows: given a closed smooth manifold $M$ and a smooth function $f \in C^{\infty}(M)$, construct a Riemannian metric on $M$ whose scalar curvature equals $f$. 

Based on these ideas, in this paper we study the following Kazdan-Warner-type problem in the setting of homogeneous Hermitian vector bundles over rational varieties.
\begin{problem}
\label{P1}
For some prescribed $f \in L^{2}(X,\mu_{\omega})$, find conditions under which there exists a (possibly singular) Hermitian structure ${\bf{h}}$ on a holomorphic vector bundle ${\bf{E}} \to (X,\omega)$, such that 
\begin{equation}
\label{KWHVB}
{\rm{K}}({\bf{E}},{\bf{h}}) = f 1_{\bf{E}},
\end{equation}
in the sense of distributions.
\end{problem} 
Here, by following \cite{10.1215/00127094-2008-054}, we assume that a singular Hermitian metric ${\bf{h}}$ on ${\bf{E}}$ is a measurable map from the base space $X$ to the space of non-negative Hermitian forms on the fibers.

The primary obstruction to solving Problem \ref{P1} lies in the analytic pathology of non-abelian gauge groups. For higher-rank bundles, the curvature $F({\bf{h}}) = \bar{\partial}({\bf{h}}^{-1}\partial {\bf{h}})$ inherently involves the multiplication of matrix components. If we allow the metric to develop singularities to accommodate a rough target $f \in L^{2}(X, \mu_{\omega})$, we are forced to multiply distributions, an operation that is generally undefined. This highly coupled non-linearity renders the general prescribed mean curvature equation intractable via standard singular current theory, which is strictly abelian in nature.

Although it is not immediate how to make sense of the curvature of a singular Hermitian metric ${\bf{h}}$, it is nevertheless still possible to define what it means in some cases, see for instance \cite{damailly1992singular}, \cite[\S 2.2]{cataldo1998singular}, and \cite[\S 3]{raufi2015singular}.

Our strategy to tackle Problem \ref{P1} in the homogeneous setting is based on the following idea. Suppose that a holomorphic vector bundle ${\bf{E}} \to (X,\omega)$ has a splitting of the form ${\bf{E}} \cong {\bf{E}}_{0} \otimes {\bf{L}}_{0}$, such that ${\bf{L}}_{0} \in {\rm{Pic}}(X)$. In this situation, by considering ${\bf{h}} = {\bf{h}}_{{\bf{E}}_{0}} \otimes {\bf{h}}_{{\bf{L}}_{0}}$, since 
\begin{equation}
{\rm{K}}({\bf{E}},{\bf{h}}) = {\rm{K}}({\bf{E}}_{0},{\bf{h}}_{{\bf{E}}_{0}}) \otimes 1_{{\bf{L}}_{0}} + 1_{{\bf{E}}_{0}} \otimes {\rm{K}}({\bf{L}}_{0},{\bf{h}}_{{\bf{L}}_{0}}),
\end{equation}
the mean curvature equation completely decouples. Hence, we can solve Problem \ref{P1} for some $f \in L^{2}(X,\mu_{\omega})$ if we find a solution for the following system of equations
\begin{equation}
\label{systemdecupled}
\begin{cases}
{\rm{K}}({\bf{E}}_{0},{\bf{h}}_{{\bf{E}}_{0}}) = 0,\\
{\rm{K}}({\bf{L}}_{0},{\bf{h}}_{{\bf{L}}_{0}}) = f1_{{\bf{L}}_{0}}.
\end{cases}
\end{equation}
The first equation in the system above is the classical Hermitian-Yang-Mills equation. It imposes the condition that $c_{1}({\bf{E}}_{0}) = 0$. The second equation in the system above is the weak Hermite-Einstein equation. For line bundles, it reduces to a Poisson equation 
\begin{equation}
\Delta_{\omega}u + \rho = 0,
\end{equation}
on the compact manifold $(X,\omega)$, which can be solved by standard methods. From this, we study the following associated problem:
\begin{problem}
\label{P2}
Given a holomorphic vector bundle ${\bf{E}} \to (X,\omega)$, under what conditions do we have a splitting ${\bf{E}} = {\bf{E}}_{0} \otimes {\bf{L}}_{0}$, such that ${\bf{L}}_{0} \in {\rm{Pic}}(X)$ and ${\rm{K}}({\bf{E}}_{0},{\bf{h}}_{{\bf{E}}_{0}}) = 0$?
\end{problem}

The main contribution of this work is to provide a complete answer to Problem \ref{P2} in the homogeneous setting by establishing an explicit algebraic criterion based on Cartan's highest weight theory. By achieving this topological splitting, we completely decouple the non-abelian mean curvature equation. As a consequence, we systematically bypass the analytic pathologies inherent to Problem \ref{P1}, reducing it to Demailly's abelian theory of singular line bundle metrics. This framework not only makes the prescribed mean curvature problem completely solvable via standard tools from spectral geometry, but also provides a robust mechanism to construct singular Hermitian structures with prescribed singularities along analytic subvarieties.

\subsection{Main Results}
In order to state our main results, let us recall some basic generalities on rational homogeneous varieties. 

A rational homogeneous variety can be described as a quotient $X_{P} = G^{\mathbbm{C}}/P$, where $G^{\mathbbm{C}}$ is a semisimple complex algebraic group with Lie algebra $\mathfrak{g}^{\mathbbm{C}} = {\rm{Lie}}(G^{\mathbbm{C}})$, and $P$ is a parabolic Lie subgroup (e.g. \cite{BorelRemmert}). Considering $G^{\mathbbm{C}}$ as a complex analytic space, without loss of generality, we assume that $G^{\mathbbm{C}}$ is a connected simply connected complex simple Lie group. Fixed a compact real form $G \subset G^{\mathbbm{C}}$, we can also consider the underlying reductive homogeneous $G$-space provided by the quotient $X_{P} = G/K$, where $K = G \cap P$. 

By choosing a Cartan subalgebra $\mathfrak{h} \subset \mathfrak{g}^{\mathbbm{C}}$ and a simple root system $\Delta \subset \mathfrak{h}^{\ast}$, up to conjugation, we have the underlying parabolic Lie subgroup $P \subset G^{\mathbbm{C}}$ completely determined by some subset of simple roots $I \subset \Delta$, e.g. \cite[\S 3.1]{Akhiezer}. Considering the associated fundamental weights $\varpi_{\alpha} \in \mathfrak{h}^{\ast}$, $\alpha \in \Delta$, it follows that 
\begin{equation}
{\rm{Pic}}(X_{P}) \cong H^{1,1}(X_{P},\mathbbm{Z}) \cong \Lambda_{P}:= \bigoplus_{\alpha \in \Delta \backslash I}\mathbbm{Z}\varpi_{\alpha},
\end{equation}
see for instance \cite{correa2023deformed}. From the aforementioned isomorphisms, we have a map\footnote{We also can define $[\omega] \in H^{1,1}(X_{P},\mathbbm{R}) \mapsto \lambda([\omega])  \in \Lambda_{P} \otimes \mathbbm{R}$.}
\begin{equation}
\label{map11character}
[\omega] \in H^{1,1}(X_{P},\mathbbm{Z}) \mapsto \lambda([\omega]) \in \Lambda_{P}.
\end{equation}
In particular, we can set $\lambda({\bf{E}}) := \lambda(c_{1}({\bf{E}})) \in \Lambda_{P}$, for every holomorphic vector bundle ${\bf{E}} \to X_{P}$. 

Given an irreducible $P$-module $\theta \colon P \to {\rm{GL}}(W)$, we can form a holomorphic homogeneous vector bundle by setting 
\begin{equation}
{\bf{E}} = G^{\mathbbm{C}} \times_{P}W.
\end{equation}
Considering the Levi decomposition
\begin{equation}
P = L_{P} \ltimes R_{u}(P)^{+},
\end{equation}
it follows that the irreducible representations of $P$ are completely classified by the irreducible representations of its Levi component $L_{P}$. Since $L_{P}$ is a reductive Lie group, considering its associated root datum
\begin{equation}
(\mathbbm{X}(T^{\mathbbm{C}}),\langle I \rangle, \alpha \mapsto \alpha^{\vee}),
\end{equation}
it follows that the set of isomorphism classes of irreducible representations of $L_{P}$ is parametrized by the integral weights 
\begin{equation}
{\rm{Rep}}(P)_{\text{irrd.}} = \Lambda_{\mathfrak{s}_{P}}^{+} \oplus \Lambda_{P},
\end{equation}
where 
\begin{equation}
\Lambda_{\mathfrak{s}_{P}}^{+} = \bigoplus_{\alpha \in I}\mathbbm{Z}_{\geq 0}\varpi_{\alpha}
\end{equation}
is the monoid of integral dominant weights of the semisimple component $S_{P} = [L_{P},L_{P}]$ of $L_{P}$. 

By considering the above data, we prove the following auxiliary theorem of independent interest which provides a complete answer for Problem \ref{P2} in the homogeneous setting.

\begin{thm}
\label{theoremA}
Let ${\bf{E}}\to  X_{P}$ be a homogeneous vector bundle defined by some finite-dimensional irreducible representation $\theta \colon P \to {\rm{GL}}(W(\lambda))$. Considering $\lambda = \lambda_{s} + \lambda_{c}$, such that $\lambda_{s} \in \Lambda_{\mathfrak{s}_{P}}^{+}$ and $\lambda_{c} \in \Lambda_{P}$, we have the following equivalent conditions:
\begin{enumerate}
\item[(A)] ${\bf{E}} \cong {\bf{E}}_{0} \otimes {\bf{L}}_{0}$, such that ${\bf{L}}_{0} \in {\rm{Pic}}(X_{P})$ and $c_{1}({\bf{E}}_{0}) = 0$,
\item[(B)] $\displaystyle \frac{1}{{\rm{rank}}({\bf{E}})} \int_{C}c_{1}({\bf{E}}) \in \mathbbm{Z}, \forall [C] \in {\rm{NE}}(X_{P})_{\mathbbm{Z}}$,
\item[(C)] For every $\beta \in \Delta \backslash I$, we have 
\begin{equation}
\sum_{\alpha \in I}\frac{\det(C_{I}(\lambda_{s},\alpha))}{\det(C_{I})} \langle \alpha,\beta^{\vee} \rangle \in \mathbbm{Z},
\end{equation}
\end{enumerate}
where $C_{I} = (\langle \alpha,\beta^{\vee} \rangle)_{\alpha,\beta \in I}$ and $C_{I}(\lambda_{s},\alpha)$ is the matrix obtained from the Cartan matrix $C_{I}$ replacing its $\alpha$-row by the row vector $(\langle \lambda_{s},\beta^{\vee} \rangle)_{\beta \in I}$.
\end{thm}

In the above theorem $C_{I}$ denotes the Cartan matrix of the semisimple component $\mathfrak{s}_{P} = {\rm{Lie}}(S_{P})$ and 
\begin{equation}
{\rm{NE}}(X_{P})_{\mathbbm{Z}} = \Big \{ \sum_{i}a_{i}[C_{i}] \in {\rm{NE}}(X_{P}) \ \Big | \ a_{i} \in \mathbbm{Z} \Big \},
\end{equation}
where ${\rm{NE}}(X_{P})$ is the Mori cone of $X_{P}$. As we see, condition (B) is a numerical necessary and sufficient condition to ensure the existence of a splitting ${\bf{E}} \cong {\bf{E}}_{0} \otimes {\bf{L}}_{0}$, such that ${\bf{L}}_{0} \in {\rm{Pic}}(X_{P})$ and $c_{1}({\bf{E}}_{0}) = 0$. Further, we have from item (C) an algebraic checkable criterion for the existence of such a splitting. 

As an application of the above theorem in the setting of Problem \ref{P1}, we have the following result.

\begin{thm}
\label{Hcase1}
Given a homogeneous holomorphic vector bundle ${\bf{E}} \to X_{P}$, defined by some irreducible $P$-module, fixed some $G$-invariant (integral) K\"{a}hler metric $\omega_{0}$ on $X_{P}$, if $f \in L^{2}(X_{P},\mu_{\omega_{0}})$ satisfies 
\begin{equation}
\frac{1}{2\pi}\dashint_{X_{P}}f(x){\rm{d}}\mu_{\omega_{0}}(x) = \frac{\mu_{[\omega_{0}]}({\bf{E}})}{(n-1)!{\rm{Vol}}(X_{P},\omega_{0})},
\end{equation}
and
\begin{equation}
\displaystyle \frac{1}{{\rm{rank}}({\bf{E}})} \int_{C}c_{1}({\bf{E}}) \in \mathbbm{Z}, \ \ \ \ \ \forall [C] \in {\rm{NE}}(X_{P})_{\mathbbm{Z}},
\end{equation}
then there exists a sequence of smooth Hermitian structures $\{{\bf{h}}_{n}\}$ on ${\bf{E}}$, such that 
\begin{equation}
{\rm{K}}({\bf{E}},{\bf{h}}_{n}) \overset{L^2}{\longrightarrow}  f1_{{\bf{E}}} \ \ \ \ {\text{and}} \ \ \ \ {\bf{h}}_{n} \overset{a.e.}{\longrightarrow} {\bf{h}}_{\infty},
\end{equation}
as $n \uparrow +\infty$, where ${\bf{h}}_{\infty}$ is a singular Hermitian structure on ${\bf{E}}$. Moreover, we have
\begin{equation}
{\rm{K}}({\bf{E}},{\bf{h}}_{\infty}) = f1_{{\bf{E}}},
\end{equation}
in the sense of distribution. 
\end{thm}

As we see from the above theorem, by leveraging Cartan's highest weight theory, our algebraic splitting condition provided by Theorem \ref{theoremA} completely bypasses the non-abelian analytic difficulties faced in the Bando-Siu framework, allowing the analytical singular obstruction to be perfectly modeled using Demailly's abelian theory of singular line bundle metrics.

To further underscore the analytical scope of Theorem \ref{Hcase1}, we now formalize a consequence demonstrating that the framework can seamlessly accommodate rough curvature profiles that blow up on analytic smooth subvarieties. This application highlights the transition from classical smooth setups to the distributional realm, breaking standard continuity and boundedness assumptions.

\begin{corollary-non}
\label{corollarysing}
Let ${\bf{E}} \rightarrow X_{P}$ be a homogeneous holomorphic vector bundle defined by some irreducible $P$-module over a rational homogeneous variety $X_{P}$. Suppose that ${\bf{E}}$ satisfies 
\begin{equation}
    \frac{1}{\rm{rank}({\bf{E}})}\int_{C}c_{1}({\bf{E}})\in\mathbb{Z}, \quad \forall[C]\in {\rm{NE}}(X_{P})_{\mathbbm{Z}}.
\end{equation}
Given any analytic smooth subvariety $ Y\subset X_P$ and a real exponent $0 < s < {\rm{codim}}(Y)$, then there exists a singular Hermitian structure ${\bf{h}}_{\infty}$ on ${\bf{E}}$ such that its mean curvature current satisfies
\begin{equation}
    {\rm{K}}({\bf{E}},{\bf{h}}_{\infty}) = \Bigg (\frac{1}{d(x,Y)^{s}} + C_{0} \Bigg) 1_{{\bf{E}}},
\end{equation}
in the sense of distributions, where $d(\cdot,Y)$ denotes the Riemannian distance function to $Y$, and $C_0 \in \mathbbm{R}$ is a topologically determined constant.
\end{corollary-non}

The significance of the above corollary lies in its illustration of the sheer analytical flexibility offered by Theorem \ref{Hcase1}. In classical geometric analysis, particularly within the framework established by Bando and Siu , prescribing a mean curvature candidate inherently requires the target function to remain globally bounded to avoid non-abelian analytical pathologies. Corollary \ref{corollarysing} breaks through this regular barrier by demonstrating that the prescribed mean curvature can accommodate prescribed singularities that remain structurally inaccessible to general pluripotential approaches.

\subsection*{Acknowledgments.}E. M. Correa is supported by S\~{a}o Paulo Research Foundation FAPESP grant 2025/18843-1.

\section{Generalities on flag varieties}\label{generalities}
In this section, we review some basic generalities about flag varieties. For more details on the subject presented in this section, we suggest \cite{Akhiezer}, \cite{Flagvarieties}, \cite{HumphreysLAG}, \cite{BorelRemmert}, \cite{knapp1996lie}.
\subsection{Basic structure of semisimple Lie algebras}
\label{subsec3.1}
Let $\mathfrak{g}^{\mathbbm{C}}$ be a complex semisimple Lie algebra. By fixing a Cartan subalgebra $\mathfrak{h}$ and a simple root system $\Delta \subset \mathfrak{h}^{\ast}$, we have a triangular decomposition of $\mathfrak{g}^{\mathbbm{C}}$ given by
\begin{center}
$\mathfrak{g}^{\mathbbm{C}} = \mathfrak{n}^{-} \oplus \mathfrak{h} \oplus \mathfrak{n}^{+}$, 
\end{center}
where $\mathfrak{n}^{-} = \sum_{\alpha \in \Phi^{-}}\mathfrak{g}_{\alpha}$ and $\mathfrak{n}^{+} = \sum_{\alpha \in \Phi^{+}}\mathfrak{g}_{\alpha}$, here we denote by $\Phi = \Phi^{+} \cup \Phi^{-}$ the root system associated with the simple root system $\Delta \subset \mathfrak{h}^{\ast}$. 

Let us denote by $\kappa$ the Cartan-Killing form of $\mathfrak{g}^{\mathbbm{C}}$. From this, for every  $\alpha \in \Phi^{+}$, $x \in \mathfrak{g}_{\alpha}$, and $y \in \mathfrak{g}_{-\alpha}$, we have $t_{\alpha} \in \mathfrak{h}$, such that $\alpha = \kappa(t_{\alpha},\cdot)$, and $[x,y] = \kappa(x,y)t_{\alpha}$.
Moreover, we can choose $x_{\alpha} \in \mathfrak{g}_{\alpha}$ and $y_{-\alpha} \in \mathfrak{g}_{-\alpha}$, such that 
\begin{equation}
[x_{\alpha},y_{-\alpha}] = h_{\alpha}, \ \ [h_{\alpha},x_{\alpha}] = 2x_{\alpha}, \ \ [h_{\alpha},y_{-\alpha}] = -2y_{-\alpha},
\end{equation}
where 
\begin{equation}
\label{generatorvee}
h_{\alpha} = \frac{2}{\kappa(t_{\alpha},t_{\alpha})}t_{\alpha},
\end{equation}
see for instance \cite{HumphreysLAG}. Hence, we have $\mathfrak{sl}_{2}(\alpha) := {\rm{Span}}_{\mathbbm{C}}\{x_{\alpha},y_{-\alpha},h_{\alpha}\} \cong {\mathfrak{sl}}_{2}(\mathbbm{C})$, for every $\alpha \in \Phi^{+}$. The set $\{x_{\alpha},y_{-\alpha},h_{\alpha}\}$ is called standard ${\mathfrak{sl}}_{2}(\mathbbm{C})$-triple associated with $\alpha \in \Phi^{+}$. 

Now we recall the following definitions.
\begin{definition}
A Lie algebra $\mathfrak{g}$ over $\mathbbm{R}$ is said to be compact if its Killing form $\kappa$ is negative definite.
\end{definition}
\begin{definition}
Let $\mathfrak{g}^{\mathbbm{C}}$ be a Lie algebra over $\mathbbm{C}$. A real subalgebra $\mathfrak{g} \subset \mathfrak{g}^{\mathbbm{C}}$ is said to be a real form of $\mathfrak{g}^{\mathbbm{C}}$ if $\mathfrak{g}^{\mathbbm{C}} = \mathfrak{g} \oplus \sqrt{-1}\mathfrak{g}$ and $\mathfrak{g} \cap \sqrt{-1}\mathfrak{g} = \{0\}$.
\end{definition}
Given a complex semisimple Lie algebra $\mathfrak{g}^{\mathbbm{C}}$, we can construct a compact real form of $\mathfrak{g}^{\mathbbm{C}}$ in the following manner: for each $\alpha \in \Phi$, it is possible to choose a root vector $e_{\alpha} \in \mathfrak{g}_{\alpha}$, such that
\begin{enumerate}
\item $\kappa(e_{\alpha},e_{-\alpha}) = 1$, $[e_{\alpha},e_{-\alpha}]= t_{\alpha}$, $\forall \alpha \in \Phi^{+}$,
\item $[e_{\alpha},e_{\beta}] = m_{\alpha,\beta}e_{\alpha+\beta}$, if $\alpha + \beta \in \Phi$,
\item $[e_{\alpha},e_{\beta}] = 0$,  if $\alpha + \beta \neq 0$ and $\alpha + \beta \notin \Phi$,
\end{enumerate}
such that $m_{\alpha,\beta} \in \mathbbm{R}$ and $m_{-\alpha,-\beta} = - m_{\alpha,\beta}$, $\forall \alpha,\beta \in \Phi$, see for instance \cite{knapp1996lie}. From above, it follows that the set $\{e_{\alpha},e_{-\alpha},t_{\beta} \ | \ \alpha \in \Phi^{+}, \ \beta \in \Delta\}$ defines a basis for $\mathfrak{g}^{\mathbbm{C}}$ (a.k.a. Weyl basis). By means of the above data we define:
\begin{enumerate}
\item[(i)] $\mathfrak{h}_{\mathbbm{R}}:= {\rm{Span}}_{\mathbbm{R}}\big \{ t_{\alpha} \ \big | \ \alpha \in \Delta\big\}$,
\item[(ii)] $A_{\alpha} := e_{\alpha} - e_{-\alpha}$, \ \ $B_{\alpha} = \sqrt{-1}(e_{\alpha} + e_{-\alpha})$, $\alpha \in \Phi^{+}$.
\end{enumerate}
From above, we set
\begin{equation}
\label{cptreal}
\mathfrak{g}:= \sqrt{-1}\mathfrak{h}_{\mathbbm{R}} \oplus \bigg ( \sum_{\alpha \in \Phi^{+}}\underbrace{{\rm{Span}}_{\mathbbm{R}}\{A_{\alpha},B_{\alpha}\}}_{\mathfrak{u}_{\alpha}}\bigg). 
\end{equation}
By construction, it can be shown that the structure constants of $\mathfrak{g}$ defined above lie in $\mathbbm{R}$. Thus, $\mathfrak{g} \subset \mathfrak{g}^{\mathbbm{C}}$ is a real Lie subalgebra. Further, we have $\mathfrak{g}^{\mathbbm{C}} = \mathfrak{g} \oplus \sqrt{-1}\mathfrak{g}$ and $\mathfrak{g} \cap \sqrt{-1}\mathfrak{g} = \{0\}$.

\subsection{Semisimple Lie groups and flag varieties} 
\label{simplelieandflags}
From now on, we fix a complex semisimple Lie algebra $\mathfrak{g}^{\mathbbm{C}}$, and a triangular decomposition
\begin{equation}
\label{triangular}
\mathfrak{g}^{\mathbbm{C}} = \mathfrak{n}^{-} \oplus \mathfrak{h} \oplus \mathfrak{n}^{+}, 
\end{equation}
induced by a triple $(\mathfrak{g}^{\mathbbm{C}},\mathfrak{h},\Delta)$, where $\mathfrak{h}$ is a Cartan subalgebra, and $\Delta$ is a simple root system. Further, we will refer to a compact real form $\mathfrak{g}$ of $\mathfrak{g}^{\mathbbm{C}}$ as being a compact real form as described in Eq. (\ref{cptreal}). Let $G^{\mathbbm{C}}$ be a connected, simply connected, and complex Lie group, such that ${\rm{Lie}}(G^{\mathbbm{C}}) = \mathfrak{g}^{\mathbbm{C}}$. In this setting, we consider the following definitions.

\begin{definition}
A Lie subalgebra $\mathfrak{b} \subset \mathfrak{g}^{\mathbbm{C}}$ is called a Borel subalgebra if $\mathfrak{b}$ is a maximal solvable subalgebra of $\mathfrak{g}^{\mathbbm{C}}$. 
\end{definition}
\begin{definition}
A Lie subgroup $B \subset G^{\mathbbm{C}}$ is called a Borel subgroup if ${\rm{Lie}}(B) \subset \mathfrak{g}^{\mathbbm{C}}$ is a Borel subalgebra.
\end{definition}
Associated with the decomposition given in Eq. (\ref{triangular}) we can define a Borel subalgebra by setting $\mathfrak{b}(\Delta) := \mathfrak{h} \oplus \mathfrak{n}^{+}$. In this last case, $\mathfrak{b}(\Delta)$ is called standard Borel subalgebra relative to $\mathfrak{h}$. Now we have the following result (see for instance \cite{Flagvarieties}, \cite{HumphreysLAG}, \cite{Humphreys}):
\begin{theorem}
\label{Borelconjugate}
Any two Borel (subgroups) subalgebras are conjugate.
\end{theorem}
From the result above, given a Borel subgroup $B \subset G^{\mathbbm{C}}$, up to conjugation, we can always suppose that $B = \exp(\mathfrak{b})$, with $\mathfrak{b} = \mathfrak{b}(\Delta)$. Consider now the following definition.
\begin{definition}
A Lie subalgebra $\mathfrak{p} \subset \mathfrak{g}^{\mathbbm{C}}$ is called parabolic if $\mathfrak{p}$ contains some Borel subalgebra. A Lie subgroup $P \subset G^{\mathbbm{C}}$ is called parabolic if $P$ contains some Borel subgroup.
\end{definition}
Given $I \subset \Delta$, we can construct a parabolic Lie subalgebra by setting 
\begin{equation}
\mathfrak{p}_{I} := \mathfrak{b}(\Delta)\oplus \Big ( \sum_{\alpha \in \langle I \rangle^{-}} \mathfrak{g}_{\alpha}\Big ).
\end{equation}
A Lie subalgebra as above is called standard parabolic Lie subalgebra relative to $\Delta$. Let $P_{I} \subset G^{\mathbbm{C}}$ be the Lie subgroup, such that ${\rm{Lie}}(P_{I}) = \mathfrak{p}_{I}$. By construction, $P_{I}$ is a parabolic Lie subgroup. We can show that $P_{I} = N_{G^{\mathbbm{C}}}(\mathfrak{p}_{I})$, where $N_{G^{\mathbbm{C}}}(\mathfrak{p}_{I})$ is the normalizer in  $G^{\mathbbm{C}}$ of $\mathfrak{p}_{I} \subset \mathfrak{g}^{\mathbbm{C}}$, see for instance \cite[\S 3.1]{Akhiezer}. From Theorem \ref{Borelconjugate}, we have the following result.
\begin{theorem}
Every parabolic Lie subalgebra is conjugated to one and only one standard parabolic Lie subalgebra $\mathfrak{p}_{I}$ relative to $\Delta$, for some $I \subset \Delta$.
\end{theorem}
From the result above, given a parabolic subgroup $P \subset G^{\mathbbm{C}}$, up to conjugation, we can always suppose that $P = P_{I}$, for some $I \subset \Delta$. It is worth observing that 
\begin{equation}
\label{Leviparabolic}
\mathfrak{p}_{I} = \underbrace{\Big ( \mathfrak{h} \oplus \sum_{\alpha \in \langle I \rangle }\mathfrak{g}_{\alpha}\Big)}_{\mathfrak{l}_{P}} \oplus \underbrace{\Big ( \sum_{\alpha \in \Phi^{+} \backslash \langle I \rangle^{+} }\mathfrak{g}_{\alpha}\Big )}_{\mathfrak{u}_{P}^{+}}.
\end{equation}
In the above decomposition (direct sum of vector spaces), $\mathfrak{l}_{P}$ is called the Levi factor and $\mathfrak{u}_{P}^{+}$ is called the nilpotent radical (maximal nilpotent ideal) of $\mathfrak{p}_{I}$. Notice that $\mathfrak{l}_{P}$ is a reductive Lie subalgebra, thus 
\begin{equation}
\mathfrak{l}_{P} =  [\mathfrak{l}_{P},\mathfrak{l}_{P}]\oplus \mathfrak{c}_{P}, 
\end{equation}
where $[\mathfrak{l}_{P},\mathfrak{l}_{P}]$ is semisimple and $\mathfrak{c}_{P}$ is the center of $\mathfrak{l}_{P}$. Moreover, we have 
\begin{equation}
\mathfrak{c}_{P} = \bigcap_{\alpha \in I}\ker(\alpha).
\end{equation}
In the context of Lie groups, we have the Levi decomposition
\begin{equation}
\label{Levidecparabolic}
P_{I} = L_{P}R_{u}(P_{I})^{+} = L_{P} \ltimes R_{u}(P_{I})^{+},
\end{equation}
such that 
\begin{equation}
L_{P} = [L_{P},L_{P}]Z(L_{P})^{0}, 
\end{equation}
where $Z(L_{P})^{0}:= \exp(\mathfrak{c}_{P})$. Here $Z(L_{P}) \subset L_{P}$ denotes the center of $L_{P}$. Further, since in this case $ R_{u}(P_{I})^{+} \subset [P_{I},P_{I}]$, we obtain
\begin{equation}
\label{commutatordec}
P_{I} = [P_{I},P_{I}]Z(L_{P})^{0}.
\end{equation}
In the above setting, we have the following definition.
\begin{definition}
A homogeneous variety $X_{P} = G^{\mathbbm{C}}/P$, where $P \subset G^{\mathbbm{C}}$ is a parabolic subgroup, is called a flag variety (of $G^{\mathbbm{C}}$).
\end{definition}

\begin{remark}
Given $I \subset \Delta$, we shall denote $\Phi_{I}^{\pm}:= \Phi^{\pm} \backslash \langle I \rangle^{\pm}$, such that $\langle I \rangle^{\pm} = \langle I \rangle \cap \Phi^{\pm}$. In particular, we have $\Phi_{I} := \Phi_{I}^{+} \cup \Phi_{I}^{-}$.
\end{remark}

Let $G \subset G^{\mathbbm{C}}$ be the compact connected and simply connected Lie subgroup, such that ${\rm{Lie}}(G) = \mathfrak{g}$ (compact real form). Given a parabolic Lie subgroup $P \subset G^{\mathbbm{C}}$, such that $P = P_{I}$, for some $I \subset \Delta$, considering $\mathfrak{p}_{I} = {\rm{Lie}}(P)$, it follows from Eq. (\ref{cptreal}) and Eq. (\ref{Leviparabolic}) that $\mathfrak{k}:= \mathfrak{p}_{I} \cap \mathfrak{g}$ can be written  in the following way
\begin{equation}
\mathfrak{k} = \sqrt{-1}\mathfrak{h}_{\mathbbm{R}} \oplus \Big ( \sum_{\alpha \in \langle I \rangle^{+}} {\rm{Span}}_{\mathbbm{R}} \{A_{\alpha},B_{\alpha} \}\Big) =  \sqrt{-1}\mathfrak{h}_{\mathbbm{R}} \oplus \Big ( \sum_{\alpha \in \langle I \rangle^{+}} \mathfrak{u}_{\alpha}\Big).
\end{equation}
Considering the closed connected Lie subgroup $K \subset G$, such that ${\rm{Lie}}(K) = \mathfrak{k}$, we obtain the following characterization 
\begin{equation}
X_{P} = G^{\mathbbm{C}}/P = G / P \cap G = G/K.
\end{equation}
From above, we notice that 
\begin{equation}
\label{reductivespacecondition}
\mathfrak{g} = \mathfrak{k} \oplus \mathfrak{m}_{0}, \ \ {\rm{Ad}}(k)\mathfrak{m}_{0} = \mathfrak{m}_{0}, \ \forall k \in K,
\end{equation}
such that $\mathfrak{m}_{0} = \sum_{\alpha \in \Phi_{I}^{+}} \mathfrak{u}_{\alpha}$. We should also observe that 
\begin{equation}
\mathfrak{m} := \mathfrak{m}_{0} \otimes \mathbbm{C} =  \sum_{\alpha \in \Phi_{I}} \mathfrak{g}_{\alpha}. 
\end{equation}
Denoting ${\rm{o}} = eK \in X_{P}$, it follows that 
\begin{equation}
\label{tangentisomorphism}
T_{\rm{o}}X_{P} \cong \mathfrak{g}/\mathfrak{k} \cong \mathfrak{m}_{0} = \sum_{\alpha \in \Phi_{I}^{+}} \mathfrak{u}_{\alpha}.
\end{equation}
Since $X_{P}$ is a homogeneous space of a complex Lie group, it has a natural structure of a complex manifold.

The associated $G$-invariant integrable almost complex structure $J$ commutes with the adjoint action of $K$ on $\mathfrak{m}_{0}$, and is completely determined at ${\rm{o}} \in X_{P}$ by $J \colon \mathfrak{m}\to \mathfrak{m}$, such that 
\begin{equation}
\label{amcplx}
J(e_{-\alpha}) = \sqrt{-1}e_{-\alpha}, \ \ \ J(e_{\alpha}) = -\sqrt{-1}e_{\alpha},
\end{equation}
for all $\alpha\in \Phi^{+}$, see for instance \cite{san2003invariant}. In particular, notice that 
\begin{equation}
\label{isoholomorphictgbundle}
T_{\rm{o}}^{1,0}X_{P} = \sum_{\alpha \in \Phi^{+}_{I}}\mathfrak{g}_{-\alpha}, \ \ \ {\text{and}} \ \ \ T_{\rm{o}}^{0,1}X_{P} = \sum_{\alpha \in \Phi^{+}_{I}}\mathfrak{g}_{\alpha}.
\end{equation}
For more details on almost complex structures on flag varieties, we suggest \cite{san2003invariant}.

\subsection{Highest weight modules} Let us recall some basic facts about the representation theory of complex semisimple Lie algebras, a detailed exposition on the subject to be presented below can be found in \cite{Humphreys}. By keeping the previous notation, for every $\alpha \in \Phi$, we set 
$$\alpha^{\vee} := \frac{2}{\langle \alpha, \alpha \rangle}\alpha.$$ 

\begin{remark}
In what follows, $\forall \phi,\psi \in \mathfrak{h}^{\ast}$, we denote 
\begin{equation}
\label{paringkilling}
\langle \phi, \psi \rangle = \kappa (t_{\phi},t_{\psi}), 
\end{equation}
where $t_{\phi},t_{\psi} \in  \mathfrak{h}$ are, respectively, the dual of $\phi$ and $\psi$ with respect to the Killing form $\kappa$.
\end{remark}
In this setting, the fundamental weights $\{\varpi_{\alpha} \ | \ \alpha \in \Delta\} \subset \mathfrak{h}^{\ast}$ of $(\mathfrak{g}^{\mathbbm{C}},\mathfrak{h})$ are defined by requiring that $\langle \varpi_{\alpha}, \beta^{\vee} \rangle= \delta_{\alpha \beta}$, $\forall \alpha, \beta \in \Delta$. We denote by 
\begin{equation}
\Lambda = \bigoplus_{\alpha \in \Delta}\mathbbm{Z}\varpi_{\alpha}, \ \ \ \Lambda^{+} = \bigoplus_{\alpha \in \Delta}\mathbbm{Z}_{\geq 0}\varpi_{\alpha}, \ \ \ \Lambda^{++} = \bigoplus_{\alpha \in \Delta}\mathbbm{Z}_{> 0}\varpi_{\alpha},
\end{equation}
respectively, the set of weights, the set of integral dominant weights, and the set of strongly dominant weights of $\mathfrak{g}^{\mathbbm{C}}$. From above, given $\alpha \in \Delta$, it follows that 
\begin{equation}
\label{Cartanchange}
\alpha = \sum_{\beta \in \Delta} C_{\alpha \beta}\varpi_{\beta}, \ \ {\text{s.t.}} \ \ C_{\alpha \beta} = \langle \alpha,\beta^{\vee} \rangle, \ \forall \alpha, \beta \in \Delta.
\end{equation}
In other words, the Cartan matrix $C = (C_{\alpha \beta})$ of $\mathfrak{g}^{\mathbbm{C}}$ express the change of basis defined by simple roots and fundamental weights. Let $\Pi \colon \mathfrak{g}^{\mathbbm{C}} \to \mathfrak{gl}(V)$ be an arbitrary finite-dimensional $\mathfrak{g}^{\mathbbm{C}}$-module. By considering its weight space decomposition
\begin{center}
$\displaystyle{V = \bigoplus_{\mu \in \Phi(V)}V_{\mu}},$ \ \ \ \ 
\end{center}
such that $V_{\mu} = \{v \in V \ | \ \Pi(h)v = \mu(h)v, \ \forall h \in \mathfrak{h}\} \neq \{0\}$, $\forall \mu \in \Phi(V) \subset \mathfrak{h}^{\ast}$, we have the following definition.

\begin{definition}
\label{hightweightdef}
A highest weight vector (of weight $\lambda$) in a $\mathfrak{g}^{\mathbbm{C}}$-module $V$ is a non-zero vector $v_{\lambda}^{+} \in V_{\lambda}$, such that 
\begin{center}
$\Pi(x)v_{\lambda}^{+} = 0$, \ \ \ \ \ ($\forall x \in \mathfrak{n}^{+}$).
\end{center}
A weight $\lambda \in \Phi(V)$ associated with a highest weight vector is called highest weight of $V$.
\end{definition}

From above, we consider the following Cartan's highest weight theory results (e.g. \cite{Humphreys}):
\begin{enumerate}

\item[(A)] Every finite-dimensional irreducible $\mathfrak{g}^{\mathbbm{C}}$-module $V$ admits a highest weight vector $v_{\lambda}^{+}$. Moreover, $v_{\lambda}^{+}$ is the unique highest weight vector of $V$, up to non-zero scalar multiples.

\item[(B)] Let $V$ and $W$ be finite-dimensional irreducible $\mathfrak{g}^{\mathbbm{C}}$-modules with highest weight $\lambda \in \mathfrak{h}^{\ast}$. Then, $V$ and $W$ are isomorphic. We will denote by $\Pi_{\lambda} \colon \mathfrak{g} \to \mathfrak{gl}(V(\lambda))$ an irreducible finite-dimensional  $\mathfrak{g}^{\mathbbm{C}}$-module with highest weight $\lambda \in \mathfrak{h}^{\ast}$.

\item[(C)] In the above setting, the following hold:
 
\begin{itemize}
\item[(C1)] If $V$ is an irreducible finite-dimensional  $\mathfrak{g}^{\mathbbm{C}}$-module with highest weight $\lambda \in \mathfrak{h}^{\ast}$, then $\lambda \in \Lambda^{+}$.

\item[(C2)] If $\lambda \in \Lambda^{+}$, then there exists an irreducible finite-dimensional  $\mathfrak{g}^{\mathbbm{C}}$-module $V$, such that $V = V(\lambda)$. 
\end{itemize}

\end{enumerate}
From item (C), it follows that the map $\lambda \mapsto V(\lambda)$ induces an one-to-one correspondence between $\Lambda^{+}$ and the set of isomorphism classes of finite-dimensional irreducible $\mathfrak{g}^{\mathbbm{C}}$-modules.

\begin{remark} In what follows, it will be useful also to consider the following facts:
\begin{enumerate}
\item[(i)] For all $\lambda \in \Lambda^{+}$, we have $V(\lambda) = \Pi_{\lambda}(\mathfrak{U}(\mathfrak{g}^{\mathbbm{C}})) v_{\lambda}^{+}$, where $\mathfrak{U}(\mathfrak{g}^{\mathbbm{C}})$ is the universal enveloping algebra of $\mathfrak{g}^{\mathbbm{C}}$;
\item[(ii)] The fundamental representations are defined by $V(\varpi_{\alpha})$, $\alpha \in \Delta$; 
\item[(iii)] For all $\lambda \in \Lambda^{+}$, we have the following equivalence of induced irreducible representations
\begin{center}
$\Pi_{\lambda} \colon G^{\mathbbm{C}} \to {\rm{GL}}(V(\lambda))$ \ $\Longleftrightarrow$ \ $({\rm{d}}\Pi_{\lambda})_{e} \colon \mathfrak{g}^{\mathbbm{C}} \to \mathfrak{gl}(V(\lambda))$,
\end{center}
such that $\Pi_{\lambda}(\exp(x)) = \exp(({\rm{d}}\Pi_{\lambda})_{e}(x))$, $\forall x \in \mathfrak{g}^{\mathbbm{C}}$, notice that $G^{\mathbbm{C}} = \langle \exp(\mathfrak{g}^{\mathbbm{C}}) \rangle$. For the sake of simplicity, we denote $({\rm{d}}\Pi_{\lambda})_{e}$ also by $\Pi_{\lambda}$.
\end{enumerate}
\end{remark}

\subsection{Representation Theory of Parabolic Lie Groups} For the details on the subject presented in this subsection, we suggest \cite{jantzen2003representations} and \cite{milne2017algebraic}.

Given a parabolic Lie subgroup $P = P_{I}\subset G^{\mathbbm{C}}$, let $(L_{P},T^{\mathbbm{C}})$ be the split reductive group defined by its Levi component $L_{P} \subset P$. Considering the root datum of $(L_{P},T^{\mathbbm{C}})$ defined as 
\begin{center}
$(\mathbbm{X}(T^{\mathbbm{C}}),\langle I \rangle, \alpha \mapsto \alpha^{\vee})$,
\end{center}
such that 
\begin{equation}
\label{charactertrosuabelian}
\mathbbm{X}(T^{\mathbbm{C}}) = {\rm{Hom}}(T^{\mathbbm{C}},\mathbbm{C}^{\times}) \cong \Lambda = \bigoplus_{\alpha \in \Delta} \mathbbm{Z}\varpi_{\alpha},
\end{equation}
is the character group of $T^{\mathbbm{C}}$, we can show that
\begin{equation}
\label{characterparabolic}
{\rm{Hom}}(P,\mathbbm{C}^{\times})  \cong {\rm{Hom}}(L_{P},\mathbbm{C}^{\times})  \cong \Big \{ \chi \in \mathbbm{X}(T^{\mathbbm{C}}) \ \Big | \ \langle \chi, \alpha^{\vee}\rangle = 0, \ \ \forall \alpha \in I\Big \}.
\end{equation}
The isomorphism above is obtained from the restriction homomorphism
\begin{equation}
{\rm{res}} \colon {\rm{Hom}}(P,\mathbbm{C}^{\times}) \to \mathbbm{X}(T^{\mathbbm{C}}), \ \ \chi \mapsto \chi|_{T^{\mathbbm{C}}}, 
\end{equation}
that is, ${\rm{Hom}}(P,\mathbbm{C}^{\times}) \cong {\rm{Im}}({\rm{res}})$. In the above setting, considering the subgroups 
\begin{equation}
\begin{split}
\mathbbm{X}_{0}(T^{\mathbbm{C}}) & := \Big \{ \chi \in \mathbbm{X}(T^{\mathbbm{C}}) \ \Big | \ \langle \chi, \alpha^{\vee}\rangle = 0, \ \ \forall \alpha \in I\Big \}, \\ \mathbbm{X}_{0}'(T^{\mathbbm{C}}) &:= \bigoplus_{\alpha \in \langle I \rangle}\mathbbm{Q}\alpha \cap \mathbbm{X}(T^{\mathbbm{C}}) = \langle I \rangle_{\mathbbm{Z}},
\end{split}
\end{equation}
we can describe the character groups of the subtorus $Z(L_{P})^{0} \subset T^{\mathbbm{C}}$ as follows 
\begin{equation}
 \mathbbm{X}(Z(L_{P})^{0}) \cong \mathbbm{X}(T^{\mathbbm{C}})/\mathbbm{X}_{0}'(T^{\mathbbm{C}}),
\end{equation}
for more details on the above facts, see for instance \cite[Part II, p. 169]{jantzen2003representations}. 

\begin{remark}
From now on, we will not distinguish the isomorphic abelian groups described in Eq. (\ref{charactertrosuabelian}) and Eq. (\ref{characterparabolic}). Therefore, we have
\begin{equation}
\label{identificationcharacter}
{\rm{Hom}}(P,\mathbbm{C}^{\times}) = \mathbbm{X}_{0}(T^{\mathbbm{C}}) = \bigoplus_{\alpha \in \Delta \backslash I}\mathbbm{Z}\varpi_{\alpha}.
\end{equation}
\end{remark}

Let $P  = P_{I}\subset G^{\mathbbm{C}}$ be a parabolic subgroup. Given a finite-dimensional irreducible representation $\theta \colon P \to {\rm{GL}}(V)$, from the Levi decomposition
\begin{center}
$P= L_{P}R_{u}(P)^{+} = L_{P} \ltimes R_{u}(P)^{+}$,
\end{center}
it follows that the unipotent radical $R_{u}(P)^{+} \subset P$ acts trivially on $V$, i.e., $\theta|_{R_{u}(P)^{+}} = {\rm{id}}_{V}$. Therefore, the finite-dimensional irreducible representations of $P$ can be completely classified by finite-dimensional irreducible representations of its Levi component $L_{P} \subset P$. 

In order to classify the finite-dimensional irreducible representations of $L_{P}$, let us introduce some terminology. Considering the root datum
\begin{equation}
(\mathbbm{X}(T^{\mathbbm{C}}),\langle I \rangle, \alpha \mapsto \alpha^{\vee}),
\end{equation}
we have the following definition.

\begin{definition}
Given $\chi \in \mathbbm{X}(T^{\mathbbm{C}})$, we say that $\chi$ is dominant for $L_{P}$ if 
\begin{equation}
\langle \chi,\alpha^{\vee} \rangle \geq 0, \ \ \forall \alpha \in I.
\end{equation}
Under the identification $\mathbbm{X}(T^{\mathbbm{C}}) \cong \Lambda$, we say that an integral weight $\lambda \in \mathbbm{X}(T^{\mathbbm{C}})$ is dominant for $L_{P}$ if $\langle \lambda,\alpha^{\vee} \rangle \geq 0, \forall \alpha \in I$.
\end{definition}
In the above setting, we denote by 
\begin{equation}
 {\rm{Rep}}(P)_{\text{irrd.}} := \Big \{ \chi \in \mathbbm{X}(T^{\mathbbm{C}}) \ \Big | \ \langle \chi, \alpha^{\vee}\rangle \geq 0, \ \ \forall \alpha \in I\Big \}.
\end{equation}
the set of characters which are dominant for $L_{P}$. Since $L_{P}$ is a reductive Lie group, the classification of its finite-dimensional irreducible representations follows from the following theorem.

\begin{theorem}
\label{L_Pirreduciblemodule}
For every dominant weight $\lambda \in {\rm{Rep}}(P)_{\text{irrd.}}$, there exists a finite dimensional irreducible representation $W(\lambda)$ of $L_{P}$, unique up to isomorphism, that decomposes as a representation of $T^{\mathbbm{C}}$ into
\begin{equation}
W(\lambda) = \bigoplus_{\mu \in \Delta(W(\lambda))}W(\lambda)_{\mu},
\end{equation}
such that 
\begin{enumerate}
\item[(a)] $\lambda \in \Delta(W(\lambda))$ and $\dim(W(\lambda)_{\lambda}) = 1$,
\item[(b)] $\forall \mu \in \Delta(W(\lambda))$, we have $\mu = \lambda - \sum_{\alpha \in I}m_{\alpha}\alpha$, $m_{\alpha} \in \mathbbm{Z}_{\geq 0}$, $\forall \alpha \in I$.
\end{enumerate}
Moreover, every finite dimensional irreducible representation of $L_{P}$ is isomorphic to $W(\lambda)$ for a unique dominant $\lambda \in {\rm{Rep}}(P)_{\text{irrd.}}$.
\end{theorem}

\begin{remark}
\label{character1-dimmod}
In the above setting, it is worth pointing out that 
\begin{equation}
 {\rm{Rep}}(P)_{\text{irrd.}} = \Big (\bigoplus_{\alpha \in I}\mathbbm{Z}_{\geq 0}\varpi_{\alpha} \Big )\oplus \Big ( \bigoplus_{\alpha \in \Delta \backslash I}\mathbbm{Z}\varpi_{\alpha} \Big ).
 \end{equation}
Also, notice that ${\rm{Hom}}(P_{I},\mathbbm{C}^{\times}) = \mathbbm{X}_{0}(T^{\mathbbm{C}})  \subset {\rm{Rep}}(P)_{\text{irrd.}} $.
\end{remark}
From the last theorem, we have the following corollary.

\begin{corollary}
The set of isomorphism classes of finite dimensional irreducible representations of a parabolic Lie subgroup $P \subset G^{\mathbbm{C}}$ is in one-to-one correspondence with the set of dominant weights (characters) for $L_{P} \subset P$.
\end{corollary}

\begin{remark}
Considering the semisimple component $S_{P} := [L_{P},L_{P}]$ of $L_{P}$ and its Lie algebra $\mathfrak{s}_{P} = {\rm{Lie}}(S_{P})$, we shall denote 
\begin{equation}
\Lambda_{\mathfrak{s}_{P}}^{+} := \bigoplus_{\alpha \in I}\mathbbm{Z}_{\geq 0}\varpi_{\alpha} \ \ \ \text{and} \ \ \ \Lambda_{P} := \bigoplus_{\alpha \in \Delta \backslash I}\mathbbm{Z}\varpi_{\alpha},
\end{equation}
notice that $\Lambda_{P} = \mathbbm{X}_{0}(T^{\mathbbm{C}}) = {\rm{Hom}}(P,\mathbbm{C}^{\times})$, see Eq. (\ref{identificationcharacter}). From above, we have that the isomorphism classes of irreducible representations of $P$ (and $L_{P}$) is parameterized by
\begin{equation}
{\rm{Rep}}(P)_{\text{irrd.}} = \Lambda_{\mathfrak{s}_{P}}^{+} \oplus \Lambda_{P}
\end{equation}
\end{remark}

\begin{remark}
In view of the above corollary and Remark \ref{character1-dimmod}, it follows that the one-dimensional representations of $P$ are classified by $\mathbbm{X}_{0}(T^{\mathbbm{C}})  \subset {\rm{Rep}}(P)_{\text{irrd.}}$. In this case, given $\chi \in \mathbbm{X}_{0}(T^{\mathbbm{C}})$ we denote by $\mathbbm{C}_{\chi} = \mathbbm{C}$ the associated one-dimensional representations of $P$ defined by 
\begin{center}
$p \cdot z = \chi(p)z$, \ \ $\forall p \in P$, \ \ $\forall z \in \mathbbm{C}$.
\end{center}
In particular, if $(d\chi)_{e} = \lambda$, we also denote $\mathbbm{C}_{\lambda}$ to mean $\mathbbm{C}_{\chi}$.
\end{remark}

Now we consider the following result.

\begin{lemma}
\label{irredparabolic}
Given $\lambda \in {\rm{Rep}}(P)_{\text{irrd.}}$, such that $\lambda = \lambda_{s} + \lambda_c$, where $\lambda_{s} \in \Lambda_{\mathfrak{s}_{P}}$ and $\lambda_{c} \in \Lambda_{P}$, then we have an isomorphism of finite-dimensional irreducible $L_{P}$-representations
\begin{equation}
W(\lambda) \cong W(\lambda_{s}) \otimes \mathbbm{C}_{\lambda_{c}}.
\end{equation}
\end{lemma}
\begin{proof}
Considering $\lambda \in {\rm{Rep}}(P)_{\text{irrd.}}$, such that $\lambda = \lambda_{s} + \lambda_c$, where $\lambda_{s} \in \Lambda_{\mathfrak{s}_{P}}^{+}$ and $\lambda_{c} \in \Lambda_{P}$, since $\lambda_{s} \in {\rm{Rep}}(P)_{\text{irrd.}}$, it follows from Theorem \ref{L_Pirreduciblemodule} that there exists a finite-dimensional irreducible representation $W(\lambda_{s})$ of $L_{P}$ with highest weight $\lambda_{s}$. 

Therefore, since tensor product $W(\lambda_{s}) \otimes \mathbbm{C}_{\lambda_{c}}$ is a finite-dimensional irreducible representation of $L_{P}$ with highest weight $\lambda_{s} + \lambda_{c} = \lambda$, we conclude from Theorem \ref{L_Pirreduciblemodule} that $W(\lambda) \cong W(\lambda_{s}) \otimes \mathbbm{C}_{\lambda_{c}}$ as finite-dimensional irreducible $L_{P}$-representations.
\end{proof}

\begin{remark}
In the case that $I = \varnothing$, that is, $P = B$ (Borel subgroup), it follows that every irreducible representation of $B$ is one-dimensional. Thus, we have ${\rm{Rep}}(B)_{\text{irrd.}} = {\rm{Hom}}(B,\mathbbm{C}^{\times})$.
\end{remark}

\section{Homogeneous vector bundles} We start this section recalling some generalities on Hermitian vector bundles which can be found in \cite{Kobayashi+1987}, \cite{wells1980differential}, \cite{MR2093043}, \cite{rudolph2017differential}, \cite{Griffiths}. Let $X$ be a complex manifold. Given a $C^{\infty}$-complex vector bundle ${\bf{E}} \to X$, we denote
\begin{itemize}
\item[(1)] $\mathcal{A}^{k}({\bf{E}}) := \Gamma^{\infty}(\bigwedge^{k}(TX)_{\mathbbm{C}} \otimes {\bf{E}}) = {\bf{E}}$-valued complex $k$-differential forms;
\item[(2)] $\mathcal{A}^{0}({\bf{E}}) = \Gamma^{\infty}({\bf{E}})$ and $\mathcal{A}^{k}({\bf{E}}) = \bigoplus_{p+q = k}\mathcal{A}^{p,q}({\bf{E}})$, such that $\mathcal{A}^{p,q}({\bf{E}}) = \Gamma(\Lambda^{p,q}TX \otimes {\bf{E}})$.
\end{itemize}
A connection $\nabla$ on ${\bf{E}} \to X$ is a $\mathbbm{C}$-linear homomorphism 
\begin{equation}
\nabla \colon \mathcal{A}^{0}({\bf{E}}) \to  \mathcal{A}^{1}({\bf{E}}),
\end{equation}
satisfying the Leibniz rule
\begin{equation}
\nabla(f{\bf{s}}) = {\rm{d}}f \otimes {\bf{s}} + f \nabla {\bf{s}}, \ \ \ \forall f \in C^{\infty}(X,\mathbbm{C}), \ \ \forall {\bf{s}} \in \mathcal{A}^{0}({\bf{E}}).
\end{equation}
By extending the connection $\nabla$ to a $\mathbbm{C}$-linear map $\nabla \colon \mathcal{A}^{p}({\bf{E}}) \to  \mathcal{A}^{p+1}({\bf{E}})$, satisfying
\begin{equation}
\nabla (\eta \otimes {\bf{s}}) = \eta \wedge (\nabla {\bf{s}}) + {\rm{d}}\eta \otimes {\bf{s}}, \ \ \ \forall \eta \in \Omega^{p}(X;\mathbbm{C}), \ \ \forall {\bf{s}} \in \mathcal{A}^{0}({\bf{E}}),
\end{equation}
we define the curvature $F_{\nabla}$ of $\nabla$ as being
\begin{equation}
F_{\nabla} := \nabla \circ \nabla \colon \mathcal{A}^{0}({\bf{E}}) \to  \mathcal{A}^{2}({\bf{E}}).
\end{equation}
It is straightforward to show that $F_{\nabla}$ is $C^{\infty}(X,\mathbbm{C})$-linear, so we can consider 
\begin{equation}
F_{\nabla} \in \mathcal{A}^{2}({\rm{End}}({\bf{E}})).
\end{equation}
A holomorphic vector bundle is defined by a $C^{\infty}$-complex vector bundle ${\bf{E}} \to X$ together with an operator (a.k.a. holomorphic structure)
\begin{equation}
\bar{\partial}_{{\bf{E}}} \colon \mathcal{A}^{p,q}({\bf{E}}) \to \mathcal{A}^{p,q+1}({\bf{E}}),
\end{equation}
satisfying the following conditions\footnote{A section ${\bf{s}} \colon U \subseteq X_{P} \to {\bf{E}}$ is said to be holomorphic if $\bar{\partial}_{{\bf{E}}} {\bf{s}} \equiv 0$. The space of global holomorphic sections of a holomorphic vector bundle is denoted by ${\rm{H}}^{0}(X_{P},{\bf{E}})$.}:
\begin{enumerate}
\item[(1)] $\bar{\partial}_{{\bf{E}}} \circ \bar{\partial}_{{\bf{E}}} = 0$;
\item[(2)] $\bar{\partial}_{{\bf{E}}} (f{\bf{s}}) = \bar{\partial}f \otimes {\bf{s}} + f \bar{\partial}_{{\bf{E}}} {\bf{s}}, \ \ \ \forall f \in C^{\infty}(X,\mathbbm{C}), \ \ \forall {\bf{s}} \in \mathcal{A}^{0}({\bf{E}}).$
\end{enumerate}
In the presence of a holomorphic structure $\bar{\partial}_{{\bf{E}}}$ on ${\bf{E}} \to X$, we say that a connection $\nabla$ is compatible with $\bar{\partial}_{{\bf{E}}}$ if 
\begin{equation}
\nabla^{0,1} = \bar{\partial}_{{\bf{E}}},
\end{equation}
here we consider the decomposition $\mathcal{A}^{1}({\bf{E}}) = \mathcal{A}^{1,0}({\bf{E}}) \oplus \mathcal{A}^{0,1}({\bf{E}})$. In this work we are concerned with $C^{\infty}$-complex vector bundle ${\bf{E}} \to X$ endowed with a Hermitian structure ${\bf{h}}$ (i.e., a $\mathbbm{C}$-anti-linear isomorphism ${\bf{h}} \colon {\bf{E}} \to {\bf{E}}^{\ast}$). We shall refer to a pair $({\bf{E}},{\bf{h}})$ as a Hermitian vector bundle. In this setting, we have the following standard result (e.g. \cite{Kobayashi+1987}, \cite{rudolph2017differential}).

\begin{theorem}
Let $({\bf{E}},{\bf{h}})$ be a holomorphic Hermitian vector bundle. Then, there exists a unique connection $\nabla$ (Chern connection), such that 
\begin{enumerate}
\item[(h1)] $\nabla {\bf{h}} \equiv 0$;
\item[(h2)] $\nabla^{0,1} = \bar{\partial}_{{\bf{E}}}$.
\end{enumerate}
\end{theorem}
\begin{remark}
In the setting of the above theorem, a connection satisfying condition (h1) is said to be compatible with ${\bf{h}}$. It is worth mentioning that, locally, from the conditions (h1) and (h2), we have the Chern connection $\nabla = \nabla^{1,0} + \nabla^{0,1}$ given by
\begin{equation}
\nabla^{1,0} = \partial + A \ \ \ \ {\text{and}} \ \ \ \ \nabla^{0,1} = \bar{\partial}, 
\end{equation}
such that $A = {\bf{h}}^{-1}\partial {\bf{h}}$ is a matrix of $(1,0)$-forms on $X_{P}$. From the conditions (h1) and (h2), we also have that the curvature $F_{\nabla} \in \mathcal{A}^{2}({\rm{End}}({\bf{E}}))$ of the Chern connection $\nabla$ is skew-Hermitian w.r.t. ${\bf{h}}$, and it is given locally by
\begin{equation}
F_{\nabla} = {\rm{d}}A + A \wedge A,
\end{equation}
i.e., we have  $F_{\nabla} \in \mathcal{A}^{1,1}({\rm{End}}({\bf{E}}))$. We shall denote the curvature of the Chern connection by $F({\bf{h}})$ in order to emphasize the dependence on the underlying Hermitian structure ${\bf{h}}$.
\end{remark}

With this general Hermitian framework established, we now restrict our attention to the homogeneous setting. Consider the canonical holomorphic $P$-principal bundle $P \hookrightarrow G^{\mathbbm{C}} \to G^{\mathbbm{C}}/P$. 

For every finite-dimensional representation $\theta \colon P \to {\rm{GL}}(W)$ we can form an associated (homogeneous) holomorphic vector bundle
\begin{equation}
{\bf{E}}_{\theta} :=  G^{\mathbbm{C}} \times_{P}W = (G^{\mathbbm{C}} \times W)/ \sim, \ {\text{s.t.}} \ \ \ (gp,v) \sim (g,\theta(p)v).
\end{equation}
By choosing a trivializing good open covering $G^{\mathbbm{C}}/P = \bigcup_{i \in J}U_{i}$ for the holomorphic $P$-principal bundle $P \hookrightarrow G^{\mathbbm{C}} \to G^{\mathbbm{C}}/P$, in terms of $\check{C}$ech cocycles we can write 
\begin{equation}
\label{Pcocycle}
G^{\mathbbm{C}} = \Big \{(U_{i})_{i \in J}, \psi_{ij} \colon U_{i} \cap U_{j} \to P \Big \}.
\end{equation}
From above, it follows that 
\begin{equation}
{\bf{E}}_{\theta} =  \Big \{(U_{i})_{i \in J}, (\theta \circ \psi_{ij})\colon U_{i} \cap U_{j} \to {\rm{GL}}(W) \Big \}.
\end{equation}
The construction above defines an equivalence between the category of homogeneous holomorphic vector bundles over $G^{\mathbbm{C}}/P$ and the category of finite-dimensional representations of $P$.\\

We observe that, if we have a parabolic Lie subgroup $P \subset G^{\mathbbm{C}}$, such that $P = P_{I}$, for some $I \subset \Delta$, the identification provided in Eq. (\ref{characterparabolic}) shows us that 
\begin{equation}
\label{generatorspicard}
{\text{Hom}}(P,\mathbbm{C}^{\times}) = \mathbbm{X}_{0}(T^{\mathbbm{C}}) = \big \langle \vartheta_{{\varpi}_{\alpha}}\  \big | \ \ \alpha \in \Delta \backslash I \big \rangle = \Lambda_{P},
\end{equation}
where $\vartheta_{{\varpi}_{\alpha}}$ is the character defined by the fundamental weight $\varpi_{\alpha}$, $\alpha \in I$. In particular, since ${\text{Hom}}(P,\mathbbm{C}^{\times})$ classifies the one-dimensional representations of $P$, considering 
\begin{equation}
\label{linecocycle}
 \mathscr{O}_{\alpha}(1) := {\bf{E}}_{\vartheta_{{\varpi}_{\alpha}}^{-1}} = \Big \{(U_{i})_{i \in J}, (\vartheta_{\varpi_{\alpha}}^{-1} \circ \psi_{i j}) \colon U_{i} \cap U_{j} \to \mathbbm{C}^{\times} \Big \},
\end{equation}
for every $\alpha \in \Delta \backslash I$. We also denote 
\begin{equation}
\mathscr{O}_{\alpha}(k) := \mathscr{O}_{\alpha}(1)^{\otimes k}, \ \ \ \forall k \in \mathbbm{Z}, \ \ \forall \alpha \in \Delta \backslash I.
\end{equation}
Notice that $\mathscr{O}_{\alpha}(0) = \mathcal{O}_{X_{P}}$, where $\mathcal{O}_{X_{P}}$ is the structure sheaf of $X_{P}$.

In the above setting we consider $\mathbbm{C}_{-\varpi_{\alpha}}$ as being the one-dimensional $P$-module defined by the character $\vartheta_{\varpi_{\alpha}}^{-1}$. As we shall see, the choice of $\vartheta_{\varpi_{\alpha}}^{-1}$ instead of $\vartheta_{\varpi_{\alpha}}$ is made so that the concept of "positivity" for Chern classes is compatible for the sheaf-theoretic and differential-geometric point of view.

\begin{remark}
By Lemma \ref{irredparabolic}, every irreducible representation $(\theta, W(\lambda))$ of $P \subset G^{\mathbbm{C}}$ splits algebraically as $W(\lambda_s) \otimes \mathbb{C}_{\lambda_c}$. This induces a canonical geometric splitting of the associated homogeneous bundle as follows
\begin{equation}
\label{irreddecomp}
{\bf{E}}_{\theta} := {\bf{E}}_{s} \otimes {\bf{L}}_{c},
\end{equation}
such that ${\bf{E}}_{s} = G^{\mathbbm{C}} \times_{P} W(\lambda_{s})$ and ${\bf{L}}_{c} = G^{\mathbbm{C}} \times_{P}\mathbbm{C}_{\lambda_{c}}$ is a homogeneous line bundle.
\end{remark}

The isomorphism given in Eq. (\ref{generatorspicard}) can be understood in terms of the following map from ${\rm{Hom}}(P,\mathbbm{C}^{\times})$ to $H^{1}(X_{P},\mathcal{O}_{X_{P}}^{\ast})$:
\begin{equation}
\label{cocycle}
\vartheta  \mapsto \Big \{(U_{i})_{i \in J}, (\vartheta^{-1} \circ \psi_{i j}) \colon U_{i} \cap U_{j} \to \mathbbm{C}^{\times} \Big \} = {\bf{E}}_{\vartheta^{-1}},
\end{equation}
for every $\vartheta \in {\rm{Hom}}(P,\mathbbm{C}^{\times})$. In particular, in terms of generators, we have $\vartheta_{\varpi_{\alpha}} \mapsto \mathscr{O}_{\alpha}(1)$, for every $\alpha \in \Delta \backslash I$, so
\begin{equation}
{\rm{Pic}}(X_{P}) = H^{1}(X_{P},\mathcal{O}_{X_{P}}^{\ast}) = \big \langle  \mathscr{O}_{\alpha}(1) \ \big |  \ \alpha \in \Delta \backslash I \big \rangle.
\end{equation}
The second equality in the above equation follows from the following fact. If ${\bf{L}} \to X_{P}$ is a holomorphic line bundle, since $H^{2}(X_{P},\mathbb{Z})$ is discrete and $G^{\mathbb{C}}$ is connected, the map $g \mapsto c_{1}(g^{\ast}{\bf{L}})$ must be constant, so ${\bf{L}} \cong g^{\ast}{\bf{L}}$, $\forall g \in G^{\mathbb{C}}$. Hence, we conclude that every holomorphic line bundle over $X_{P}$ is a homogeneous line bundle.

Further, from the short exact sequence of sheaves 
\begin{center}
\begin{tikzcd}0 \arrow[r] &  \underline{\mathbbm{Z}} \arrow[r] & \mathcal{O}_{X_{P}} \arrow[r,"\exp"]  & \mathcal{O}_{X_{P}}^{\ast} \arrow[r] & 0 \end{tikzcd}
\end{center}
we get a long exact sequence 
\begin{center}
\begin{tikzcd} \cdots \arrow[r] & H^{1}(X_{P},\mathcal{O}_{X_{P}}^{\ast}) \arrow[r,"\delta"]  & H^{2}(X_{P},\mathbbm{Z}) \arrow[r] & \cdots \end{tikzcd}
\end{center}
where $\delta$ is the Bockstein operator (e.g. \cite{wells1980differential}). From above, through the isomorphism provided by Eq. (\ref{cocycle}), we can define
\begin{equation}
c_{1}(\mathscr{O}_{\alpha}(1)) := \delta(\mathscr{O}_{\alpha}(1)) = \delta \big \{ \vartheta_{\varpi_{\alpha}}^{-1} \circ \psi_{ij}\big\},
\end{equation}
for every $\alpha \in \Delta \backslash I$ (e.g. \cite{Griffiths}). In order to describe the $(1,1)$-form which represents the cohomology class $\delta \big \{ \vartheta_{\varpi_{\alpha}}^{-1} \circ \psi_{ij}\} \in H^{2}(X_{P},\mathbbm{Z})$, $\alpha \in \Delta \backslash I$, we consider the following theorem.

\begin{theorem}[Azad-Biswas, \cite{AZAD}]
\label{AZADBISWAS}
Let $\omega \in \Omega^{1,1}(X_{P})^{G}$ be a closed invariant real $(1,1)$-form, then we have

\begin{center}

$\pi^{\ast}\omega = \sqrt{-1}\partial \overline{\partial}\varphi$,

\end{center}
where $\pi \colon G^{\mathbbm{C}} \to X_{P}$ is the natural projection, and $\varphi \colon G^{\mathbbm{C}} \to \mathbbm{R}$ is given by 
\begin{center}
$\varphi(g) = \displaystyle \sum_{\alpha \in \Delta \backslash I}c_{\alpha}\log \Big ( \big | \big |\Pi_{\varpi_{\alpha}}(g)v_{\varpi_{\alpha}}^{+}\big | \big | \Big )$, \ \ \ \ $(\forall g \in G^\mathbbm{C})$
\end{center}
with $c_{\alpha} \in \mathbbm{R}$, for every $\alpha \in \Delta \backslash I$. Conversely, every function $\varphi$ as above defines a closed invariant real $(1,1)$-form $\omega_{\varphi} \in \Omega^{1,1}(X_{P})^{G}$. Moreover, $\omega_{\varphi}$ defines a $G$-invariant K\"{a}hler form on $X_{P}$ if and only if $c_{\alpha} > 0$, for every $\alpha \in \Delta \backslash I$.
\end{theorem}
\begin{remark}
\label{innerproduct}
It is worth pointing out that the norm $|| \cdot ||$ considered in the above theorem is a norm induced from some $G$-invariant inner product $\langle \cdot, \cdot \rangle_{\alpha}$ fixed on $V(\varpi_{\alpha})$, $\forall \alpha \in \Delta \backslash I$. 
\end{remark}
Given $\mathscr{O}_{\alpha}(1) \in {\text{Pic}}(X_{P})$, for some $\alpha \in \Delta \backslash I$, consider a good cover $X_{P} = \bigcup_{i \in J} U_{i}$ which trivializes both $P \hookrightarrow G^{\mathbbm{C}} \to X_{P}$ and $ \mathscr{O}_{\alpha}(1) \to X_{P}$. Taking a collection of local sections $({\bf{s}}_{i})_{i \in J}$, such that ${\bf{s}}_{i} \colon U_{i} \to G^{\mathbbm{C}}$, we can define $q_{i} \colon U_{i} \to \mathbbm{R}^{+}$, such that 
\begin{equation}
\label{functionshermitian}
q_{i}(p) := \frac{1}{ \big | \big |\Pi_{\varpi_{\alpha}}({\bf{s}}_{i}(p))v_{\varpi_{\alpha}}^{+} \big | \big |^{2}},
\end{equation}
$\forall p \in U_{i}$, and $\forall i \in J$. Since 
\begin{enumerate}
\item ${\bf{s}}_{j} = {\bf{s}}_{i}\psi_{ij}$ on $U_{i} \cap U_{j} \neq \emptyset$,
\item $\Pi_{\varpi_{\alpha}}(a)v_{\varpi_{\alpha}}^{+} = \vartheta_{\varpi_{\alpha}}(a)v_{\varpi_{\alpha}}^{+}$, $\forall a \in P$, and $\forall \alpha \in \Delta \backslash I$, 
\end{enumerate}
the collection of functions $(q_{i})_{i \in J}$ satisfies
\begin{equation}
q_{j} = \frac{1}{|(\vartheta_{\varpi_{\alpha}} \circ \psi_{ij})|^{2}}q_{i} = |(\vartheta_{\varpi_{\alpha}}^{-1} \circ \psi_{ij})|^{2}q_{i},
\end{equation}
on $U_{i} \cap U_{j} \neq \emptyset$, $\forall i,j \in J$. From this, we can define a Hermitian structure ${\bf{h}}$ on $\mathscr{O}_{\alpha}(1)$ by taking on each trivialization $\tau_{i} \colon U_{i} \times \mathbbm{C} \to \mathscr{O}_{\alpha}(1)$ the metric defined by
\begin{equation}
\label{hermitian}
{\bf{h}}(\tau_{i}(x,v),\tau_{i}(x,w)) := \frac{v\overline{w}}{\big|\big|\Pi_{\varpi_{\alpha}}({\bf{s}}_{i}(x))v_{\varpi_{\alpha}}^{+}\big|\big|^{2}},
\end{equation}
for every $(x,v),(x,w) \in U_{i} \times \mathbbm{C}$. The Hermitian metric above induces a Chern connection 
\begin{center}
$\nabla \colon \mathcal{A}^{0}(\mathscr{O}_{\alpha}(1)) \to \mathcal{A}^{1}(\mathscr{O}_{\alpha}(1))$, 
\end{center}
such that 
\begin{equation}
\nabla {\bf{\sigma}}_{i} = \partial (\log q_{i} )\otimes {\bf{\sigma}}_{i},
\end{equation}
where ${\bf{\sigma}}_{i}(\cdot) = \tau_{i}(\cdot,1), \forall i \in J$. From this, the curvature of the Chern connection $\nabla$ described above is given by
\begin{equation}
\displaystyle \nabla (\nabla \sigma_{i}) = \partial \overline{\partial}\log \Big ( \big | \big | (\Pi_{\varpi_{\alpha}} \circ {\bf{s}}_{i})v_{\varpi_{\alpha}}^{+}\big | \big |^{2} \Big) \otimes \sigma_{i} = F({\bf{h}}) \otimes \sigma_{i}.
\end{equation}
Therefore, by considering the closed $G$-invariant $(1,1)$-form ${\bf{\Omega}}_{\alpha} \in \Omega^{1,1}(X_{P})^{G}$, which satisfies $\pi^{\ast}{\bf{\Omega}}_{\alpha} = \sqrt{-1}\partial \overline{\partial} \varphi_{\varpi_{\alpha}}$, where $\pi \colon G^{\mathbbm{C}} \to G^{\mathbbm{C}} / P = X_{P}$ is the canonical projection map, and 
\begin{equation}
\varphi_{\varpi_{\alpha}}(g) := \frac{1}{2\pi}\log \Big ( \big | \big |\Pi_{\varpi_{\alpha}}(g)v_{\varpi_{\alpha}}^{+} \big | \big |^{2} \Big ), 
\end{equation}
$\forall g \in G^{\mathbbm{C}}$, we have ${\bf{\Omega}}_{\alpha} \myeq (\pi \circ {\bf{s}}_{i})^{\ast}{\bf{\Omega}}_{\alpha} \myeq \frac{\sqrt{-1}}{2\pi}F({\bf{h}})$, i.e., 
\begin{equation}
c_{1}(\mathscr{O}_{\alpha}(1)) = \delta \big \{ \vartheta_{\varpi_{\alpha}}^{-1} \circ \psi_{ij} \big \} = [ {\bf{\Omega}}_{\alpha}] = \bigg [ \frac{\sqrt{-1}}{2\pi}F({\bf{h}})\bigg ], 
\end{equation}
for every $\alpha \in \Delta \backslash I$ (e.g. \cite{Griffiths}, \cite{wells1980differential}). Consider now the following result.

\begin{lemma}
\label{funddynkinline}
Let $\mathbbm{P}_{\beta}^{1} = \overline{\exp(\mathfrak{g}_{-\beta}){\rm{o}}} \subset X_{P}$, such that $\beta \in \Phi_{I}^{+}$. Then, 
\begin{equation}
\int_{\mathbbm{P}_{\beta}^{1}} {\bf{\Omega}}_{\alpha} = \langle \varpi_{\alpha}, \beta^{\vee}  \rangle, \ \forall \alpha \in \Delta \backslash I.
\end{equation}
\end{lemma}

A proof for the above result can be found in \cite{correa2023deformed}, see also \cite{FultonWoodward} and \cite{AZAD}. From the above lemma and Theorem \ref{AZADBISWAS}, we obtain the following fundamental result.

\begin{proposition}[\cite{correa2023deformed}]
\label{C8S8.2Sub8.2.3P8.2.6} 
Let $X_{P}$ be a complex flag variety associated with some parabolic Lie subgroup $P = P_{I}$. Then, we have
\begin{equation}
\label{picardeq}
{\text{Pic}}(X_{P}) = H^{1,1}(X_{P},\mathbbm{Z}) = H^{2}(X_{P},\mathbbm{Z}) = \displaystyle \bigoplus_{\alpha \in \Delta \backslash I}\mathbbm{Z}[{\bf{\Omega}}_{\alpha} ].
\end{equation}
\end{proposition}
From the above ideas we make the following remarks.
\begin{remark}[Harmonic 2-forms on $X_{P}$]Given any $G$-invariant Riemannian metric $g$ on $X_{P}$, let us denote by $\mathscr{H}^{2}(X_{P},g)$ the space of real harmonic 2-forms on $X_{P}$ with respect to $g$, and let us denote by $\mathscr{I}_{G}^{1,1}(X_{P})$ the space of closed invariant real $(1,1)$-forms. Combining the result of Proposition \ref{C8S8.2Sub8.2.3P8.2.6} with \cite[Lemma 3.1]{MR528871}, we obtain 
\begin{equation}
\mathscr{I}_{G}^{1,1}(X_{P}) = \mathscr{H}^{2}(X_{P},g). 
\end{equation}
Therefore, the closed $G$-invariant real $(1,1)$-forms described in Theorem \ref{AZADBISWAS} are harmonic with respect to any $G$-invariant Riemannian metric on $X_{P}$.
\end{remark}

\begin{remark}[K\"{a}hler cone of $X_{P}$]
\label{Kcone}
It follows from Eq. (\ref{picardeq}) and Theorem \ref{AZADBISWAS} that the K\"{a}hler cone (e.g. \cite{Lazarsfeld}) of a complex flag variety $X_{P}$ is given explicitly by
\begin{equation}
\mathcal{K}(X_{P}) = \displaystyle \bigoplus_{\alpha \in \Delta \backslash I} \mathbbm{R}^{+}[ {\bf{\Omega}}_{\alpha}].
\end{equation}
\end{remark}

\begin{remark}[Mori Cone] It is worth observing that the cone of curves, e.g. \cite{Lazarsfeld}, or Mori cone,
\begin{equation}
{\rm{NE}}(X_{P}) := \Big \{ \sum_{i}a_{i}[C_{i}] \ \Big | \ C_{i} \subset X_{P} \ \text{is irreducible}, \ a_{i} \geq 0 \Big \}, 
\end{equation}
of a flag variety $X_{P}$ is generated by the rational curves $[\mathbbm{P}_{\alpha}^{1}] \in \pi_{2}(X_{P})$, $\alpha \in \Delta \backslash I$, see for instance \cite[\S 18.3]{Timashev} and references therein. In what follows, we shall consider the Mori semigroup 
\begin{equation}
{\rm{NE}}(X_{P})_{\mathbbm{Z}} := \Big \{ \sum_{i}a_{i}[C_{i}] \in {\rm{NE}}(X_{P})  \ \Big | \ a_{i} \in \mathbbm{Z}  \Big \}.
\end{equation}
\end{remark}


The previous result shows that the isomorphisms 
\begin{center}
${\rm{Pic}}(X_{P}) \cong {\rm{Hom}}(P,\mathbbm{C}^{\times}) \cong \bigoplus_{\alpha \in \Delta \backslash I} \mathbbm{Z} \varpi_{\alpha} = \Lambda_{P}$,
\end{center}
can be described as
\begin{enumerate}
\item $ \displaystyle {\bf{L}} \mapsto \vartheta_{{\bf{L}}}: = \prod_{\alpha \in \Delta \backslash I} \vartheta_{\varpi_{\alpha}}^{\langle c_{1}({\bf{L}}),[\mathbbm{P}^{1}_{\alpha}] \rangle} \mapsto \lambda({\bf{L}}) := \sum_{\alpha \in \Delta \backslash I}\langle c_{1}({\bf{L}}),[\mathbbm{P}^{1}_{\alpha}] \rangle\varpi_{\alpha}$,
\item$ \displaystyle \lambda = \sum_{\alpha \in \Delta \backslash I}k_{\alpha}\varpi_{\alpha} \mapsto \vartheta_{\lambda} = \prod_{\alpha \in \Delta \backslash I} \vartheta_{\varpi_{\alpha}}^{k_{\alpha}} \mapsto \bigotimes_{\alpha \in \Delta \backslash I} \mathscr{O}_{\alpha}(1)^{\otimes k_{\alpha}} = {\bf{E}}_{\vartheta_{\lambda}^{-1}}$.
\end{enumerate}
Thus, $\forall {\bf{L}} \in {\rm{Pic}}(X_{P})$, we have an integral weight $\lambda({\bf{L}}) \in \Lambda_{P}$. In general, $\forall \xi \in H^{1,1}(X_{P},\mathbbm{R})$, we can attach $\lambda (\xi) \in \Lambda_{P} \otimes \mathbbm{R}$, such that
\begin{equation}
\label{weightcohomology}
\lambda(\xi) := \sum_{\alpha \in \Delta \backslash I}\langle \xi,[\mathbbm{P}^{1}_{\alpha}] \rangle\varpi_{\alpha}.
\end{equation}
From above, for every holomorphic vector bundle ${\bf{E}} \to X_{P}$, we can define $\lambda({\bf{E}}) \in \Lambda_{P}$, such that 
\begin{equation}
\label{weightholomorphicvec}
\lambda({\bf{E}}) := \sum_{\alpha \in \Delta \backslash I} \langle c_{1}({\bf{E}}),[\mathbbm{P}_{\alpha}^{1}] \rangle \varpi_{\alpha},
\end{equation}
where $c_{1}({\bf{E}}) = c_{1}(\bigwedge^{r}{\bf{E}})$, such that $r = \rank({\bf{E}})$.

\section{Proof of Theorem A}

In order to prove Theorem \ref{theoremA}, we shall stablish some preliminary results. In what follows, we keep the notation introduced in the previous sections.

\begin{lemma}
\label{firsthomovec}
Let ${\bf{E}} \to  G^{\mathbbm{C}}/P$ be a homogeneous vector bundle defined by some finite-dimensional representation $\theta \colon P \to {\rm{GL}}(W)$. Then, we have
\begin{equation}
\label{Chernclassform}
\lambda({\bf{E}}) = -\sum_{\alpha \in \Delta \backslash I}\langle \det \circ \theta ,\alpha^{\vee} \rangle \varpi_{\alpha}.
\end{equation} 
\end{lemma}
\begin{proof}
If $r = \dim(W)$, it follows that 
\begin{equation}
\Lambda^{r}({\bf{E}}) = G^{\mathbbm{C}} \times_{P}(\wedge^{r}W) = \Big \{(U_{i})_{i \in J}, (\det \circ \theta) \circ \psi_{i j} \colon U_{i} \cap U_{j} \to \mathbbm{C}^{\times} \Big \}.
\end{equation}
From Eq. (\ref{cocycle}) we see that the line bundle on the right-hand side above is defined by the character 
\begin{equation}
\label{determinantbundle}
\chi = (\det \circ \theta)^{-1} \in {\text{Hom}}(P,\mathbbm{C}^{\times}).
\end{equation}
From Eq. (\ref{characterparabolic}), we obtain that $\Lambda^{r}({\bf{E}}) = {\bf{E}}_{\chi^{-1}}$, thus
\begin{equation}
\lambda({\bf{E}}) = \sum_{\alpha \in \Delta \backslash I}\langle \chi ,\alpha^{\vee} \rangle \varpi_{\alpha} = -\sum_{\alpha \in \Delta \backslash I}\langle \det \circ \theta ,\alpha^{\vee} \rangle \varpi_{\alpha},
\end{equation}
which concludes the proof.
\end{proof}

\begin{example}[Canonical Bundle]
Let $X_{P} = G^{\mathbbm{C}}/P$ be a complex flag variety defined by some parabolic Lie subgroup $P = P_{I} \subset G^{\mathbbm{C}}$. Considering the identification $T_{{\rm{o}}}^{1,0}X_{P}  \cong \mathfrak{u}_{P}^{-}\subset \mathfrak{g}^{\mathbbm{C}}$, such that ${\rm{o}} = eP \in X_{P}$ and 
\begin{center}
$\mathfrak{u}_{P}^{-} := \displaystyle \sum_{\alpha \in \Phi_{I}^{-}} \mathfrak{g}_{\alpha}$,
\end{center}
 see Eq. (\ref{isoholomorphictgbundle}), we can realize $T^{1,0}X_{P}$ as being a holomoprphic vector bundle, associated with the holomorphic principal $P$-bundle $P \hookrightarrow G^{\mathbbm{C}} \to X$, such that 

\begin{center}

$T^{1,0}X_{P} = \Big \{(U_{i})_{i \in J}, \underline{{\rm{Ad}}}\circ \psi_{i j} \colon U_{i} \cap U_{j} \to {\rm{GL}}(\mathfrak{u}_{P}^{-}) \Big \}$,

\end{center}
where $\underline{{\rm{Ad}}} \colon P \to {\rm{GL}}(\mathfrak{u}_{P}^{-})$ is the isotropy representation. From this, we obtain 
\begin{equation}
\label{canonicalbundleflag}
{\bf{K}}_{X}^{-1} = \Lambda^{\text{top}} \big(T^{1,0}X_{P} \big) = \Big \{(U_{i})_{i \in J}, \det (\underline{{\rm{Ad}}}\circ \psi_{i j}) \colon U_{i} \cap U_{j} \to \mathbbm{C}^{\times} \Big \}.
\end{equation}
Since the character $\det \circ \underline{{\rm{Ad}}} \in {\text{Hom}}(P,\mathbbm{C}^{\times})$ is completely determined by its restriction to the torus $Z(L_{P})^{0} = \exp(\mathfrak{c}_{P}) \subset P$, see Eq. (\ref{commutatordec}), observing that 
\begin{equation}
\det \underline{{\rm{Ad}}}(\exp({\bf{t}})) = {\rm{e}}^{{\rm{tr}}({\rm{ad}}({\bf{t}})|_{\mathfrak{u}_{P}^{-}})} = {\rm{e}}^{- \langle \delta_{P},{\bf{t}}\rangle },
\end{equation}
$\forall {\bf{t}} \in \mathfrak{c}_{P}$, such that $\delta_{P} = \sum_{\alpha \in \Phi_{I}^{+} } \alpha$. Now we notice that
\begin{equation}
\langle \delta_{P},\alpha^{\vee} \rangle = 0, \ \ \ \forall \alpha \in I.
\end{equation}
In fact, considering the representation ${\rm{ad}}(\cdot)|_{\mathfrak{u}_{P}^{+}} \colon \mathfrak{l}_{P} \to \mathfrak{gl}(\mathfrak{u}_{P}^{+})$, for every $\alpha \in I$, we have
\begin{equation}
\langle \delta_{P},\alpha^{\vee} \rangle = \delta_{P}(h_{\alpha}) = {\rm{Tr}}({\rm{ad}}(h_{\alpha})|_{\mathfrak{u}^{+}_{P}}) = 0,
\end{equation}
since $h_{\alpha}  \in \mathfrak{s}_{P} = [\mathfrak{l}_{P},\mathfrak{l}_{P}]$, see Eq. (\ref{generatorvee}) and Eq. (\ref{paringkilling}). Hence, denoting $\vartheta_{\delta_{P}}^{-1} = \det \circ \underline{{\rm{Ad}}}$, it follows that 
\begin{center}
${\bf{K}}_{X_{P}}^{-1} = {\bf{E}}_{\vartheta_{\delta_{P}}^{-1}}$. 
\end{center}
From this, we conclude that 
\begin{equation} 
\label{charactercanonical}
\lambda(T^{1,0}X_{P}) = -\sum_{\alpha \in \Delta \backslash I} \langle \det \circ \underline{{\rm{Ad}}}, \alpha^{\vee} \rangle \varpi_{\alpha} = \sum_{\alpha \in \Delta \backslash I} \langle \vartheta_{\delta_{P}}, \alpha^{\vee} \rangle \varpi_{\alpha} = \delta_{P}
\end{equation}
If we consider the invariant K\"{a}hler metric $\rho_{0} \in \Omega^{1,1}(X_{P})^{G}$ defined by
\begin{equation}
\label{riccinorm}
\rho_{0} = \sum_{\alpha \in \Delta \backslash I}2 \pi \langle \delta_{P}, \alpha^{\vee} \rangle {\bf{\Omega}}_{\alpha} \myeq \sum_{\alpha \in \Delta \backslash I} \langle \delta_{P}, \alpha^{\vee} \rangle \sqrt{-1}\partial \overline{\partial}\log \Big ( \big | \big | \Pi_{\varpi_{\alpha}}({\bf{s}}_{U})v_{\varpi_{\alpha}}^{+}\big | \big |^{2} \Big),
\end{equation}
for some local section ${\bf{s}}_{U} \colon U \subset X_{P}\to G^{\mathbbm{C}}$, it follows that
\begin{equation}
\label{ChernFlag}
c_{1}(X_{P}) = \Big [ \frac{\rho_{0}}{2\pi}\Big].
\end{equation}
By the uniqueness of $G$-invariant representative of $c_{1}(X_{P})$, we have ${\rm{Ric}}(\rho_{0}) = \rho_{0}$, in other words, $\rho_{0} \in \Omega^{1,1}(X_{P})^{G}$ defines a $G$-ivariant K\"{a}hler-Einstein metric on $X_{P}$ (cf. \cite{MATSUSHIMA}). 
\end{example}

Consider now the following lemma.
\begin{lemma}
Let $P = P_{I} \subset G^{\mathbbm{C}}$ be a parabolic Lie subgroup. Then, we have 
\begin{equation}
\mathfrak{c}_{P}^{\ast} = \Big \{ \sum_{\alpha \in \Delta \backslash I}c_{\alpha}(\varpi_{\alpha}|_{\mathfrak{c}_{P}}) \ \Big | \ c_{\alpha} \in \mathbbm{C}, \ \forall \alpha \in \Delta \backslash I\Big \}.
\end{equation}
\end{lemma}
\begin{proof}
Let us verify that $\varpi_{\alpha}|_{\mathfrak{c}_{P}}$, $\alpha \in \Delta \backslash I$, defines a basis for $\mathfrak{c}_{P}^{\ast}$. Suppose that 
\begin{equation}
\sum_{\alpha \in \Delta \backslash I}c_{\alpha}(\varpi_{\alpha}|_{\mathfrak{c}_{P}}) = 0.
\end{equation}
Since $\mathfrak{c}_{P} = \bigcap_{\alpha \in I} \ker(\alpha)$, we obtain 
\begin{equation}
\sum_{\alpha \in \Delta \backslash I}c_{\alpha}\varpi_{\alpha} \in \mathfrak{c}_{P}^{0} = \langle I \rangle_{\mathbbm{C}} \iff \sum_{\alpha \in \Delta \backslash I}c_{\alpha}\varpi_{\alpha} = \sum_{\alpha \in I}a_{\alpha} \alpha,
\end{equation}
for some $a_{\alpha} \in \mathbbm{C}$, $\alpha \in \Delta \backslash I$. From this, it follows that 
\begin{equation}
\begin{split}
\sum_{\alpha \in \Delta \backslash I}c_{\alpha}\varpi_{\alpha} &= \sum_{\alpha \in I}a_{\alpha} \Big (\sum_{\beta \in \Delta}C_{\alpha \beta}\varpi_{\beta} \Big ) = \sum_{\beta \in \Delta}\Big ( \sum_{\alpha \in I}a_{\alpha}C_{\alpha \beta} \Big )\varpi_{\beta} \\
&= \sum_{\beta \in \Delta \backslash I}\Big ( \sum_{\alpha \in I}a_{\alpha}C_{\alpha \beta} \Big )\varpi_{\beta} + \sum_{\beta \in I}\Big ( \sum_{\alpha \in I}a_{\alpha}C_{\alpha \beta} \Big )\varpi_{\beta}.
\end{split}
\end{equation}
Observing that $A_{I} = (C_{\alpha \beta})_{\alpha,\beta \in I}$ is the Cartan matrix of the semisimples Lie algebra $\mathfrak{s}_{P} = [\mathfrak{l}_{P},\mathfrak{l}_{P}]$, we conclude that the unique solution of the equation
\begin{equation}
\sum_{\beta \in I}\Big ( \sum_{\alpha \in I}a_{\alpha}C_{\alpha \beta} \Big )\varpi_{\beta} = 0,
\end{equation}
is the trivial solution $a_{\alpha} = 0$, for every $\alpha \in I$. Hence, we conclude that $\sum_{\alpha \in \Delta \backslash I}c_{\alpha}\varpi_{\alpha} = 0$. Since the set $\{\varpi_{\alpha}\}_{\alpha \in \Delta \backslash I}$ is an algebraically independent set, we conclude that $c_{\alpha} = 0$, for every $\alpha \in \Delta \backslash I$. From this, we conclude that $\{\varpi_{\alpha}|_{\mathfrak{c}_{P}}\}_{\alpha \in \Delta \backslash I}$ is an algebraically independent set in $\mathfrak{c}_{P}^{\ast}$. Since $\dim_{\mathbbm{C}}(\mathfrak{c}_{P}^{\ast}) = |\Delta \backslash I|$, we conclude the proof.
\end{proof}

\begin{lemma}
\label{restrictionweightbundle}
Let ${\bf{E}}_{\theta}\to  X_{P}$ be a homogeneous vector bundle defined by some finite-dimensional irreducible representation $\theta \colon P \to {\rm{GL}}(W(\lambda))$. Then, we have
\begin{equation}
\label{weightvector}
\lambda({\bf{E}}_{\theta})|_{\mathfrak{c}_{P}} = - \dim(W(\lambda)) \lambda|_{\mathfrak{c}_{P}}
\end{equation}
\end{lemma}
\begin{proof}
Given a finite dimensional irreducible representation $\theta \colon P \to {\rm{GL}}(W(\lambda))$, we have ${\bf{E}}_{\theta} = G^{\mathbbm{C}} \times_{P} W(\lambda)$. Considering $r = \dim_{\mathbbm{C}}(W(\lambda))$, we obtain
\begin{equation}
\wedge^{r}({\bf{E}}_{\theta}) = G^{\mathbbm{C}} \times_{P} \wedge^{r}W(\lambda).
\end{equation}
From above, it follows that 
\begin{equation}
\wedge^{r}({\bf{E}}_{\theta}) = \Big \{(U_{i})_{i \in J}, \det \circ (\theta \circ \psi_{i j}) \colon U_{i} \cap U_{j} \to \mathbbm{C}^{\times} \Big \}.
\end{equation}
Observing that $\det \circ \theta$ is completely determined on $Z(L_{P})^{0} = \exp(\mathfrak{c}_{P})$, it follows that
\begin{equation}
\det(\theta(\exp({\bf{t}}))) = \exp({\rm{Tr}}(({\rm{d}}\theta)_{e}({\bf{t}}))),
\end{equation}
for every $ {\bf{t}} \in \mathfrak{c}_{P}$. Since 
\begin{equation}
({\rm{d}}\theta)_{e}({\bf{t}}) = \lambda({\bf{t}}){\rm{Id}}_{W(\lambda)}, \ \ \ \forall {\bf{t}} \in \mathfrak{c}_{P},
\end{equation}
it follows that  
\begin{equation}
r \lambda |_{\mathfrak{c}_{P}} = {\rm{d}}(\det \circ \theta_{s})_{e}|_{\mathfrak{c}_{P}}  = \sum_{\alpha \in \Delta \backslash I}\langle \det \circ \theta,\alpha^{\vee} \rangle (\varpi_{\alpha}|_{\mathfrak{c}_{P}}) = - \lambda({\bf{E}})|_{\mathfrak{c}_{P}}.
\end{equation}
\end{proof}

\begin{lemma}
\label{firstcherclassformula}
Let ${\bf{E}}\to  X_{P}$ be a homogeneous vector bundle defined by some finite-dimensional irreducible representation $\theta \colon P \to {\rm{GL}}(W(\lambda))$. Then, considering $\lambda = \lambda_{s} + \lambda_{c}$, such that $\lambda_{s} \in \Lambda_{\mathfrak{s}_{P}}$ and $\lambda_{c} \in \Lambda_{P}$, we have
\begin{equation}
\lambda({\bf{E}}) = \sum_{\beta \in \Delta \backslash I} \Big ( \sum_{\alpha \in I}a_{\alpha}({\bf{E}})\langle \alpha,\beta^{\vee} \rangle\Big )\varpi_{\beta} - \dim(W(\lambda)) \lambda_{c},
\end{equation}
such that
\begin{equation}
a_{\alpha}({\bf{E}}) := \dim(W(\lambda))\frac{\det(C_{I}(\lambda_{s},\alpha))}{\det(C_{I})}, \ \ \ \alpha \in I,
\end{equation}
where $C_{I} = (\langle \alpha,\beta^{\vee} \rangle)_{\alpha,\beta \in I}$ and $C_{I}(\lambda_{s},\alpha)$, $\alpha \in I$, is the matrix obtained from the Cartan matrix $C_{I}$ replacing its $\alpha$-row by the row vector $(\langle \lambda_{s},\beta^{\vee} \rangle)_{\beta \in I}$.
\end{lemma}

\begin{proof}
Given a finite dimensional irreducible representation $\theta \colon P \to {\rm{GL}}(W(\lambda))$, such that $r = \dim(W(\lambda))$, from the previous lemma, we have $r\lambda + \lambda({\bf{E}}_{\theta}) \in \mathfrak{c}_{P}^{0} = \langle I \rangle_{\mathbbm{C}}$. Thus, we have 
\begin{equation}
r\lambda + \lambda({\bf{E}}_{\theta}) = \sum_{\alpha \in I}a_{\alpha}({\bf{E}}_{\theta}) \alpha = \sum_{\beta \in \Delta} \Big ( \sum_{\alpha \in I}a_{\alpha}({\bf{E}}_{\theta})C_{\alpha \beta}\Big )\varpi_{\beta}.
\end{equation}
From above we obtain the following equations on the variables $(a_{\alpha}({\bf{E}}_{\theta}))_{\alpha \in I}$:
\begin{enumerate}
\item $\sum_{\beta \in I} \langle r\lambda,\beta^{\vee} \rangle \varpi_{\beta} = \sum_{\beta \in I} \Big ( \sum_{\alpha \in I}a_{\alpha}({\bf{E}}_{\theta})C_{\alpha \beta}\Big )\varpi_{\beta}$,
\item $\sum_{\beta \in \Delta \backslash I} \langle r\lambda,\beta^{\vee} \rangle \varpi_{\beta} + \lambda({\bf{E}}_{\theta}) = \sum_{\beta \in \Delta \backslash I} \Big ( \sum_{\alpha \in I}a_{\alpha}({\bf{E}}_{\theta})C_{\alpha \beta}\Big )\varpi_{\beta}$.
\end{enumerate}
Once we know the weight $\lambda$, the linear system (1) determines the coefficients $a_{\alpha}({\bf{E}}_{\theta})$ in terms of $\langle r\lambda,\beta^{\vee} \rangle$, $\beta \in I$, and the Cartan matrix $C_{I} = (C_{\alpha \beta})_{\alpha,\beta \in I}$, $C_{\alpha \beta} = \langle \alpha,\beta^{\vee} \rangle, $of the semissimple algebra $\mathfrak{s}_{P}$. In order to describe explicitly the solution, we denote $\Delta = \{\alpha_{1},\ldots,\alpha_{n}\}$ and consider $P = P_{I} \subset G^{\mathbbm{C}}$, such that $I = \{\alpha_{i_{1}},\ldots,\alpha_{i_{k}}\}$. In this case, the linear system can be written as 
\begin{equation}
\underbrace{\begin{pmatrix}  C_{\alpha_{i_{1}} \alpha_{i_{1}}} & \cdots & C_{\alpha_{i_{k}}\alpha_{i_{1}}} \\ 
 \vdots & \ddots & \vdots \\
 C_{\alpha_{i_{1}} \alpha_{i_{k}}} & \cdots & C_{\alpha_{i_{k}}\alpha_{i_{k}}} 
 \end{pmatrix}}_{C_{I}^{T}} \begin{pmatrix} a_{\alpha_{i_{1}}}({\bf{E}}_{\theta}) \\ \vdots \\ a_{\alpha_{i_{k}}}({\bf{E}}_{\theta})\end{pmatrix} = \begin{pmatrix} \langle r \lambda_{s},\alpha_{i_{1}}^{\vee} \rangle \\ \vdots \\ \langle r \lambda_{s},\alpha_{i_{1}}^{\vee} \rangle \end{pmatrix}.
\end{equation}
Therefore, from Cramer's rule, we conclude that 
\begin{equation}
a_{\alpha}({\bf{E}}_{\theta}) = \dim(W(\lambda))\frac{\det(C_{I}(\lambda_{s},\alpha))}{\det(C_{I})}, \ \ \ \alpha \in I,
\end{equation}
where $C_{I} = (C_{\alpha \beta})_{\alpha,\beta \in I}$ and $C_{I}(\lambda_{s},\alpha)$ is the matrix obtained from $C_{I}$ replacing its $\alpha$-line by the line vector $(\langle \lambda_{s},\beta^{\vee} \rangle)_{\beta \in I}$.
\end{proof}

\begin{remark}
\label{matricessystemchern}
In the setting of the proof of the last theorem, we shall denote 
\begin{equation}
{\bf{a}}({\bf{E}}) = \begin{pmatrix} a_{\alpha_{i_{1}}}({\bf{E}}) \\ \vdots \\ a_{\alpha_{i_{k}}}({\bf{E}})\end{pmatrix} \ \ \ \ \text{and} \ \ \ \ {\bf{b}}_{I}({\bf{E}}) = \begin{pmatrix} \langle r \lambda_{s},\alpha_{i_{1}}^{\vee} \rangle \\ \vdots \\ \langle r \lambda_{s},\alpha_{i_{k}}^{\vee} \rangle \end{pmatrix}.
\end{equation}
Hence, the linear system involved in the proof can be written as $C_{I}^{T}{\bf{a}}({\bf{E}}) = {\bf{b}}_{I}({\bf{E}})$.
\end{remark}

Let us provide some detailed examples which illustrate the consistence of the formula described in Eq. (\ref{Chernclassform}).

\begin{example}[{\bf{Universal Bundle}}]
\label{Klein_universal}
Given $G^{\mathbbm{C}} = {\rm{SL}}_{4}(\mathbb{C})$, fixing the Cartan subalgebra $\mathfrak{h} \subset \mathfrak{sl}_{4}(\mathbbm{C})$ provided by the diagonal matrices, we consider the simple root system 
\begin{equation}
\Delta = \{\alpha_{i} = \epsilon_{i} - \epsilon_{i+1} \ | \ 1 \leq i \leq 3\}, 
\end{equation}
such that $\epsilon_{i}({\rm{diag}}(x_{1},\ldots,x_{4})) = x_{i}$, for every $i = 1,\ldots,4$. 
\begin{equation}
{\dynkin[labels={\alpha_{1},\alpha_{2},\alpha_{3}},scale=3]A{ooo}} 
\end{equation}
The associated Cartan matrix $C = (C_{ij}) $ of $\mathfrak{sl}_{4}(\mathbbm{C})$ is given by
\begin{equation}
C = \begin{pmatrix} 2 & -1 & 0 \\
-1 & 2 & -1 \\
0 & -1 & 2\end{pmatrix}, \ \ C_{ij} = \langle \alpha_{i}, \alpha_{j}^{\vee} \rangle, \ \ i,j = 1,2,3,
\end{equation}
By considering $I = \{\alpha_{1}\alpha_{3}\} \subset \Delta$, that is 
\begin{equation}
{\dynkin[labels={\alpha_{1},\alpha_{2},\alpha_{3}},scale=3]A{o*o}} 
\end{equation}
we have the flag variety (Klein quadric)
\begin{equation}
X = {\rm{SL}}_{4}(\mathbbm{C})/P_{I} = {\rm{Gr}}_{2}(\mathbbm{C}^{4}).
\end{equation}
In this case the Levi component of $P = P_{I}$ is given by 
\begin{equation}
L_{P} = \bigg \{ \begin{pmatrix} A & 0 \\
0 & D\end{pmatrix} \ \bigg | \ A,D \in {\rm{GL}}_{2}(\mathbbm{C}) \bigg \} \cap {\rm{SL}}_{4}(\mathbbm{C}).
\end{equation}
Thus, we have the Cartan matrix $C_{I}$ for $\mathfrak{s}_{P} = \mathfrak{sl}_2 \oplus \mathfrak{sl}_2$ is given by
\begin{equation}
C_I = \begin{pmatrix} 
\langle \alpha_1, \alpha_1^\vee \rangle & \langle \alpha_3, \alpha_1^\vee \rangle \\ 
\langle \alpha_1, \alpha_3^\vee \rangle & \langle \alpha_3, \alpha_3^\vee \rangle 
\end{pmatrix} = \begin{pmatrix} 2 & 0 \\ 0 & 2 \end{pmatrix}.
\end{equation}
Note that $C_I^T = C_I$ due to the symmetry of type $A_n$ Lie algebra.

From above, considering the irreducible representation $\theta\colon L_{P} \to {\rm{GL}}(\mathbbm{C}^{2})$, such that 
\begin{equation}
\theta\begin{pmatrix} A & 0 \\
0 & D\end{pmatrix} := A, \ \ \ \forall \begin{pmatrix} A & 0 \\
0 & D\end{pmatrix} \in L_{P},
\end{equation}
we obtain an irreducible representation $\theta \colon P  \to {\rm{GL}}(\mathbbm{C}^{2})$ which extends trivially $\theta$. This irreducible representation above corresponds to the weight $\varpi_{\alpha_{1}}$, that is, $W(\varpi_{\alpha_{1}}) = \mathbbm{C}^{2}$. Now we can form the homogeneous vector bundle
\begin{equation}
\mathcal{S}:= {\rm{SL}}_{4}(\mathbbm{C})\times_{P}W(\varpi_{\alpha_{1}}) = {\rm{SL}}_{4}(\mathbbm{C})\times_{P}\mathbbm{C}^{2}.
\end{equation}
From Lemma \ref{firstcherclassformula}, we have
\begin{equation}
\lambda(\mathcal{S}) = \big( \sum_{\alpha \in I} a_{\alpha}(\mathcal{S}) C_{\alpha, \alpha_2} \big) \varpi_{\alpha_{2}} - 2\lambda_c.
\end{equation}
Since in this case $\lambda = \lambda_{s} + \lambda_{c}$, with $\lambda_{s} = \varpi_{\alpha_{1}}$ and $\lambda_{c} = 0$, we conclude that 
\begin{equation}
\lambda(\mathcal{S}) = \big( \sum_{\alpha \in I} a_{\alpha}(\mathcal{S}) C_{\alpha, \alpha_2} \big) \varpi_{\alpha_{2}}.
\end{equation}
Following the proof of Lemma \ref{firstcherclassformula} and Remark \ref{matricessystemchern}, we solve the linear system $C_I^T {\bf{a}}(\mathcal{S}) = {\bf{b}}_{I}(\mathcal{S})$. Since $r = 2$ and $\lambda_{s} = \varpi_{\alpha_1}$, we have 
\begin{equation}
    \begin{cases}\langle 2\varpi_{\alpha_{1}}, \alpha_{1}^\vee \rangle = 2(1) = 2 \\
     \langle 2\varpi_{\alpha_{1}}, \alpha_{3}^\vee \rangle = 2(0) = 0\end{cases} \Longrightarrow {\bf{b}}_{I}(\mathcal{S}) = \begin{pmatrix} 2 \\ 0\end{pmatrix}
\end{equation}
The system is given
\begin{equation}
\begin{pmatrix} 2 & 0 \\ 0 & 2 \end{pmatrix} \begin{pmatrix} a_{\alpha_{1}}(\mathcal{S}) \\ a_{\alpha_{3}}(\mathcal{S}) \end{pmatrix} = \begin{pmatrix} 2 \\ 0 \end{pmatrix}.
\end{equation}
The solution in this case is $a_{\alpha_{1}}(\mathcal{S}) = 1$ and $a_{\alpha_{2}}(\mathcal{S}) = 0$. Notice that 
\begin{equation}
a_{\alpha_{1}}(\mathcal{S})= \frac{\det \begin{pmatrix} 2 & 0 \\ 0 & 2 \end{pmatrix}}{4} = \frac{4}{4} = 1, \ \ \ \  a_{\alpha_{2}}(\mathcal{S}) = \frac{\det \begin{pmatrix} 2 & 2 \\ 0 & 0 \end{pmatrix}}{4} = 0.
\end{equation}
Observing that $C_{\alpha_1, \alpha_2} = -1$ and $C_{\alpha_3, \alpha_2} = -1$, we conclude that 
\begin{equation}
\lambda(\mathcal{S}) = \left( 1 \cdot (-1) + 0 \cdot (-1) \right) \varpi_2 - 0 = -\varpi_2.
\end{equation}
We observe that the vector bundle $\mathcal{S}$ is the universal bundle which fits in the tautological exact sequence 
\begin{equation}
\begin{tikzcd}
0 \arrow[r] & \mathcal{S} \arrow[r]& \mathcal{O}_{X}^{\oplus 4} \arrow[r] & \mathcal{Q} \arrow[r] & 0.
\end{tikzcd} 
\end{equation}
\end{example}

\begin{example}[{\bf{Spinor Bundle}}]
\label{spinorbunlde}
Let $G^{\mathbb{C}} = {\rm{Spin}}(7,\mathbb{C})$ (of type $B_{3}$) and consider the identification 
\begin{equation}
\mathfrak{spin}(7,\mathbbm{C}) \cong \{X \in {\rm{M}}_{7\times 7}(\mathbbm{C}) \ | \ X^{T}J + JX = 0\}, 
\end{equation}
such that
\begin{equation}
J = \begin{pmatrix} 
0 & 0 & 0 & 0 & 0 & 0 & 1 \\ 
0 & 0 & 0 & 0 & 0 & 1 & 0 \\ 
0 & 0 & 0 & 0 & 1 & 0 & 0 \\ 
0 & 0 & 0 & 1 & 0 & 0 & 0 \\ 
0 & 0 & 1 & 0 & 0 & 0 & 0 \\ 
0 & 1 & 0 & 0 & 0 & 0 & 0 \\ 
1 & 0 & 0 & 0 & 0 & 0 & 0 
\end{pmatrix}.
\end{equation}
Fixing the Cartan subalgebra $\mathfrak{h} \subset\mathfrak{spin}(7,\mathbbm{C})$ defined by the diagonal matrices, under the above identification, we have the simple root system 
\begin{equation}
\Delta = \{\alpha_{1} = \epsilon_{1} - \epsilon_{2}, \alpha_{2} = \epsilon_{2} - \epsilon_{3}, \alpha_{3} = \epsilon_{3}\}
\end{equation}
such that $\epsilon_{i}({\rm{diag}}(x_{1},\ldots,x_{7})) = x_{i}$, for every $i = 1,2,3$, corresponding to the Dynkin diagram
\begin{equation}
{\dynkin[labels={\alpha_{1},\alpha_{2},\alpha_{3}},scale=3]B{ooo}} 
\end{equation}
The associated Cartan matrix $C = (C_{ij}) $ of $\subset\mathfrak{spin}(7,\mathbbm{C})$ is given by 
\begin{equation}
\label{CartanB3}
C = \begin{pmatrix}
 2 & -1 &  0 \\
-1 &  2 & -2 \\
 0 & -1 &  2
\end{pmatrix}, \ \ C_{ij} = \langle \alpha_{i}, \alpha_{j}^{\vee} \rangle, \ \ i,j = 1,2,3.
\end{equation}
By taking $I = \{\alpha_{2}, \alpha_{3}\} \subset \Delta$, that is,
\begin{equation}
{\dynkin[labels={\alpha_{1},\alpha_{2},\alpha_{3}},scale=3]B{*oo}} 
\end{equation}
we define the parabolic subgroup $P=P_{I}$, giving rise to the flag variety ($5$-dimensional quadric)
\begin{equation}
X_{P} = {\rm{Spin}}(7,\mathbb{C})/P_{I} = Q_{5}. 
\end{equation}
In this case, the Levi component $L_{P}$ has a semisimple part $S_{P}$ of type $B_{2}$, that is $\mathfrak{s}_{P} = \mathfrak{spin}(5,\mathbb{C})$.  The Cartan matrix $C_{I} = (\langle \alpha, \beta^{\vee} \rangle)_{\alpha, \beta \in I}$ of $\mathfrak{s}_{P}$ in this case is given by
\begin{equation}
C_{I} = \begin{pmatrix} 2 & -2 \\ -1 & 2 \end{pmatrix}.
\end{equation}
Now let ${\bf{\Sigma}} \to Q_{5}$ be the homogeneous spinor bundle defined by the irreducible representation of $P_{I}$ with highest weight $\lambda = \varpi_{\alpha_{3}} \in \Lambda_{\mathfrak{s}_{P}}$, i.e.,
\begin{equation}
{\bf{\Sigma}} = {\rm{Spin}}(7,\mathbb{C}) \times_{P} W(\varpi_{\alpha_{3}}).
\end{equation}
Here, since ${\rm{Spin}}(5,\mathbbm{C}) \cong {\rm{Sp}}(4,\mathbbm{C})$, considering $L_{P} \cong {\rm{Sp}}(4,\mathbbm{C}) \times \mathbbm{C}^{\times}$, we have $W(\varpi_{\alpha_{3}}) \cong \mathbbm{C}^{4}$ and ${\bf{\Sigma}}$ is defined by the irreducible representation $\theta \colon L_{P} \to {\rm{GL}}(\mathbbm{C}^{4})$, such that 
\begin{equation}
\theta(A,z)v = zAv, \ \ \forall v \in \mathbbm{C}^{4}, \ \ \forall (A,z) \in L_{P} \cong {\rm{Sp}}(4,\mathbbm{C}) \times \mathbbm{C}^{\times}.
\end{equation}
From Lemma \ref{firstcherclassformula}, we have
\begin{equation}
\lambda({\bf{\Sigma}}) = \big( \sum_{\alpha \in I} a_{\alpha}({\bf{\Sigma}}) C_{\alpha, \alpha_{1}} \big) \varpi_{\alpha_{1}} - 2\lambda_c.
\end{equation}
Since in this case $\lambda = \lambda_{s} + \lambda_{c}$, with $\lambda_{s} = \varpi_{\alpha_{3}}$ and $\lambda_{c} = 0$, we conclude that 
\begin{equation}
\lambda({\bf{\Sigma}}) = \big( \sum_{\alpha \in I} a_{\alpha}({\bf{\Sigma}}) C_{\alpha, \alpha_{1}} \big) \varpi_{\alpha_{1}}.
\end{equation}
Following the proof of Lemma \ref{firstcherclassformula} and Remark \ref{matricessystemchern}, we solve the linear system $C_{I}^{T} {\bf{a}}({\bf{\Sigma}}) = {\bf{b}}_{I}({\bf{\Sigma}})$. Since $r=4$ and $\lambda_{s} = \varpi_{\alpha_{3}}$, we have 
\begin{equation}
    \begin{cases}\langle 4\varpi_{\alpha_{3}}, \alpha_{2}^\vee \rangle = 2(0) = 0 \\
     \langle 4\varpi_{\alpha_{3}}, \alpha_{3}^\vee \rangle = 4(1) = 4\end{cases} \Longrightarrow {\bf{b}}_{I}({\bf{\Sigma}}) = \begin{pmatrix} 0 \\ 4\end{pmatrix}.
\end{equation}
The system $C_{I}^{T}{\bf{a}}({\bf{\Sigma}}) = {\bf{b}}_{I}({\bf{\Sigma}})$ becomes
\begin{equation}
\begin{pmatrix} 2 & -1 \\ -2 & 2 \end{pmatrix} \begin{pmatrix} a_{\alpha_{2}}({\bf{\Sigma}}) \\ a_{\alpha_{3}}({\bf{\Sigma}}) \end{pmatrix} = \begin{pmatrix} 0 \\ 4 \end{pmatrix}.
\end{equation}
Thus, we have $a_{\alpha_{2}}({\bf{\Sigma}}) = 2$ and $a_{\alpha_{3}}({\bf{\Sigma}}) = 4$. From this, we conclude that 
\begin{equation}
\begin{split}
\lambda({\bf{\Sigma}}) &= \Big (\sum_{\alpha \in I} a_{\alpha}({\bf{\Sigma}}) \langle \alpha, \alpha_{1}^{\vee} \rangle \Big) \varpi_{\alpha_{1}} = \Big( a_{\alpha_{2}}({\bf{\Sigma}})\langle \alpha_{2}, \alpha_{1}^{\vee} \rangle + a_{\alpha_{3}}({\bf{\Sigma}})\langle \alpha_{3}, \alpha_{1}^{\vee} \rangle \Big) \varpi_{\alpha_{1}}\\
&= \big(2(-1) + 4(0)\big) \varpi_{\alpha_{1}} = -2\varpi_{\alpha_{1}},
\end{split}
\end{equation}
see Eq. (\ref{CartanB3}). Therefore, we conclude that $\lambda({\bf{\Sigma}}) = -2\varpi_{\alpha_{1}}$. In particular, notice that 
\begin{equation}
\wedge^{4}({\bf{\Sigma}}) = \det({\bf{\Sigma}}) = \mathscr{O}_{\alpha_{1}}(-1).
\end{equation}
\end{example}

Now we consider the following result.

\begin{lemma}
\label{integralconditionsplit}
Let ${\bf{E}}\to  X_{P}$ be a homogeneous vector bundle defined by some finite-dimensional representation $\theta \colon P \to {\rm{GL}}(W)$. If 
\begin{equation}
\label{integralconditioncone}
 \frac{1}{{\rm{rank}}({\bf{E}})} \int_{C}c_{1}({\bf{E}}) \in \mathbbm{Z}, \ \ \ \forall [C] \in {\rm{NE}}(X_{P})_{\mathbbm{Z}},
\end{equation}
then there exists ${\bf{L}}_{0} \in {\rm{Pic}}(X_{P})$ and a homogeneous vector bundle ${\bf{E}}_{0}$, such that 
\begin{equation}
\label{splitingprodlemma}
{\bf{E}} \cong {\bf{E}}_{0} \otimes {\bf{L}}_{0} \ \  \ \ \ \text{and} \ \ \ \ \ c_{1}({\bf{E}}_{0}) = 0.
\end{equation}
\end{lemma}

\begin{proof}
Suppose that ${\bf{E}}$ satisfies the condition given in Eq. (\ref{integralconditioncone}). In particular, it follows that
\begin{equation}
\frac{1}{{\rm{rank}}({\bf{E}})} \int_{\mathbbm{P}_{\alpha}^{1}}c_{1}({\bf{E}}) \in \mathbbm{Z},
\end{equation}
for every $\alpha \in \Delta \backslash I$. Denoting $r = {\rm{rank}}({\bf{E}})$ and considering $\Lambda_{P} \cong {\rm{Hom}}(P,\mathbbm{C}^{\times})$, we obtain that 
\begin{equation}
\vartheta := \frac{\lambda({\bf{E}})}{r} = \sum_{\alpha \in \Delta \backslash I} \frac{\langle c_{1}({\bf{E}}),[\mathbbm{P}_{\alpha}^{1}] \rangle}{r} \varpi_{\alpha}\in \Lambda_{P}.
\end{equation}
From this, we set
\begin{equation}
{\bf{L}}_{0} := \Big \{(U_{i})_{i \in J}, \vartheta^{-1} \circ \psi_{i j} \colon U_{i} \cap U_{j} \to \mathbbm{C}^{\times} \Big \} = \bigotimes_{\alpha \in \Delta \backslash I}\mathscr{O}_{\alpha}(1)^{\otimes \langle \vartheta,\alpha^{\vee}\rangle}.
\end{equation}
By construction, we have ${\bf{L}}_{0}^{\otimes r} = \det ({\bf{E}})$, thus $c_{1}({\bf{L}}_{0}) = \frac{1}{r}c_{1}({\bf{E}})$, see Lemma \ref{firsthomovec}. Now we define
\begin{equation}
{\bf{E}}_{0} := {\bf{E}} \otimes {\bf{L}}_{0}^{-1}.
\end{equation}
With a direct computation, it follows that 
\begin{equation}
c_{1}({\bf{E}}_{0}) = c_{1}({\bf{E}}) - rc_{1}({\bf{L}}_{0}) = 0.
\end{equation}
In order to conclude the proof, we observe that ${\bf{E}} \cong ({\bf{E}} \otimes {\bf{L}}_{0}^{-1}) \otimes {\bf{L}}_{0} = {\bf{E}}_{0} \otimes {\bf{L}}_{0}$.
\end{proof}

\begin{example}
\label{tangentgrassman}
By keeping the notation of Example \ref{Klein_universal}, let $X = {\rm{Gr}}_{2}(\mathbbm{C}^{4})$. In this setting, observing that 
\begin{equation}
\Phi^{+} = \big \{\alpha_{1},\alpha_{2},\alpha_{3}, \alpha_{4}:= \alpha_{1} + \alpha_{2}, \alpha_{5}:= \alpha_{2}+\alpha_{3}, \alpha_{6} := \alpha_{1}+\alpha_{2} +\alpha_{3}\big \},
\end{equation}
since $I = \{\alpha_{1},\alpha_{3}\}$, it follows that 
\begin{equation}
\Phi_{I}^{+} = \big \{\alpha_{2}, \alpha_{1} + \alpha_{2}, \alpha_{2}+\alpha_{3}, \alpha_{1}+\alpha_{2} +\alpha_{3}\big \}.
\end{equation}
From above, we obtain
\begin{equation}
\delta_{P} = \sum_{\alpha \in \Phi_{I}^{+}}\alpha = 2\alpha_{1} + 4\alpha_{2} + 2\alpha_{3}.
\end{equation} 
Considering the associated Cartan matrix $C = (C_{ij}) $ of $\mathfrak{sl}_{4}(\mathbbm{C})$, i.e.,
\begin{equation}
C = \begin{pmatrix} 2 & -1 & 0 \\
-1 & 2 & -1 \\
0 & -1 & 2\end{pmatrix}, \ \ C_{ij} = \langle \alpha_{i}, \alpha_{j}^{\vee} \rangle, \ \ i,j = 1,2,3,
\end{equation}
we conclude that $\delta_{P} = \langle \delta_{P},\alpha_{2}^{\vee} \rangle \varpi_{\alpha_{2}} = 4 \varpi_{\alpha_{2}}$. From this, we have
\begin{equation}
\frac{\langle c_{1}(T^{1,0}{\rm{Gr}}_{2}(\mathbbm{C}^{4})), [\mathbbm{P}_{\alpha_{2}}^{1}] \rangle}{4} = \frac{\langle \delta_{P},h_{\alpha_{2}}^{\vee} \rangle}{4} = 1 \in \mathbbm{Z}.
\end{equation}
From Lemma \ref{integralconditionsplit}, it follows that 
\begin{equation}
T^{1,0}{\rm{Gr}}_{2}(\mathbbm{C}^{4}) \cong {\bf{E}}_{0} \otimes {\bf{L}}_{0},
\end{equation}
such that ${\bf{L}}_{0} \in {\rm{Pic}}({\rm{Gr}}_{2}(\mathbbm{C}^{4}))$ and $c_{1}({\bf{E}}_{0}) = 0$.
\end{example}

Let us provide now an example where the splitting criterion of Lemma \ref{integralconditionsplit} fails.

\begin{example}
As in Example \ref{spinorbunlde}, consider the spinor bundle over the $5$-dimensional quadric ${\bf{\Sigma}} \to Q_{5}$, that is,
\begin{equation}
Q_{5} = {\rm{Spin}}(7,\mathbb{C})/P_{I}  \ \ \ \text{and} \ \ \
{\bf{\Sigma}} = {\rm{Spin}}(7,\mathbb{C}) \times_{P} W(\varpi_{\alpha_{3}}).
\end{equation}
As we have seen, in this case, we have $\lambda({\bf{\Sigma}}) = -2 \varpi_{\alpha_{1}}$. Hence, since ${\rm{rank}}({\bf{\Sigma}}) = 4$, we conclude that 
\begin{equation}
\frac{1}{{\rm{rank}}({\bf{\Sigma}})} \int_{\mathbbm{P}_{\alpha_{1}}^{1}}c_{1}({\bf{\Sigma}}) = \frac{1}{4}(-2) = -\frac{1}{2} \notin \mathbbm{Z}.
\end{equation}
Therefore, ${\bf{\Sigma}}$ does not admit a splitting as in Eq. (\ref{splitingprodlemma}). 
\end{example}

Combining the previous results, we arrive in the proof of Theorem \ref{theoremA}.

\begin{proof}[Proof of Theorem \ref{theoremA}]
Denoting $r = {\rm{rank}}({\bf{E}})$, the proof goes as follows.
\begin{itemize}
\item \underline{(A) $\iff (B)$}: Suppose that ${\bf{E}} \cong {\bf{E}}_{0} \otimes {\bf{L}}_{0}$, such that ${\bf{L}}_{0} \in {\rm{Pic}}(X_{P})$ and $c_{1}({\bf{E}}_{0}) = 0$. It follows that 
\begin{equation}
c_{1}({\bf{E}}) = c_{1}({\bf{E}}_{0}) + rc_{1}({\bf{L}}_{0}) = rc_{1}({\bf{L}}_{0}) \iff \frac{c_{1}({\bf{E}})}{r} = c_{1}({\bf{L}}_{0}) \in H^{2}(X_{P},\mathbbm{Z}).
\end{equation}
From above, we conclude that 
\begin{equation}
\frac{1}{r} \int_{\mathbbm{P}_{\alpha}^{1}}c_{1}({\bf{E}}) \in \mathbbm{Z},
\end{equation}
for every $\alpha \in \Delta \backslash I$. Since ${\rm{NE}}(X_{P})_{\mathbbm{Z}} = \bigoplus_{\alpha \in \Delta \backslash I}\mathbbm{Z}_{\geq 0}[\mathbbm{P}_{\alpha}^{1}]$, we obtain that $\int_{C}c_{1}({\bf{E}}) \in \mathbbm{Z}$, $\forall [C] \in {\rm{NE}}(X_{P})_{\mathbbm{Z}}$. The implication $(B) \Rightarrow (A)$ follows from Lemma \ref{integralconditionsplit}.

\item \underline{(B) $\iff (C)$}: By definition of $\lambda({\bf{E}})$ and the fact that $[\mathbbm{P}_{\alpha}^{1}]$, $\alpha \in \Delta \backslash I$, generates ${\rm{NE}}(X_{P})_{\mathbbm{Z}}$, we have 
\begin{equation}
\displaystyle \frac{1}{r} \int_{C}c_{1}({\bf{E}}) \in \mathbbm{Z}, \forall [C] \in {\rm{NE}}(X_{P})_{\mathbbm{Z}} \iff \frac{\lambda({\bf{E}})}{r} \in \Lambda_{P}.
\end{equation}
From Lemma \ref{firstcherclassformula}, we have 
\begin{equation}
\frac{\lambda({\bf{E}})}{r} = \sum_{\beta \in \Delta \backslash I} \Big ( \sum_{\alpha \in I}\frac{\det(C_{I}(\lambda_{s},\alpha))}{\det(C_{I})} \langle \alpha,\beta^{\vee} \rangle \Big )\varpi_{\beta} - \lambda_{c}.
\end{equation}
Since $\lambda_{c} \in \Lambda_{P}$, we conclude that 
\begin{equation}
\frac{\lambda({\bf{E}})}{r} \in \Lambda_{P} \iff  \sum_{\alpha \in I}\frac{\det(C_{I}(\lambda_{s},\alpha))}{\det(C_{I})} \langle \alpha,\beta^{\vee} \rangle \in \mathbbm{Z}, \ \ \forall \beta \in \Delta \backslash I.
\end{equation}
Therefore, we obtain $(B)\iff(C)$, which concludes the proof.
\end{itemize}

\end{proof}

\begin{example}
Considering $X = {\rm{Gr}}_{2}(\mathbbm{C}^{4})$ as in Example \ref{tangentgrassman}, since ${\rm{Gr}}_{2}(\mathbbm{C}^{4})$ is an irreducible compact Hermitian symmetric space, it follows that $T^{1,0}{\rm{Gr}}_{2}(\mathbbm{C}^{4})$ is a homogeneous vector bundle associate to an irreducible representation satisfying the equivalent conditions of Theorem \ref{theoremA}.
\end{example}

Let us provide a detailed example outside of type $A_{n}$.

\begin{example}[{\bf{Symmetric Square of Spinor Bundle}}]
\label{Squarespinorbundle}
Consider the holomorphic vector bundle ${\bf{E}} = S^{2}({\bf{\Sigma}})$ provided by the Symmetric Square of Spinor Bundle ${\bf{\Sigma}} \to Q_{5}$. In this case, observing that 
\begin{equation}
W(2\varpi_{\alpha_{3}}) \subseteq S^{2}(W(\varpi_{\alpha_{3}})) \subset W(\varpi_{\alpha_{3}}) \otimes W(\varpi_{\alpha_{3}}),
\end{equation}
and $\dim(W(2\varpi_{\alpha_{3}})) = \dim S^{2}(W(\varpi_{\alpha_{3}})) = 10$, we conclude that $W(2\varpi_{\alpha_{3}}) = S^{2}(W(\varpi_{\alpha_{3}}))$. Thus, we have 
\begin{equation}
S^{2}({\bf{\Sigma}}) = {\rm{Spin}}(7,\mathbbm{C}) \times_{P}W(2\varpi_{\alpha_{3}}).
\end{equation}
From Lemma \ref{firstcherclassformula}, we have
\begin{equation}
\lambda(S^{2}({\bf{\Sigma}})) = \big( \sum_{\alpha \in I} a_{\alpha}(S^{2}({\bf{\Sigma}})) C_{\alpha, \alpha_{1}} \big) \varpi_{\alpha_{1}} - 2\lambda_c.
\end{equation}
Since in this case $\lambda = \lambda_{s} + \lambda_{c}$, with $\lambda_{s} = 2\varpi_{\alpha_{3}}$ and $\lambda_{c} = 0$, we conclude that 
\begin{equation}
\lambda(S^{2}({\bf{\Sigma}})) = \big( \sum_{\alpha \in I} a_{\alpha}(S^{2}({\bf{\Sigma}})) C_{\alpha, \alpha_{1}} \big) \varpi_{\alpha_{1}}.
\end{equation}
Following the method from Lemma \ref{firstcherclassformula}, we solve the linear system $C_{I}^{T}{\bf{a}}(S^{2}({\bf{\Sigma}})) = {\bf{b}}_{I}(S^{2}({\bf{\Sigma}}))$. Recall from Example \ref{spinorbunlde} that 
\begin{equation}
C_{I} = \begin{pmatrix} 2 & -2 \\ -1 & 2 \end{pmatrix}.
\end{equation}
Here, the vector ${\bf{b}}_{I}(S^{2}({\bf{\Sigma}}))$ is determined by the rank $r=10$ and the highest weight $\lambda_{s} = 2\varpi_{\alpha_{3}}$
\begin{equation}
{\bf{b}}_{I}(S^{2}({\bf{\Sigma}})) = \begin{pmatrix} \langle 10(2\varpi_{\alpha_{3}}), \alpha_{2}^{\vee} \rangle \\ \langle 10(2\varpi_{\alpha_{3}}), \alpha_{3}^{\vee} \rangle \end{pmatrix} = \begin{pmatrix} 0 \\ 20 \end{pmatrix}.
\end{equation}

The system $C_{I}^{T}{\bf{a}}(S^{2}({\bf{\Sigma}})) = {\bf{b}}_{I}(S^{2}({\bf{\Sigma}}))$ becomes:
\begin{equation}
\begin{pmatrix} 2 & -1 \\ -2 & 2 \end{pmatrix} \begin{pmatrix} a_{\alpha_{2}}(S^{2}({\bf{\Sigma}})) \\ a_{\alpha_{3}}(S^{2}({\bf{\Sigma}})) \end{pmatrix} = \begin{pmatrix} 0 \\ 20 \end{pmatrix}.
\end{equation}

Solving this yields $a_{\alpha_{2}}(S^{2}({\bf{\Sigma}})) = 10$ and $a_{\alpha_{3}}(S^{2}({\bf{\Sigma}})) = 20$. Now we evaluate condition (C) of Theorem \ref{theoremA} for the remaining root $\beta = \alpha_{1} \in \Delta \backslash I$. We must verify if the following sum 
\begin{equation}
\frac{1}{r} \sum_{\alpha \in I} a_{\alpha}(S^{2}({\bf{\Sigma}})) \langle \alpha, \alpha_{1}^{\vee} \rangle = \frac{1}{10} \Big( a_{\alpha_{2}}(S^{2}({\bf{\Sigma}}))\langle \alpha_{2}, \alpha_{1}^{\vee} \rangle + a_{\alpha_{3}}(S^{2}({\bf{\Sigma}}))\langle \alpha_{3}, \alpha_{1}^{\vee} \rangle \Big)
\end{equation}
belongs to $\mathbb{Z}$. Since $\langle \alpha_{2}, \alpha_{1}^{\vee} \rangle = -1$ and $\langle \alpha_{3}, \alpha_{1}^{\vee} \rangle = 0$ in the $B_{3}$ root system, we have
\begin{equation}
\frac{1}{10} \Big( 10(-1) + 20(0) \Big) = -1 \in \mathbb{Z}.
\end{equation}
Hence, Theorem \ref{theoremA} guarantees that ${\bf{E}} = S^{2}({\bf{\Sigma}})$ admits a topological splitting 
\begin{center}
$S^{2}({\bf{\Sigma}})\cong {\bf{E}}_{0} \otimes {\bf{L}}_{0}$, 
\end{center}
such that ${\bf{L}}_{0} \in {\rm{Pic}}(Q_{5})$ and $c_{1}({\bf{E}}_{0}) = 0$. 
\end{example}

Our next example illustrates how to characterize completely the irreducible homogeneous vector bundles which admits a splitting of the form ${\bf{E}}_{0} \otimes {\bf{L}}_{0}$, with ${\bf{L}}_{0} \in {\rm{Pic}}(X_{P})$ and $c_{1}({\bf{E}}_{0}) = 0$, just using the underlying Lie-theoretical data.

\begin{example}
\label{examplespin}
Given $G^{\mathbb{C}} = \text{Spin}(8, \mathbb{C})$, we consider the identification
\begin{equation}
\mathfrak{spin}(8, \mathbb{C}) \cong \mathfrak{so}(8, \mathbb{C}) = \{X \in {\rm{M}}_{7\times 7}(\mathbbm{C}) \ | \ X^{T}J + JX = 0\},
\end{equation}
such that
\begin{equation}
J = \begin{pmatrix} 0 & I_{4} \\
I_{4} & 0\end{pmatrix}.
\end{equation}
Fixing the Cartan subalgebra $\mathfrak{h} \subset \mathfrak{spin}(8, \mathbb{C})$ provided by the diagonal matrices in the standard representation, we consider the simple root system 
\begin{equation}
\Delta = \{ \alpha_{1} = \epsilon_{1} - \epsilon_{2}, \alpha_{2} = \epsilon_{2} - \epsilon_{3}, \alpha_{3} = \epsilon_{3} - \epsilon_{4}, \alpha_{4} = \epsilon_{3} + \epsilon_{4} \}, 
\end{equation}
such that $\epsilon_{i}(\text{diag}(x_{1}, \ldots, x_{4}, -x_{4}, \ldots, -x_{1})) = x_{i}$, for every $i = 1, \ldots, 4$. 
\begin{equation}
{\dynkin[labels={\alpha_{1},\alpha_{2},\alpha_{3},\alpha_{4}},scale=3]D{oooo}} 
\end{equation}
The associated Cartan matrix $C = (C_{ij})$ of $\mathfrak{spin}(8, \mathbb{C})$ is given by
\begin{equation}
C = \begin{pmatrix} 
2 & -1 & 0 & 0 \\
-1 & 2 & -1 & -1 \\
0 & -1 & 2 & 0 \\
0 & -1 & 0 & 2
\end{pmatrix}, \ \ C_{ij} = \langle \alpha_{i}, \alpha_{j}^{\vee} \rangle, \ \ i,j = 1,2,3,4.
\end{equation}
In this case, we consider $I = \{\alpha_{1},\alpha_{2}\}$, i.e.,
\begin{equation}
{\dynkin[labels={\alpha_{1},\alpha_{2},\alpha_{3},\alpha_{4}},scale=3]D{oo**}} 
\end{equation}
From above, we have the flag variety 
\begin{equation}
X_{P} = {\rm{Spin}}(8,\mathbbm{C})/P_{I}.
\end{equation}
In this case, we consider $I = \{\alpha_{1}, \alpha_{2}\}$. The Levi component $L_{P}$ has a semisimple part $S_{P}$ of type $A_{2}$, that is $S_{P} \cong {\rm{SL}}_{3}(\mathbbm{C})$. The Cartan matrix $C_{I} = (\langle \alpha, \beta^{\vee} \rangle)_{\alpha, \beta \in I}$ of $\mathfrak{s}_{P}$ is given by
\begin{equation}
C_{I} = \begin{pmatrix} 2 & -1 \\ -1 & 2 \end{pmatrix},
\end{equation}
which implies $\det(C_{I}) = 3$. 

Let us evaluate the algebraic criterion provided by item (C) of Theorem \ref{theoremA} for an arbitrary homogeneous vector bundle 
\begin{center}
${\bf{E}}_{\theta} = {\rm{Spin}}(8,\mathbbm{C}) \times_{P}W(\lambda) \rightarrow X_{P}$ 
\end{center}
defined by some irreducible representation of $\theta \colon P \to {\rm{GL}}(W(\lambda))$, with highest weight $\lambda = \lambda_{s} + \lambda_{c}$, such that $\lambda_{s} \in \Lambda_{\mathfrak{s}_{P}}^{+}$ and $\lambda_{c} \in \Lambda_{P}$. Denoting $\langle \lambda_{s}, \alpha_{1}^{\vee} \rangle = m_{1}$ and $\langle \lambda_{s}, \alpha_{2}^{\vee} \rangle = m_{2}$, and considering the modified matrices $C_{I}(\lambda_{s}, \alpha)$ as in item (C) of Theorem \ref{theoremA}, we obtain
\begin{equation}
\det(C_{I}(\lambda_{s}, \alpha_{1})) = \det \begin{pmatrix} m_{1} & -1 \\ m_{2} & 2 \end{pmatrix} = 2m_{1} + m_{2},
\end{equation}
\begin{equation}
\det(C_{I}(\lambda_{s}, \alpha_{2})) = \det \begin{pmatrix} 2 & m_{1} \\ -1 & m_{2} \end{pmatrix} = 2m_{2} + m_{1}.
\end{equation}
Taking $\beta = \alpha_{3}$, since $\langle \alpha_{1}, \alpha_{3}^{\vee} \rangle = 0$ and $\langle \alpha_{2}, \alpha_{3}^{\vee} \rangle = -1$, we have:
\begin{equation}
\sum_{\alpha \in I} \frac{\det(C_{I}(\lambda_{s}, \alpha))}{\det(C_{I})} \langle \alpha, \alpha_{3}^{\vee} \rangle = \frac{2m_{1} + m_{2}}{3}(0) + \frac{m_{1} + 2m_{2}}{3}(-1) = -\frac{m_{1} + 2m_{2}}{3}.
\end{equation}
By symmetry, the evaluation for $\beta = \alpha_{4}$ yields the exact same value, since $\langle \alpha_{2}, \alpha_{4}^{\vee} \rangle = -1$. Therefore, the vector bundle ${\bf{E}}_{\theta}$ admits a topological splitting if and only if 
\begin{equation}
\frac{m_{1} + 2m_{2}}{3} \in \mathbb{Z}. 
\end{equation}
Let us illustrate this with two examples of vector bundles:
\begin{itemize}
\item\textbf{Failure of Splitting:}

\noindent Consider ${\bf{E}}_{\varrho} \rightarrow X_{P}$ defined by the representation $\varrho\colon P \to {\rm{GL}}(W(\lambda))$, i.e., with highest weight $\lambda = \lambda_{s} + \lambda_{c}$. In this case, supposing that $\lambda_{s}  = \varpi_{\alpha_{1}}$, we have $m_{1} = 1$ and $m_{2} = 0$. Evaluating our algebraic condition, we obtain:
\begin{equation}
-\frac{1 + 2(0)}{3} = -\frac{1}{3} \notin \mathbb{Z}.
\end{equation}
Thus, Theorem \ref{theoremA} ensures that the fundamental bundle ${\bf{E}}_{\varrho}$ does not admit a splitting of the form ${\bf{E}}_{0} \otimes {\bf{L}}_{0}$ with ${\bf{L}}_{0} \in {\rm{Pic}}(X_{P})$ and $c_{1}({\bf{E}}_{0}) = 0$.
\item\textbf{Successful Splitting:}

\noindent Now, let ${\bf{E}}_{\theta} \rightarrow X_{P}$ be the vector bundle associated with the irreducible representation $\theta \colon P \to {\rm{GL}}(W(\lambda))$, which corresponds to the highest weight $\lambda = \lambda_{s} + \lambda_{c}$. Suppose in this case that $\lambda_{c} = \varpi_{\alpha_{1}} + \varpi_{\alpha_{2}}$. From this, we have $m_{1} = 1$ and $m_{2} = 1$. Evaluating the condition for $\beta \in \{\alpha_{3}, \alpha_{4}\}$, we obtain
\begin{equation}
-\frac{1 + 2(1)}{3} = -1 \in \mathbb{Z}.
\end{equation}
Because this value is an integer, Theorem A guarantees that the adjoint bundle admits a global topological splitting ${\bf{E}}_{\theta} \cong {\bf{E}}_{0} \otimes {\bf{L}}_{0}$.
\end{itemize}
\end{example}

\section{Mean curvature of homogeneous Hermitian vector bundles}

Given a K\"{a}hler manifold $(X,\omega)$, in what follows we consider the Hodge $\ast$-operator defined by $\omega$ and the associated codifferential $\delta := - \ast {\rm{d}} \ast $, such that
\begin{equation}
(\eta, \varphi)\frac{\omega^{n}}{n!} = \eta \wedge \ast \bar{\varphi}, 
\end{equation}
for all $\eta,\varphi \in \Omega^{p,q}(X)$. From the decomposition ${\rm{d}}= \partial + \bar{\partial}$, we have $\delta := \partial^{\ast} + \bar{\partial}^{\ast}$, such that 
\begin{equation}
\partial^{\ast} := - \ast \bar{\partial} \ast, \ \ \ \bar{\partial}^{\ast}:= - \ast  \partial \ast.
\end{equation}
Let us consider the Lefschetz operator ${\rm{L}}_{\omega} \colon \Omega^{p,q}(X) \to \Omega^{p+1,q+1}(X)$, such that 
\begin{equation}
{\rm{L}}_{\omega} = \omega\wedge (-), 
\end{equation} 
its dual $\Lambda_{\omega}  \colon \Omega^{p,q}(X) \to \Omega^{p-1,q-1}(X)$, and the Laplacian operators
\begin{equation}
{\bf{\Delta}}_{\omega} = - \big ({\rm{d}} \delta + \delta {\rm{d}} \big), \ \ \square_{\omega}:= - \big (\partial \partial^{\ast} + \partial^{\ast} \partial \big ), \ \ \overline{\square}_{\omega}:= - \big (\bar{\partial} \bar{\partial}^{\ast} + \bar{\partial}^{\ast} \bar{\partial} \big ).
\end{equation}
In this setting, we have the following well-known result (e.g. \cite{wells1980differential}, \cite{MR2093043}, \cite{Griffiths}).
\begin{theorem}[K\"{a}hler identities] 
\label{Kidentities}
Let $(X,\omega)$ be a K\"{a}hler manifold. Then the following identities hold true:
\begin{enumerate}
\item[(K1)] $[\partial,{\rm{L}}_{\omega}] = [\bar{\partial},{\rm{L}}_{\omega}] = 0$ \ and \ $[\partial^{\ast},\Lambda_{\omega}] = [\bar{\partial}^{\ast},\Lambda_{\omega}] = 0$;
\item[(K2)] $[\partial^{\ast},{\rm{L}}_{\omega}] = -\sqrt{-1}\bar{\partial}$ \ and \ $[\bar{\partial}^{\ast},{\rm{L}}_{\omega}] = \sqrt{-1}\partial$;
\item[(K3)] $[\Lambda_{\omega},\partial] = \sqrt{-1}\bar{\partial}^{\ast}$ \ and \ $[\Lambda_{\omega},\bar{\partial}] = - \sqrt{-1}\partial^{\ast}$;
\item[(K4)] ${\bf{\Delta}}_{\omega} = 2 \square_{\omega} = 2 \overline{\square}_{\omega}$ and ${\bf{\Delta}}_{\omega}$ commutes with $\partial, \bar{\partial}, \partial^{\ast}, \bar{\partial}^{\ast}, {\rm{L}}_{\omega},$ and $\Lambda_{\omega}$.
\end{enumerate}
\end{theorem}

Consider now the following result.

\begin{proposition}
\label{eigenvalueatorigin}
Let $X_{P}$ be a flag variety and let $\omega_{0}$ be a $G$-invariant K\"{a}hler metric on $X_{P}$. Then, for every closed $G$-invariant real $(1,1)$-form $\psi$, the eigenvalues of the endomorphism $\omega_{0}^{-1} \circ \psi$ are given by 
\begin{equation}
\label{eigenvalues}
{\bf{q}}_{\beta}(\omega_{0}^{-1} \circ \psi) = \frac{ \langle \phi([\psi]), \beta^{\vee} \rangle}{\langle \phi([\omega_{0}]), \beta^{\vee} \rangle}, \ \ \beta \in \Phi_{I}^{+},
\end{equation}
such that $\phi([\psi]), \phi([\omega_{0}]) \in \Lambda_{P} \otimes \mathbbm{R}$.
\end{proposition}

\begin{remark}
\label{primitivecalc}
In the setting of the last proposition, since $n \psi\wedge \omega_{0}^{n-1} = \Lambda_{\omega_{0}}(\psi)\omega_{0}^{n}$, such that $n=\dim_{\mathbbm{C}}(X_{P})$, and $\Lambda_{\omega_{0}}(\psi)={\rm{tr}}(\omega_{0}^{-1} \circ \psi)$, it follows that
\begin{equation}
\label{contraction}
\Lambda_{\omega_{0}}({\bf{\Omega}}_{\alpha})=\sum_{\beta \in \Phi_{I}^{+}} \frac{\langle \varpi_{\alpha}, \beta^{\vee} \rangle}{\langle \phi([\omega_{0}]), \beta^{\vee}\rangle},
\end{equation}
for every $\alpha \in \Delta \backslash I$. In particular, for every ${\bf{L}} \in {\rm{Pic}}(X_{P})$, we have a Hermitian structure ${\bf{h}}_{0}$ on ${\bf{L}}$, such that the curvature $F({\bf{h}}_{0})$ of the Chern connection $\nabla^{{\bf{h}}_{0}} \myeq {\rm{d}} + \partial \log ({\bf{h}}_{0})$, satisfies 
\begin{equation}
\label{linebundleHYM}
\frac{\sqrt{-1}}{2\pi} \Lambda_{\omega_{0}}(F({\bf{h}}_{0})) = \sum_{\beta \in \Phi_{I}^{+} } \frac{\langle \lambda({\bf{L}}), \beta^{\vee} \rangle}{\langle \phi([\omega_{0}]), \beta^{\vee}\rangle}.
\end{equation}
From this, we have that $\nabla$ is a Hermitian-Yang-Mills (HYM) connection (e.g. \cite{Kobayashi+1987}). 
\end{remark}
Let us recall some terminology. Fixing a K\"{a}hler class $\xi \in \mathcal{K}(X_{P})$, the {\textit{slope}} $\mu_{\xi}({\bf{E}})$ of a holomorphic vector bundle ${\bf{E}} \to X_{P}$, with respect to $\xi$ (or $\xi$-slope for short), is defined by
\begin{equation}
\mu_{\xi}({\bf{E}}):= \frac{1}{\rank({\bf{E}})} \int_{X_{P}}c_{1}({\bf{E}}) \wedge \xi^{n-1}.
\end{equation}
see for instance \cite{Kobayashi+1987}. From above, a holomorphic vector bundle ${\bf{E}} \to X_{P}$ is said to be $\xi${\textit{-(semi)stable}} if 
\begin{equation}
\mu_{\xi}({\bf{E}}) \geoq \mu_{\xi}(\mathcal{F}),
\end{equation}
for every coherent subsheaf $0 \neq \mathcal{F} \varsubsetneq {\bf{E}}$. Further, we say that ${\bf{E}}$ is $\xi${\textit{-polystable}} if it is isomorphic to a direct sum of stable vector bundles of the same slope, and we say that ${\bf{E}}$ is $\xi${\textit{-unstable}} if it is not $\xi$-semistable. Given a Hermitian holomorphic vector bundle $({\bf{E}},{\bf{h}})$ over $X_{P}$, we have locally
\begin{equation}
F({\bf{h}}) = (F({\bf{h}})_{ij}), \ \ {\text{s.t.}} \ \  F({\bf{h}})_{ij} = \sum_{\alpha,\beta} R^{i}_{j\alpha \overline{\beta}}{\rm{d}}z_{\alpha} \wedge {\rm{d}}\overline{z_{\beta}}.
\end{equation}
where $F({\bf{h}})$ is the curvature of the associated Chern connection. From this, fixed some $G$-invariant K\"{a}hler metric $\omega_{0} \in \Omega^{1,1}(X_{P})^{G}$, we can define the $\omega_{0}$-trace of $F({\bf{h}})$ as being the endomorphism 
\begin{equation}
\Lambda_{\omega_{0}}(F({\bf{h}})) = (\Lambda_{\omega_{0}}(F({\bf{h}})_{ij})), \ \ {\text{s.t.}} \ \ \Lambda_{\omega_{0}}(F({\bf{h}})_{ij}) = \sum_{\alpha,\beta}g^{\alpha \overline{\beta}} R^{i}_{j\alpha \overline{\beta}},
\end{equation}
here we denote by $g_{0}^{-1} = (g^{\alpha \overline{\beta}})$ the inverse matrix of the underlying K\"{a}hler metric $g_{0} = \omega_{0}({\rm{1}} \otimes J)$. 
\begin{definition}
We define the mean curvature of $({\bf{E}},{\bf{h}}) \to (X_{P},\omega_{0})$ as being the endomorphism ${\rm{K}}({\bf{E}},{\bf{h}}) \in \mathcal{A}^{0}({\rm{End}}({\bf{E}}))$ given by
\begin{equation}
{\rm{K}}({\bf{E}},{\bf{h}}) := \sqrt{-1}\Lambda_{\omega_{0}}(F({\bf{h}})),
\end{equation}
where $F({\bf{h}})$ is the curvature of the associated Chern connection.
\end{definition}

Given a holomorphic Hermitian vector bundle $({\bf{E}},{\bf{h}})$ over a compact K\"{a}hler manifold $(X,\omega)$, it follows that the associated Chern connection $\nabla \myeq {\rm{d}} + {\bf{A}}$ induces a covariant exterior derivative ${\rm{d}}_{{\bf{A}}} \colon \mathcal{A}^{\bullet}({\bf{E}}) \to  \mathcal{A}^{\bullet + 1}({\bf{E}})$, such that 
\begin{equation}
{\rm{d}}_{{\bf{A}}}(\eta \otimes {\bf{s}}) = {\rm{d}}\eta \otimes {\bf{s}} + (-1)^{\deg(\eta)}\eta \wedge \nabla \bf{s},
\end{equation}
notice that ${\rm{d}}_{{\bf{A}}} = \nabla$ on $\mathcal{A}^{0}({\bf{E}})$. Observing that ${\rm{d}}_{{\bf{A}}} = \partial_{{\bf{A}}} + \overline{\partial}_{{\bf{A}}}$, and considering the extension of first-order Kahler identities to the complex ${\bf{E}}$-valued differential forms, it follows that 
\begin{equation}
\partial_{{\bf{A}}}^{\ast} = \sqrt{-1}\big [ \Lambda_{\omega},\overline{\partial}_{{\bf{A}}}\big ], \ \ \ \ \overline{\partial}_{{\bf{A}}}^{\ast} = -\sqrt{-1}\big [ \Lambda_{\omega},\partial_{{\bf{A}}}\big ],
\end{equation}
e.g. \cite{ballmann2006lectures}. Since $F({\bf{h}})$ is of $(1,1)$-type and ${\rm{d}}_{{\bf{A}}}F({\bf{h}}) = 0$, it follows that $\overline{\partial}_{{\bf{A}}} F({\bf{h}}) = \partial_{{\bf{A}}}F({\bf{h}}) = 0$. Thus, following \cite[Eq. 2.34]{besse2007einstein}, from the Kahler identities above, we conclude that
\begin{equation}
\delta_{{\bf{A}}}F({\bf{h}}) = -\sqrt{-1} \big (\overline{\partial}_{{\bf{A}}} -  \partial_{{\bf{A}}}\big )\Lambda_{\omega}(F({\bf{h}})) = -J({\rm{d}}_{{\bf{A}}}\Lambda_{\omega}F({\bf{h}})),
\end{equation}
Therefore, we obtain
\begin{equation}
\label{YMisHYM}
\delta_{{\bf{A}}}F({\bf{h}}) = 0 \ \ \ \iff \ \ {\rm{d}}_{\bf{A}} {\rm{K}}({\bf{E}},{\bf{h}}) = 0,
\end{equation}
see for instance \cite[p. 6]{donaldson1985anti}. Now we consider the following definition.
\begin{definition}
A Hermitian structure ${\bf{h}}$ on a holomorphic vector bundle ${\bf{E}} \to (X,\omega)$ is Hermite-Einstein (or Hermitian-Yang-Mills) if 
\begin{equation}
\label{YMhomo}
{\rm{K}}({\bf{E}},{\bf{h}}) = \nu_{[\omega]}({\bf{E}})1_{{\bf{E}}},
\end{equation}
such that $\nu_{[\omega]}({\bf{E}}) := (2\pi) n \big ( \int_{X}[\omega]^{n}\big )^{-1}\mu_{[\omega]}({\bf{E}})$.
\end{definition}
\begin{remark}
For a compact K\"{a}hler manifold we shall consider ${\rm{Vol}}(X,\omega) = \frac{1}{n!}\int_{X}\omega^{n}$. In this case, we have $\nu_{[\omega]}({\bf{E}}) = \frac{2\pi \mu_{[\omega]}({\bf{E}})}{(n-1)!{\rm{Vol}}(X,\omega)}$.
\end{remark}

As we see from Eq. (\ref{YMisHYM}), the Chern connection of a Hermite-Einstein structure is a Yang-Mills connection. In the homogeneous setting we have the following general result, see for instance \cite{RAMANAN1966159}, \cite{umemura1978theorem}, \cite{kobayashi1986homogeneous}.
\begin{theorem}
\label{Ramanan}
Let $G^{\mathbbm{C}}$ be a complex simply connected and simple Lie group. Given a homogeneous holomorphic vector bundle ${\bf{E}} \to X_{P}$, such that $\theta \colon P \to {\rm{GL}}(V)$ is an irreducible $P$-module, and fixed some $G$-invariant (integral) K\"{a}hler metric $\omega_{0}$ on $X_{P}$, then have that ${\bf{E}}$ is $[\omega_{0}]$-stable. In particular, we have that ${\bf{E}}$ admits a Hermite-Einstein structure.
\end{theorem}

\section{Generalities on Spectral Geometry} 
\label{spectralsection}
In this section, we review some basic results related to the spectral geometry of flag varieties. For more details on the subject, we suggest \cite{chavel1984eigenvalues}, \cite{li2012geometric}, \cite{grigoryan2009heat}, \cite{aubin1998some}, \cite{takeuchi1994modern}, \cite{helgason2001differential}.

On a Riemannian manifold $(M,g)$, we have the Laplace-Beltrami operator ${\bf{\Delta}}_{g}$ acting on smooth functions in the following way
\begin{equation}
{\bf{\Delta}}_{g}f = {\rm{div}}(\nabla f) =  -\delta {\rm{d}}f,
\end{equation}
for every smooth function $f \in C^{\infty}(M)$. In local coordinates, we have ${\bf{\Delta}}_{g}$ described as follows
\begin{equation}
{\bf{\Delta}}_{g} = \frac{1}{\sqrt{\det(g)}}\sum_{i,j}\frac{\partial}{\partial x_{i}} \Bigg ( \sqrt{\det(g)}g^{ij}\frac{\partial}{\partial x_{j}}\Bigg),
\end{equation}
such that $g = (g_{ij})$. Let us suppose from now on that $M$ is compact. In this setting, considering the underlying weighted manifold $(M,g,\mu_{g})$, where $\mu_{g}$ denotes the induced Riemannian measure, there exists a complete orthonormal basis $\{\varphi_{0},\varphi_{1},\varphi_{2},\ldots\}$ for the space of square integrable real-valued measurable functions $L^{2}(M) = L^{2}(M,\mu_{g})$, consisting of smooth eigenfunctions of ${\bf{\Delta}}_{g}$, i.e., such that $\varphi_{j} \in C^{\infty}(M)$, and 
\begin{equation}
{\bf{\Delta}}_{g}\varphi_{j} + \lambda_{j}({\bf{\Delta}}_{g})\varphi_{j} = 0,
\end{equation}
for every $j =0,1,2, \dots$, with corresponding eigenvalues arranged in increasing order 
\begin{equation}
0 = \lambda_{0}({\bf{\Delta}}_{g}) < \lambda_{1}({\bf{\Delta}}_{g}) \leq \lambda_{2}({\bf{\Delta}}_{g}) \leq \cdots \uparrow +\infty,
\end{equation}
with each eigenvalue repeated according to its multiplicity. Notice that $\varphi_{0} = {\rm{Vol}}(M,g)^{-\frac{1}{2}}$, see for instance \cite{chavel1984eigenvalues}, \cite{li2012geometric}. In particular, for every $f \in L^{2}(M)$, we have the Parseval identities
\begin{equation}
f = \sum_{j = 0}^{+\infty}(f,\varphi_{j})\varphi_{j}, \ \ \ ||f||_{L^{2}}^{2} = \sum_{j = 0}^{+\infty}(f,\varphi_{j})^{2}.
\end{equation}
Since the operator ${\bf{\Delta}}_{g}$ is non-positive definite, it is frequently more convenient to work with the Dirichlet Laplace operator 
\begin{equation}
\mathcal{L}_{g} := - {\bf{\Delta}}_{g}.
\end{equation}
The Dirichlet Laplace operator $\mathcal{L}_{g}$ is self-adjoint non-negative definite operator in $L^{2}(M)$, such that
\begin{equation}
{\rm{Spec}}(\mathcal{L}_{g}) = \big\{\lambda_{j}({\bf{\Delta}}_{g}) \ \big | \ j = 1,2,\cdots \ \big\}.
\end{equation}
The standard elliptic theory asserts that there exists a Green’s function $G(x,y)$ for the Laplacian ${\bf{\Delta}}_{g}$, which can be defined as being the unique function satisfying the condition
\begin{equation}
{\bf{\Delta}}_{g}G(x,y) = \frac{1}{{\rm{Vol}}_{g}(M)} - \delta_{x}(y), \ \ \int_{M}G(x,y){\rm{d}}\mu_{g}(x) = 0,
\end{equation}
where $\delta_{x}(y)$ is the delta function at $x$. It is well-known that $G(x,y) = G(y,x)$ and that $G(x,y)$ is smooth on $(M \times M) \backslash \Delta(M)$, such that $\Delta(M) = \{(x,y) \in M \times M \ | \ y = x \}$. Moreover, we have that\footnote{Notice the $\dashint_{-}$ stands for $\frac{1}{\rm{Vol}(\cdot,g)} \int_{-}$} 
\begin{equation}
\varphi(x) = \dashint_{M}\varphi(y){\rm{d}}\mu_{g}(y) - \int_{M}G(x,y){\bf{\Delta}}_{g}\varphi(y){\rm{d}}\mu_{g}(y),
\end{equation}
for every $\varphi \in C^{2}(M)$, see for instance \cite{aubin1998some}. By means of the Green function $G(x,y)$ a solution of the Poisson's equation 
\begin{equation}
-{\bf{\Delta}}_{g}u = f, 
\end{equation}
for some $f \in C^{\infty}(M)$, such that $\int_{M}f{\rm{d}}\mu_{g} = 0$, can be described (up to some constant) by
\begin{equation}
u(x) =  \int_{M}G(x,y)f(y){\rm{d}}\mu_{g}(y) = \sum_{j = 1}^{+\infty}\frac{1}{\lambda_{j}}(f,\varphi_{j})\varphi_{j}(x), \ \ \forall x \in M.
\end{equation}
From above, we consider the operator $(-{\bf{\Delta}}_{g})^{-1} \colon C^{\infty}(M) \to C^{\infty}(M)$, such that 
\begin{equation}
(-{\bf{\Delta}}_{g})^{-1}f :=  \int_{M}G(x,y)f(y){\rm{d}}\mu_{g}(y),
\end{equation}
for every $f \in C^{\infty}(M)$. 

Let us introduce now some basic terminology, for more details we suggest \cite{grigoryan2009heat}, \cite{aubin1998some}. Given a compact weighted manifold $(M,g,\mu_{g})$, let $\mathcal{D}'(M)$ denotes the set of all distributions on the space $C^{\infty}(M)$ and let $L_{{\rm{loc}}}^{1}(M)$ denotes the space of locally integrable functions on $M$. In this setting, every $\psi \in L_{{\rm{loc}}}^{1}(M)$ defines a distribution $(\psi,-) \in \mathcal{D}'(M)$, such that 
\begin{equation}
(\psi,\varphi) = \int_{M}\psi(x)\varphi(x){\rm{d}}\mu_{g}(x), \ \ \ \ \forall \varphi \in C^{\infty}(M).
\end{equation}
In particular, we have 
\begin{equation}
L^{2}(M) \hookrightarrow L_{\rm{loc}}^{2}(M) \hookrightarrow L_{\rm{loc}}^{1}(M) \hookrightarrow \mathcal{D}'(M).
\end{equation}
From above, we define the distributional Laplacian of $\psi \in L_{{\rm{loc}}}^{1}(M)$ by means of the identity
\begin{equation}
({\bf{\Delta}}_{g}\psi,\varphi) = (\psi,{\bf{\Delta}}_{g}\varphi), \ \ \ \ \forall \varphi \in C^{\infty}(M).
\end{equation}
By definition, we have ${\bf{\Delta}}_{g}\psi \in \mathcal{D}'(M)$, for every $\psi \in L_{{\rm{loc}}}^{1}(M)$. In particular, we obtain the self-adjoint extension  ${\bf{\Delta}}_{g} \colon H^{2}(M) \to L^{2}(M)$, where $H^{2}(M)$ is the Sobolev space defined by the completion of $C^{\infty}(M)$ with respect to the inner product
\begin{equation}
\langle u,v \rangle_{H^{2}} = \int_{M}uv{\rm{d}}\mu_{g} + \int_{M}\langle \nabla^{2}u,\nabla^{2}v \rangle_{g}{\rm{d}}\mu_{g}.
\end{equation}
The Sobolev space $H^{2}(M)$ also can be characterized in terms of ${\rm{Spec}}(-{\bf{\Delta}}_{g})$, namely, we have
\begin{equation}
H^{2}(M) = \Big \{ f \in L^{2}(M,\mu_{g}) \ \ \Big | \ \ \sum_{j = 0}^{\infty}\lambda_{j}({\bf{\Delta}}_{g})^{2} (f,\varphi_{j})^{2} < + \infty \ \ \Big \},
\end{equation}
see for instance \cite[Remark 7.6]{lions2012non}, \cite[p. 134-139]{craioveanu2013old}.

\section{Proof of Theorem B}

In this subsection, we will prove Theorem \ref{Hcase1}. Let us establish some preliminary results concerning singular Hermitian structures on line bundles.

Consider the following definition.

\begin{definition}[Demailly, \cite{damailly1992singular}]
\label{defsingline}
Let $(X,\omega)$ be a compact K\"{a}hler manifold and ${\bf{L}} \in {\rm{Pic}}(X)$. We say that ${\bf{h}}$ is a singular Hermitian structure on ${\bf{L}}$ if for every trivializing open set $U \subset X$ we have 
\begin{equation}
{\bf{h}}|_{U} = {\rm{e}}^{\varphi}|\cdot|^{2},
\end{equation}
where $\varphi \in L_{{\rm{loc}}}^{1}(U)$ is called the local weight of ${\bf{h}}$.
\end{definition}

In the particular case that ${\bf{h}}(\varphi) := {\rm{e}}^{2\varphi}{\bf{h}}_{0}$, where $\varphi \in L_{{\rm{loc}}}^{1}(X)$ and ${\bf{h}}_{0}$ is a smooth Hermitian structure on ${\bf{L}}$, the curvature (1,1)-current of $({\bf{L}},{\bf{h}})$ (in the sense of distribution) is given by 
\begin{equation}
F({\bf{h}}) = F({\bf{h}}_{0}) - 2\partial \overline{\partial}\varphi. 
\end{equation}
see for instance \cite[\S 2.2]{cataldo1998singular}, \cite[\S 3]{raufi2015singular}. From above we define the mean curvature of a singular line bundle $({\bf{L}},{\bf{h}}(\varphi) = {\rm{e}}^{2\varphi}{\bf{h}}_{0})$ as being
\begin{equation}
{\rm{K}}({\bf{L}},{\bf{h}}(\varphi)):= \sqrt{-1}\Lambda_{\omega}(F({\bf{h}}_{0}) ) - {\bf{\Delta}}_{\omega}\varphi,
\end{equation}
such that the expression on the right-hand side above is understood in the sense of distribution. Fixed some $f \in L^{2}(X,\mu_{\omega})$, the prescribed mean curvature problem in this last setting can be formulated as the following PDE in sense of distribution
\begin{equation}
{\rm{K}}({\bf{L}},{\bf{h}}(\varphi)) = f, \ \ \ \ \ \ \varphi \in H^{2}(X).
\end{equation}
In the above setting, we have the following result.
\begin{lemma}
\label{lemmafund}
Let $(X,\omega)$ a compact K\"{a}hler manifold and ${\bf{L}} \in {\rm{Pic}}(X)$. If $f \in L^{2}(X,\mu_{\omega})$ satisfies
\begin{equation}
\frac{1}{2\pi}\dashint_{X}f(x){\rm{d}}\mu_{\omega}(x) = \frac{\deg_{\omega}({\bf{L}})}{(n-1)!{\rm{Vol}}(X,\omega)},
\end{equation}
then there exists a sequence of smooth Hermitian structures $\{{\bf{h}}_{n}\}$ on ${\bf{L}}$, such that 
\begin{equation}
{\rm{K}}({\bf{L}},{\bf{h}}_{n}) \overset{L^2}{\longrightarrow}  f \ \ \ \ {\text{and}} \ \ \ \ {\bf{h}}_{n} \overset{a.e.}{\longrightarrow} {\bf{h}}_{\infty},
\end{equation}
as $n \uparrow +\infty$, where ${\bf{h}}_{\infty}$ is a singular Hermitian structure on ${\bf{L}}$. Moreover, we have
\begin{equation}
{\rm{K}}({\bf{L}},{\bf{h}}_{\infty}) = f,
\end{equation}
in the sense of distributions. 
\end{lemma}
\begin{proof}
By fixing an orthonormal basis $\{\varphi_{0}, \varphi_{1}, \ldots \}$ of eigenfunctions of the Laplace-Beltrami operator $-{\bf{\Delta}}_{\omega}$ for $L^{2}(X, \mu_{\omega})$, we expand the target function as $f = \sum_{j=0}^{+\infty} {\bf{c}}_{j}(f) \varphi_{j}$. Assuming the necessary compatibility condition 
\begin{equation}
\frac{1}{2\pi}\dashint_{X}f(x){\rm{d}}\mu_{\omega}(x) = \frac{\deg_{\omega}({\bf{L}})}{(n-1)!{\rm{Vol}}(X,\omega)},
\end{equation}
we have ${\bf{c}}_{0}(f) = \frac{2\pi\deg_{\omega}({\bf{L}})}{(n-1)!{\rm{Vol}}(X,\omega)}$. The partial sums $f_{n} = \sum_{j=0}^{n} {\bf{c}}_{j}(f) \varphi_{j}$ are smooth and satisfy the compatibility condition for each $n$. The corresponding conformal weights are given by 
\begin{equation}
\psi_{n} = \sum_{j=1}^{n} \frac{{\bf{c}}_{j}(f)}{\lambda_{j}({\bf{\Delta}}_{\omega})} \varphi_{j} \in C^{\infty}(X),
\end{equation}
such that ${\bf{h}}_{n} = \exp(2\psi_{n}){\bf{h}}_{0}$. Here we can choose the reference metric ${\bf{h}}_{0}$, such that 
\begin{equation}
\sqrt{-1}\Lambda_{\omega}(F({\bf{h}}_{0}) ) = \frac{2\pi\deg_{\omega}({\bf{L}})}{(n-1)!{\rm{Vol}}(X,\omega)}.
\end{equation}
By construction, we have 
\begin{equation}
{\rm{K}}({\bf{L}},{\bf{h}}_{n}):= \sqrt{-1}\Lambda_{\omega}(F({\bf{h}}_{0}) ) - {\bf{\Delta}}_{\omega}\psi_{n} = f_{n}
\end{equation}
To establish the convergence of the sequence $\{\psi_{n}\}$, we consider the $H^{2}$-norm on $X$ defined by 
\begin{equation}
\|u\|_{H^{2}}^{2} = \int_{X} u^{2} {\rm{d}}\mu_{\omega} + \int_{X} |\nabla^{2}u|_{\omega}^{2} {\rm{d}}\mu_{\omega}.
\end{equation}
Consider now the Bochner identity
\begin{equation}
\label{bochner_identity}
\frac{1}{2} {\bf{\Delta}}_{\omega} (|\nabla u|_{\omega}^2) = |\nabla^2 u|_{\omega}^2 + {\rm{Ric}}(\nabla u, \nabla u) + \langle \nabla ({\bf{\Delta}}_{\omega} u), \nabla u \rangle_{\omega}.
\end{equation}
From this, we have
\begin{equation}
0  = \frac{1}{2}  \int_{X}{\bf{\Delta}}_{\omega} (|\nabla u|_{\omega}^2){\rm{d}}\mu_{\omega} =  \int_{X}|\nabla^2 u|_{\omega}^2{\rm{d}}\mu_{\omega} + \int_{X}{\rm{Ric}}(\nabla u, \nabla u){\rm{d}}\mu_{\omega} + \int_{X}\langle \nabla ({\bf{\Delta}}_{\omega} u), \nabla u \rangle_{\omega}{\rm{d}}\mu_{\omega}
\end{equation}
Observing that 
\begin{equation}
\int_{X}\langle \nabla ({\bf{\Delta}}_{\omega} u), \nabla u \rangle_{\omega}{\rm{d}}\mu_{\omega} = - \int_{X} ({\bf{\Delta}}_{\omega} u)^{2} {\rm{d}}\mu_{\omega},
\end{equation}
it follows that 
\begin{equation}
0 = \int_{X}|\nabla^2 u|_{\omega}^2{\rm{d}}\mu_{\omega} + \int_{X}{\rm{Ric}}(\nabla u, \nabla u){\rm{d}}\mu_{\omega} - \int_{X} ({\bf{\Delta}}_{\omega} u)^{2} {\rm{d}}\mu_{\omega}.
\end{equation}
Since $X$ is a compact manifold, the Ricci curvature is bounded from below. Thus, there exists a constant $\kappa \in \mathbbm{R}$, such that ${\rm{Ric}}(v, v) \geq \kappa |v|^2$, $\forall v \in TX$. Combining this with the last equality, we obtain the following inequality
\begin{equation}
\label{hessianbound}
\int_{X} |\nabla^{2} u|_{\omega}^{2} {\rm{d}}\mu_{\omega} \leq \int_{X} ({\bf{\Delta}}_{\omega} u)^{2} {\rm{d}}\mu_{\omega} + |\kappa| \int_{X} |\nabla u|_{\omega}^{2} {\rm{d}}\mu_{\omega}.
\end{equation}

To control the first-order term, we utilize the self-adjointness of the Laplacian and the integration by parts formula
\begin{equation}
\int_{X} |\nabla u|_{\omega}^{2} {\rm{d}}\mu_{\omega} = - \int_{X} u ({\bf{\Delta}}_{\omega} u) {\rm{d}}\mu_{\omega} \leq \int_{X} |u| |{\bf{\Delta}}_{\omega} u| {\rm{d}}\mu_{\omega}.
\end{equation}

Applying Young's inequality with $\epsilon > 0$ (e.g. \cite{evans2010partial}) to the right-hand side yields

\begin{equation}
\int_{X} |u| |{\bf{\Delta}}_{\omega} u| {\rm{d}}\mu_{\omega} \leq \epsilon \int_{X} ({\bf{\Delta}}_{\omega} u)^{2} {\rm{d}}\mu_{\omega} + \frac{1}{4\epsilon} \int_{X} u^{2} {\rm{d}}\mu_{\omega}.
\end{equation}

By substituting this back into Eq. \eqref{hessianbound} and choosing a sufficiently small $\epsilon$ (for instance, $\epsilon = (2|\kappa|)^{-1}$), we find:

\begin{equation}
\int_{X} |\nabla^{2} u|_{\omega}^{2} {\rm{d}}\mu_{\omega} \leq \left(1 + |\kappa|\epsilon \right) \int_{X} ({\bf{\Delta}}_{\omega} u)^{2} {\rm{d}}\mu_{\omega} + \frac{|\kappa|}{4\epsilon} \int_{X} u^{2} {\rm{d}}\mu_{\omega}.
\end{equation}

Therefore, there exists a constant $C = \max\{1 + |\kappa|\epsilon, \frac{|\kappa|}{4\epsilon}\}$, such that

\begin{equation}
\|u\|_{H^{2}}^{2} = \int_{X} u^2{\rm{d}}\mu_{\omega} + \int_{X} |\nabla^2 u|^2{\rm{d}}\mu_{\omega} \leq (1+C) \left( \int_{X} u^{2} {\rm{d}}\mu_{\omega} + \int_{X} ({\bf{\Delta}}_{\omega} u)^{2} {\rm{d}}\mu_{\omega} \right).
\end{equation}
From above, we have 
\begin{equation}
\|\psi_{n} - \psi_{m}\|_{H^{2}}^{2} \leq (1 + C) \left( \|\psi_{n} - \psi_{m}\|_{L^{2}}^{2} + \|{\bf{\Delta}}_{\omega}(\psi_{n} - \psi_{m})\|_{L^{2}}^{2} \right).
\end{equation}
Applying the Parseval-Plancherel identity, this bound becomes
\begin{equation}
\|\psi_{n} - \psi_{m}\|_{H^{2}}^{2} \leq (1+C) \sum_{j=m+1}^{n} \left( \frac{1}{\lambda_{j}({\bf{\Delta}}_{\omega})^{2}} + 1 \right) {\bf{c}}_{j}(f)^{2}.
\end{equation}
Since $\lambda_{j} \to +\infty$ and $f \in L^{2}(X)$, the right-hand side vanishes as $n, m \to +\infty$, proving that $\{\psi_{n}\}$ is a Cauchy sequence in $H^{2}(X)$. By the completeness of the Sobolev space, there exists a limit $\psi_{\infty} \in H^{2}(X)$ such that $\psi_{n} \to \psi_{\infty}$ strongly in $H^{2}$.

Since the embedding $H^{2}(X) \hookrightarrow L^{2}(X)$ is continuous, we have $\psi_{n} \to \psi_{\infty}$ in $L^{2}$. By the Riesz-Fischer theorem, there exists a subsequence (which we still denote by $\psi_{n}$) that converges a.e. to $\psi_{\infty}$. Consequently, we obtain

\begin{equation}
{\bf{h}}_{n} = \exp(2\psi_{n}){\bf{h}}_{0} \xrightarrow{a.e.} \exp(2\psi_{\infty}){\bf{h}}_{0} := {\bf{h}}_{\infty}.
\end{equation}

The limit structure ${\bf{h}}_{\infty}$ defines a singular Hermitian structure on ${\bf{L}}$ in the sense of Demailly, given that $\psi_{\infty} \in H^{2}(X) \subset L^{1}_{{\rm{loc}}}(X)$, see Definition \ref{defsingline}. 

Finally, we verify the curvature identity. Since the operator ${\bf{\Delta}}_{\omega} \colon H^{2}(X) \to L^{2}(X)$ is a bounded linear operator, the convergence $\psi_{n} \to \psi_{\infty}$ in $H^{2}$ implies that ${\bf{\Delta}}_{\omega}\psi_{n} \to {\bf{\Delta}}_{\omega}\psi_{\infty}$ in $L^{2}$. Taking the limit in the sense of distributions in the identity
\begin{equation}
{\rm{K}}({\bf{L}},{\bf{h}}_{n}) = \sqrt{-1}\Lambda_{\omega}(F({\bf{h}}_{0})) - {\bf{\Delta}}_{\omega}\psi_{n} = f_{n},
\end{equation}
and noting that $f_{n} \xrightarrow{L^{2}} f$, we conclude that
\begin{equation}
{\rm{K}}({\bf{L}},{\bf{h}}_{\infty}) = f,
\end{equation}
as an identity in $L^{2}(X)$ and, therefore, in the sense of distributions. This concludes the proof.
\end{proof}

\begin{definition}
Let ${\bf{E}} \to X$ be a holomorphic vector bundle over a complex manifold $X$. A singular Hermitian metric ${\bf{h}}$ on ${\bf{E}}$ is a measurable map from the base space $X$ to the space of non-negative Hermitian forms on the fibers.
\end{definition}

If $({\bf{E}},{\bf{h}}) \to X$ is a holomorphic vector bundle given by the tensor product of a smooth Hermitian vector bundle with a singular Hermitian line bundle, namely, ${\bf{h}} = {\bf{h}}_{0} \otimes {\bf{h}}_{{\rm{sing}}}$, such that 
\begin{equation}
({\bf{E}},{\bf{h}}) = ({\bf{E}}_{0},{\bf{h}}_{0}) \otimes ({\bf{L}},{\bf{h}}_{{\rm{sing}}}), 
\end{equation}
then we have a well-defined notion of curvature current. This notion is established via extension by density. Since ${\bf{h}}_{{\rm{sing}}}$ can be approximated by a sequence of smooth Hermitian structures $\{{\bf{h}}_{n}\}$ on ${\bf{L}}$, the smooth tensor product formula holds for every $n$, i.e.,
\begin{equation}
F({\bf{h}}_{0} \otimes {\bf{h}}_{n}) = F({\bf{h}}_{0}) \otimes 1_{{\bf{L}}} + 1_{{\bf{E}}_{0}} \otimes F({\bf{h}}_{n}),
\end{equation}
see for instance \cite{Kobayashi+1987}. Because the tensor product splitting isolates the singular sequence entirely within the scalar line bundle component, taking the limit as $n \uparrow +\infty$ in the sense of currents is analytically well-behaved. The right-hand side converges distributionally to 
\begin{equation}
F({\bf{h}}) = F({\bf{h}}_{0}) \otimes 1_{{\bf{L}}} + 1_{{\bf{E}}_{0}} \otimes F({\bf{h}}_{{\rm{sing}}}).
\end{equation}
Given that this limit is uniquely determined and analytically stable, we formally set the curvature of the singular tensor product to be exactly this distributional limit. 

Notice that, from the line bundle case (Definition \ref{defsingline}), we also obtain a well-defined notion of mean curvature. Now we can prove the Theorem \ref{Hcase1}.
\begin{proof}[Proof of Theorem \ref{Hcase1}]
Since we have 
\begin{equation}
\displaystyle \frac{1}{{\rm{rank}}({\bf{E}})} \int_{C}c_{1}({\bf{E}}) \in \mathbbm{Z}, \ \ \ \ \ \forall [C] \in {\rm{NE}}(X_{P})_{\mathbbm{Z}},
\end{equation}
we can suppose from Theorem \ref{theoremA} that
\begin{equation}
{\bf{E}} = {\bf{E}}_{0} \otimes {\bf{L}}_{0},
\end{equation}
such that ${\bf{L}}_{0} \in {\rm{Pic}}(X_{P})$ and $c_{1}({\bf{E}}_{0}) = 0$.

By considering ${\bf{h}} = {\bf{h}}_{{\bf{E}}_{0}} \otimes {\bf{h}}_{{\bf{L}}_{0}}$, where ${\bf{h}}_{{\bf{E}}_{0}}$ is some smooth Hermitian structure on ${\bf{E}}_{0}$ and ${\bf{h}}_{{\bf{L}}_{0}}$ is some smooth Hermitian structure on ${\bf{L}}$, we obtain
\begin{equation}
F({\bf{h}}) = F({\bf{h}}_{{\bf{E}}_{0}}) \otimes 1_{{\bf{L}}_{0}} + 1_{{\bf{E}}_{0}} \otimes F({\bf{h}}_{{\bf{L}}_{0}}).
\end{equation}
see for instance \cite{Kobayashi+1987}. Hence, it follows that 
\begin{equation}
\begin{split}
 {\rm{K}}({\bf{E}},{\bf{h}}) &= {\rm{K}}({\bf{E}}_{0},{\bf{h}}_{{\bf{E}}_{0}}) \otimes 1_{{\bf{L}}_{0}} + 1_{{\bf{E}}_{0}} \otimes {\rm{K}}({\bf{L}}_{0},{\bf{h}}_{{\bf{L}}_{0}}) \\
 & = {\rm{K}}({\bf{E}}_{0},{\bf{h}}_{{\bf{E}}_{0}}) \otimes 1_{{\bf{L}}_{0}} + {\rm{K}}({\bf{L}}_{0},{\bf{h}}_{{\bf{L}}_{0}}) (1_{{\bf{E}}_{0}} \otimes 1_{{\bf{L}}_{0}}),
 \end{split}
\end{equation}
Since $1_{{\bf{E}}}  = 1_{{\bf{E}}_{0}} \otimes 1_{{\bf{L}}_{0}}$, $c_{1}({\bf{E}}_{0}) = 0$ and ${\bf{E}}_{0}$ is irreducible, by choosing ${\bf{h}}_{YM}$ which solves the Hermite-Einstein equation on ${\bf{E}}_{0}$ (see Theorem \ref{Ramanan}), we have the following
\begin{equation}
\label{meanHR}
{\rm{K}}({\bf{E}},{\bf{h}}) = \big (\nu_{[\omega_{0}]}({\bf{E}}_{0}) + {\rm{K}}({\bf{L}}_{0},{\bf{h}}_{{\bf{L}}_{0}}) \big ) 1_{{\bf{E}}} =  {\rm{K}}({\bf{L}}_{0},{\bf{h}}_{{\bf{L}}_{0}}) 1_{{\bf{E}}}.
\end{equation}
Observing that $\mu_{[\omega_{0}]}({\bf{E}}) = \mu_{[\omega_{0}]}({\bf{L}}_{0}) = \deg_{\omega_{0}}({\bf{L}}_{0})$, since 
\begin{equation}
\frac{1}{2\pi}\dashint_{X_{P}}f(x){\rm{d}}\mu_{\omega_{0}}(x) = \frac{\mu_{[\omega_{0}]}({\bf{E}})}{(n-1)!{\rm{Vol}}(X_{P},\omega_{0})} = \frac{\deg_{\omega_{0}}({\bf{L}}_{0})}{(n-1)!{\rm{Vol}}(X_{P},\omega_{0})}
\end{equation}
it follows from Lemma \ref{lemmafund} that there exists a sequence of smooth Hermitian structures $\{{\bf{q}}_{n}\}$ on ${\bf{L}}_{0}$, such that 
\begin{equation}
{\rm{K}}({\bf{L}}_{0},{\bf{q}}_{n}) \overset{L^2}{\longrightarrow}  f \ \ \ \ {\text{and}} \ \ \ \ {\bf{q}}_{n} \overset{a.e.}{\longrightarrow} {\bf{q}}_{\infty},
\end{equation}
as $n \uparrow +\infty$, where ${\bf{q}}_{\infty}$ is a singular Hermitian structure on ${\bf{L}}_{0}$. Moreover, we have
\begin{equation}
{\rm{K}}({\bf{L}},{\bf{q}}_{\infty}) = f,
\end{equation}
in the sense of distributions. From above, we set
\begin{equation}
{\bf{h}}_{n} := {\bf{h}}_{YM} \otimes {\bf{q}}_{n}
\end{equation}
for every $n = 1,2,\ldots$. By construction, we have ${\bf{h}}_{n}  \overset{a.e.}{\longrightarrow}  {\bf{h}}_{YM} \otimes {\bf{q}}_{\infty} := {\bf{h}}_{\infty}$, and 
\begin{equation}
\underbrace{{\rm{K}}({\bf{E}},{\bf{h}}_{n}) = {\rm{K}}({\bf{L}}_{0},{\bf{q}}_{n})1_{{\bf{E}}}}_{\text{Eq. (\ref{meanHR})}} \overset{L^2}{\longrightarrow}  f1_{{\bf{E}}}, 
\end{equation}
as $n \uparrow +\infty$. In fact, notice that 
\begin{equation}
||{\rm{K}}({\bf{E}},{\bf{h}}_{n}) - f1_{{\bf{E}}}||_{L^{2}}^{2} = \int_{X_{P}}{\rm{tr}} \big ( ({\rm{K}}({\bf{E}},{\bf{h}}_{n}) - f1_{{\bf{E}}})^{2}\big){\rm{d}}\mu_{\omega_{0}} = \rank{({\bf{E}})} ||{\rm{K}}({\bf{L}}_{0},{\bf{q}}_{n}) - f||_{L^{2}}^{2},
\end{equation}
thus $\lim_{n \to +\infty}||{\rm{K}}({\bf{E}},{\bf{h}}_{n}) - f1_{{\bf{E}}}||_{L^{2}} = 0$. Moreover, it follows from Lemma \ref{lemmafund} that 
\begin{equation}
{\rm{K}}({\bf{E}},{\bf{h}}_{\infty}) = {\rm{K}}({\bf{L}}_{0},{\bf{q}}_{\infty})1_{{\bf{E}}} = f1_{{\bf{E}}},
\end{equation}
in the sense of distributions, which concludes the proof.
\end{proof}

\begin{example}[{\bf{Symmetric Square of Spinor Bundle}}]
\label{examplesingular}
Consider the symmetric square of the spinor bundle  $S^{2}({\bf{\Sigma}}) \to Q_{5}$ as in Example \ref{Squarespinorbundle}. Since
\begin{equation}
S^{2}({\bf{\Sigma}}) = {\bf{E}}_{0} \otimes {\bf{L}}_{0}
\end{equation}
with ${\bf{L}}_{0} \in {\rm{Pic}}(Q_{5})$ and $c_{1}({\bf{E}}_{0}) = 0$, by Theorem B, given any $f \in L^{2}(Q_{5}, \mu_{\omega_{0}})$ satisfying the compatibility condition:
\begin{equation}
\frac{1}{2\pi}\dashint_{Q_{5}}f(x){\rm{d}}\mu_{\omega_{0}}(x) = \frac{\mu_{[\omega_{0}]}(S^{2}({\bf{\Sigma}}))}{4!{\rm{Vol}}(Q_{5},\omega_{0})},
\end{equation}
where $\omega_{0}$ is some ${\rm{Spin}}(7,\mathbb{C})$-invariant K\"{a}hler metric on $Q_{5}$, there exists a singular Hermitian structure ${\bf{h}}_{\infty}$ on $S^{2}({\bf{\Sigma}})$ such that its mean curvature current satisfies
\begin{equation}
{\rm{K}}(S^{2}({\bf{\Sigma}}), {\bf{h}}_{\infty}) = f1_{S^{2}({\bf{\Sigma}})}.
\end{equation}
Let us construct an explicit example of $f \in L^{2}(Q_{5}, \mu_{\omega_{0}})$ satisfying the desired condition. Fix a point $p_{0} \in Q_{5}$ and consider the Riemannian distance function $d(\cdot,p_{0}) \colon Q_{5} \to \mathbbm{R}$. From this, by taking $s >0$, we set
\begin{center}
$\displaystyle g(x) := \frac{1}{d(x,p_{0})^{s}}$, \ \ \ $x \in Q_{5}$.
\end{center}
The function $g$ is certainly not continuous at $p_{0} \in Q_{5}$. Choosing normal coordinates around $p_{0}$, since $|g|^{2} = r^{-2s}$ within a small geodesic ball $B_{\epsilon}(p_0)$, we have 
\begin{align}
    \int_{B_\epsilon(p_0)}|g|^{2}{\rm{d}}\mu_{\omega_{0}} = \int_{B_\epsilon(p_0)} \frac{1}{r^{2s}} {\rm{d}}\mu_{\omega_{0}} &= \int_{S^{9}} \int_{0}^\epsilon \frac{1}{r^{2s}} (1 + \mathcal{O}(r^{2})) r^{9} {\rm{d}}r {\rm{d}}\phi\\
    &= {\rm{Vol}}(S^{9}) \int_{0}^\epsilon ( r^{9 - 2s} + \mathcal{O}(r^{11 - 2s})) {\rm{d}}r.
\end{align}

Since the perturbation introduced by the curvature, $\mathcal{O}(r^{11- 2s})$, has a strictly higher power of $r$ (i.e., it converges even more readily at $r=0$), the integrability is dictated exclusively by the principal term
\begin{equation}
    \int_0^\epsilon r^{9 - 2s}{\rm{d}}r,
\end{equation}
Therefore, observing that 
\begin{center}
$9 - 2s > -1 \iff s < 5$,
\end{center}
we conclude that $g \in L^{2}(Q_{5},{\rm{d}}\mu_{\omega_{0}})$ for every $s < 5$. Denoting 
\begin{equation}
c_{0}(S^{2}({\bf{\Sigma}})): = 2 \pi \frac{\mu_{[\omega_{0}]}(S^{2}({\bf{\Sigma}}))}{4!{\rm{Vol}}(Q_{5},\omega_{0})} - \dashint_{Q_5}\frac{1}{d(y,p_{0})^{s}}{\rm{d}}\mu_{\omega_{0}}(y),
\end{equation}
we set
\begin{equation}
f(x) := \frac{1}{d(x,p_{0})^{s}} + c_{0}(S^{2}({\bf{\Sigma}})), \ \ \ x \in Q_{5},
\end{equation}
for some $\alpha < 5$. By construction, we have $f \in L^{2}(Q_{5},{\rm{d}}\mu_{\omega_{0}})$ and 
\begin{equation}
\frac{1}{2\pi}\dashint_{Q_5}f(x){\rm{d}}\mu_{\omega_{0}}(x) = \frac{\mu_{[\omega_{0}]}(S^{2}({\bf{\Sigma}}))}{4!{\rm{Vol}}(Q_{5},\omega_{0})}.
\end{equation}
Hence, from Theorem \ref{Hcase1}, we have a singular Hermitian metric ${\bf{h}}_{\infty}$ on $S^{2}({\bf{\Sigma}}) \to Q_{5}$, such that 
\begin{equation}
{\rm{K}}(S^{2}({\bf{\Sigma}}),{\bf{h}}_{\infty}) = \Bigg (\frac{1}{d(x,p_{0})^{s}} + c_{0}(S^{2}({\bf{\Sigma}}))\Bigg )1_{S^{2}({\bf{\Sigma}})}.
\end{equation}
\end{example}

\begin{remark}
Given a homogeneous holomorphic vector bundle ${\bf{E}} \to X_{P}$, defined by some irreducible $P$-module, fixed some $G$-invariant (integral) K\"{a}hler metric $\omega_{0}$ on $X_{P}$, such that 
\begin{equation}
\displaystyle \frac{1}{{\rm{rank}}({\bf{E}})} \int_{C}c_{1}({\bf{E}}) \in \mathbbm{Z}, \ \ \ \ \ \forall [C] \in {\rm{NE}}(X_{P})_{\mathbbm{Z}},
\end{equation}
for every collection of points $p_{1},\ldots,p_{m} \in X_{P}$, from a similar argument as in Example \ref{examplesingular}, we have a singular Hermitian metric ${\bf{h}}_{\infty}$ on ${\bf{E}}$ satisfying
\begin{equation}
{\rm{K}}({\bf{E}},{\bf{h}}_{\infty}) = \Bigg (\sum_{j = 1}^{m}\frac{1}{d(x,p_{j})^{s_{j}}} + C_{0} \Bigg )1_{{\bf{E}}}
\end{equation}
where $0<s_{j} < \dim_{\mathbbm{C}}(X_{P})$, $\forall j= 1,\ldots,m$, and $C_{0}$ is a topological constant depending on ${\bf{E}}$.
\end{remark}

In order to obtain Corollary \ref{corollarysing}, we consider the following technical lemma.

\begin{lemma}
\label{intdistance}
Let $(M, g)$ be a compact Riemannian manifold of real dimension $n$, and let $N \subset M$ be a smooth closed submanifold of real codimension $k$ (so that $\dim_{\mathbbm{R}} N = n - k$). For any real exponent $s$, we have 
\begin{equation}
f(x) := \frac{1}{d(x,N)^{s}} \in L^{2}(M) \ \ \text{if and only if} \ \ s < \frac{k}{2}.
\end{equation}
\end{lemma}

\begin{proof}
By definition, $f \in L^{2}(M)$ if and only if 
\begin{center}
$ \displaystyle \int_{M} \frac{1}{d(x,N)^{2s}} {\rm{d}}\mu_{g}(x) < \infty.$
\end{center}
Since $M$ is compact and $d(x,N)$ is continuous and strictly positive on $M \backslash N$, the integrand is bounded on any compact subset disjoint from $N$. Therefore, the integrability depends entirely on the behavior of the function in an arbitrarily small neighborhood of $N$.

By the tubular neighborhood theorem, there exists a sufficiently small $\epsilon > 0$ such that the tubular neighborhood \cite{hirsch2012differential} defined by $T_{\epsilon} := \{x \in M \mid d(x,N) < \epsilon\}$ is diffeomorphic to an $\epsilon$-ball bundle inside the normal bundle $\mathcal{N}(N)$. 
\begin{figure}[H]
\includegraphics[scale = .22]{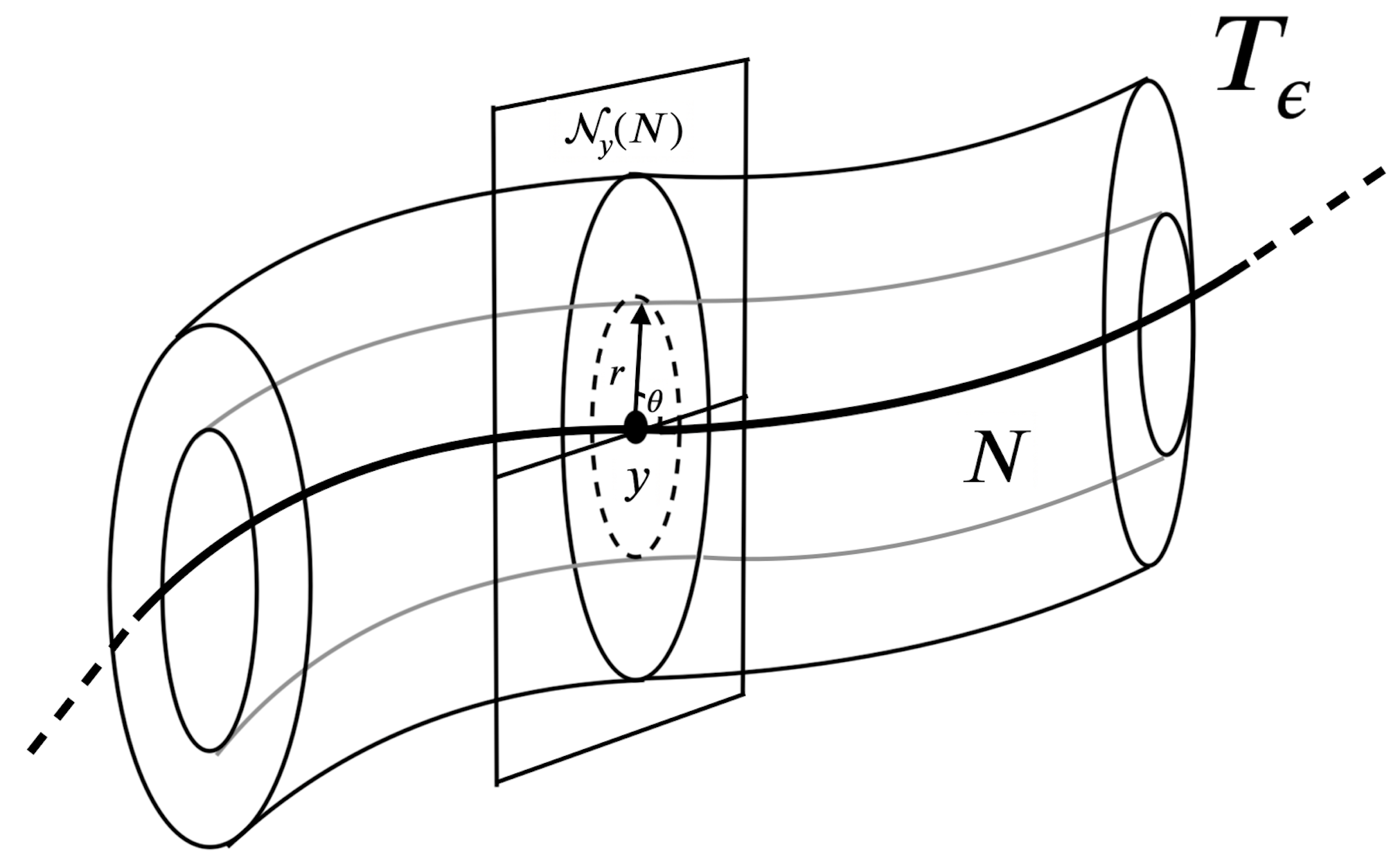}
\caption{Fermi coordinates $(y, r, \theta)$ on a tubular neighborhood $N \subset T_{\epsilon} \subset M$.}
\label{Fermi}
\end{figure}
We split the global integral into two disjoint regions:
\begin{center}
$\displaystyle \int_{M} \frac{1}{d(x,N)^{2s}} {\rm{d}}\mu_{g}(x)  = \int_{M - T_{\epsilon}} \frac{1}{d(x,N)^{2s}} {\rm{d}}\mu_{g}(x)  + \int_{T_{\epsilon}} \frac{1}{d(x,N)^{2s}} {\rm{d}}\mu_{g}(x).$
\end{center}
For the exterior region $M - T_{\epsilon}$, we have $d(x,N) \geq \epsilon > 0$. Since $M$ is compact, its total volume is finite ($\mu_{g}(M)  < \infty$), which guarantees the convergence of the first term
\begin{center}
$\displaystyle \int_{M - T_{\epsilon}} \frac{1}{d(x,N)^{2s}} {\rm{d}}\mu_{g}(x)  \leq \frac{\mu_{g} (M - T_{\epsilon})}{\epsilon^{2s}} < \infty.$
\end{center}
Thus, global $L^2$ integrability reduces strictly to the convergence of the integral over the tube $T_{\epsilon}$.

Following \cite[\S 3.3]{gray2003tubes}, inside $T_{\epsilon}$, we introduce adapted geodesic polar coordinates $(y, r, \theta)$, where $y \in N$, $r = d(x,N) \in [0, \epsilon)$ is the radial geodesic distance to $N$, and $\theta \in S^{k-1}$ parameterizes the angular directions in the normal fibers (Figure \ref{Fermi}). In these coordinates, the Riemannian volume form decomposes as
\begin{center}
${\rm{d}}\mu_{g} = \Theta(y, r, \theta) \cdot r^{k-1} \, dr \wedge {\rm{dvol}}_{S^{k-1}} \wedge {\rm{dvol}}_{N},$
\end{center}
where ${\rm{dvol}}_{N}$ and ${\rm{dvol}}_{S^{k-1}}$ are the canonical volume forms of $N$ and the unit sphere $S^{k-1}$, respectively, and $\Theta(y, r, \theta)$ is the smooth density factor (the Jacobian of the exponential map). 

Applying Fubini's theorem, the integral over $T$ becomes
\begin{center}
$\displaystyle \int_{T_{\epsilon}} \frac{1}{d(x,N)^{2s}} {\rm{d}}\mu_{g}(x) = \int_{N} \int_{S^{k-1}} \int_{0}^{\epsilon} \frac{1}{r^{2s}} \, \Theta(y, r, \theta) \cdot r^{k-1}{\rm{d}}r {\rm{dvol}}_{S^{k-1}} {\rm{dvol}}_{N}.$
\end{center}
Since $\Theta(y, r, \theta)$ is smooth on the compact domain $\overline{T_{\epsilon}}$ and $\lim_{r \to 0} \Theta = 1$, the extreme value theorem ensures the existence of uniform global bounds $C_{1}, C_{2} > 0$ such that $C_1 \leq \Theta(y, r, \theta) \le C_2$. This allows us to bound the integral from above and below
\begin{center}
$\displaystyle C_{1} \cdot K \int_{0}^{\epsilon} r^{k - 2s - 1} {\rm{d}}r \leq \int_{T_{\epsilon}} \frac{1}{d(x,N)^{2s}} {\rm{d}}\mu_{g}(x) \le C_2 \cdot K \int_{0}^{\epsilon} r^{k - 2s - 1}{\rm{d}}r,$
\end{center}
where $K = \operatorname{Vol}(N) \cdot \operatorname{Vol}(S^{k-1})$ is a finite positive constant because both $N$ and $S^{k-1}$ are compact manifolds.

The remaining radial integral $\int_{0}^{\epsilon} r^{p} \, dr$ at the origin converges if and only if the exponent satisfies $p > -1$. Applying this condition to our specific exponent yields
\begin{center}
$ \displaystyle k - 2s - 1 > -1 \iff k - 2s > 0 \iff 2s < k \iff s < \frac{k}{2}$
\end{center}
Hence, the integral converges if and only if $s < \frac{k}{2}$, which completes the proof.
\end{proof}

From the above lemma and Theorem \ref{Hcase1} we can prove Corollary \ref{corollarysing}.

\begin{proof}[Proof of Corollary \ref{corollarysing}] Given a analytic smooth subvariety $Y \subset X_{P}$, since $Y$ is a closed smooth complex submanifold of $X_{P}$, we have from Lemma \ref{intdistance} that 
\begin{center}
$g := d(\cdot,Y)^{-s} \in L^{2}(X_{P},\mu_{\omega_{0}}),$ \ \ \ \ $\forall 0  <s < {\rm{codim}}(Y)$, 
\end{center}
for some invariant integral K\"{a}hler metric $\omega_{0}$ on $X_{P}$, where ${\rm{codim}}(Y) = \dim_{\mathbbm{C}}(X_{P}) - \dim_{\mathbbm{C}}(Y)$. 

Given a homogeneous holomorphic vector bundle ${\bf{E}} \rightarrow X_{P}$, defined by some irreducible $P$-module, satisfying
\begin{equation}
    \frac{1}{\rm{rank}({\bf{E}})}\int_{C}c_{1}({\bf{E}})\in\mathbb{Z}, \quad \forall[C]\in {\rm{NE}}(X_{P})_{\mathbbm{Z}},
\end{equation}
we set
\begin{equation}
 C_{0} := 2 \pi \frac{\mu_{[\omega_{0}]}({\bf{E}})}{(n-1)!{\rm{Vol}}(X_{P},\omega_{0})} - \dashint_{X_{P}}\frac{1}{d(u,Y)^{s}}{\rm{d}}\mu_{\omega_{0}}(u).
\end{equation}
Given $0  <s < {\rm{codim}}(Y)$, we have 
\begin{equation}
f(x) = \frac{1}{d(x,Y)^{s}} + C_{0} \in L^{2}(X_{P},\mu_{\omega_{0}}),
\end{equation}
satisfying
\begin{equation}
\frac{1}{2\pi}\dashint_{X_{P}}f {\rm{d}}\mu_{\omega_{0}} = \frac{\mu_{[\omega_{0}]}({\bf{E}})}{(n-1)!{\rm{Vol}}(X_{P},\omega_{0})}.
\end{equation}
Therefore, from Theorem \ref{Hcase1}, there exists a singular Hermitian structure ${\bf{h}}_{\infty}$ on ${\bf{E}}$ such that its mean curvature current satisfies
\begin{equation}
    {\rm{K}}({\bf{E}},{\bf{h}}_{\infty}) = \Bigg (\frac{1}{d(x,Y)^{s}} + C_{0} \Bigg) 1_{{\bf{E}}},
\end{equation}
in the sense of distributions, which concludes the proof.
\end{proof}

\begin{example}
Consider the flag variety 
\begin{equation}
X_{P} = {\rm{Spin}}(8,\mathbbm{C})/P_{I}
\end{equation}
as in the Example \ref{examplespin}. Given a homogeneous holomorphic vector bundle 
\begin{center}
${\bf{E}}_{\theta} = {\rm{Spin}}(8,\mathbbm{C}) \times_{P}W(\lambda) \rightarrow X_{P}$ 
\end{center}
defined by an irreducible representation of $\theta \colon P \to {\rm{GL}}(W(\lambda))$, with highest weight $\lambda = \lambda_{s} + \lambda_{c}$, such that $\lambda_{s} \in \Lambda_{\mathfrak{s}_{P}}^{+}$ and $\lambda_{c} \in \Lambda_{P}$, let us suppose that 
\begin{equation}
\frac{\langle \lambda_{s},\alpha_{1}^{\vee} \rangle + 2\langle \lambda_{s},\alpha_{2}^{\vee} \rangle}{3} \in \mathbb{Z}. 
\end{equation}
In this situation, as we have seen in Example \ref{examplespin}, it follows that 
\begin{equation}
\displaystyle \frac{1}{{\rm{rank}}({\bf{E}}_{\theta})} \int_{C}c_{1}({\bf{E}}_{\theta}) \in \mathbbm{Z}, \ \ \ \ \ \forall [C] \in {\rm{NE}}(X_{P})_{\mathbbm{Z}}.
\end{equation}
Consider the distinguished smooth rational curves 
\begin{equation}
\mathbbm{P}_{\alpha_{j}} = \overline{\exp(\mathfrak{g}_{-\alpha_{j}}){\rm{o}}} \subset X_{P}, \ \ j = 3,4.
\end{equation}
By taking 
\begin{center}
$0 < a,b < {\rm{codim}}(\mathbbm{P}_{\alpha_{j}}) = \dim_{\mathbbm{C}}(X_{P}) - 1 = 9-1 = 8$, 
\end{center}
from Corollary \ref{corollarysing} we have a singular Hermitian structure ${\bf{h}}_{\infty}$ on ${\bf{E}}_{\theta}$ such that its mean curvature current satisfies
\begin{equation}
    {\rm{K}}({\bf{E}}_{\theta},{\bf{h}}_{\infty}) = \Bigg (\frac{1}{d(x,\mathbbm{P}_{\alpha_{2}})^{a}} + \frac{1}{d(x,\mathbbm{P}_{\alpha_{3}})^{b}}+ C_{0} \Bigg) 1_{{\bf{E}}_{\theta}},
\end{equation}
in the sense of distributions.
\end{example}

\appendix

\subsection*{Data Availability} Data sharing not applicable to this article as no datasets were generated or analysed during the current study.

\subsection*{Conflict of interest statement} The author declares that there is no conflict of interest.

\bibliographystyle{alpha}
\bibliography{biblio}

\end{document}